\documentclass[reqno]{amsart}
\usepackage{amsfonts}
\usepackage{amsmath,amssymb,amsthm,amsfonts}
\usepackage{bbm}
\usepackage{hyperref}

\usepackage{color}

\newtheorem{lemma}{Lemma}[section]
\newtheorem{theorem}{Theorem}[section]
\newtheorem{definition}{Definition}[section]
\newtheorem{proposition}{Proposition}[section]
\newtheorem{remark}{Remark}[section]

\numberwithin{equation}{section}

\newcommand{\R}{\mathbb{R}}

\renewcommand{\S}{\mathbb{S}}

\newcommand{\FP}{\mathbf{P}}

\newcommand{\FI}{\mathbf{I}}
\newcommand{\FX}{\mathbf{X}}
\newcommand{\FY}{\mathbf{Y}}

\newcommand{\CA}{\mathcal{A}}

\newcommand{\CE}{\mathcal{E}}

\newcommand{\CI}{\mathcal{I}}
\newcommand{\CJ}{\mathcal{J}}
\newcommand{\CK}{\mathcal{K}}
\newcommand{\CL}{\mathcal{L}}

\newcommand{\na}{\nabla}

\newcommand{\al}{\alpha}

\newcommand{\ga}{\gamma}
\newcommand{\om}{\omega}
\newcommand{\Om}{\Omega}
\newcommand{\la}{\lambda}
\newcommand{\de}{\delta}
\newcommand{\si}{\sigma}
\newcommand{\pa}{\partial}
\newcommand{\ka}{\kappa}
\newcommand{\eps}{\epsilon}

\newcommand{\ta}{\theta}
\newcommand{\vth}{\vartheta}
\newcommand{\vho}{\varrho}
\newcommand{\vps}{\varepsilon}

\newcommand{\Ga}{\Gamma}

\newcommand{\eqdef}{\overset{\mbox{\tiny{def}}}{=}}

\begin{document}
\title[boundary value problem for Boltzmann equation with soft potential]{The initial boundary value problem for the Boltzmann equation with soft potential}

\author[S.-Q. Liu]{Shuangqian Liu}
\address[SQL]{Department of Mathematics, Jinan University, Guangzhou 510632, P.R.~China}
\email{tsqliu@jnu.edu.cn}

\author[X.-F. Yang]{Xiongfeng Yang}
\address[XFY]{School of Mathematical Science, MOE-LSC and SHL-MAC, Shanghai Jiao Tong University,
 Shanghai, 200240, P.R. China}
\email{xf-yang@sjtu.edu.cn}


\begin{abstract}
Boundary effects are central to the dynamics of the dilute particles governed by Boltzmann equation. In this paper, we study
both the diffuse reflection and the specular reflection boundary value problems for Boltzmann equation with soft potential, in which the collision kernel is ruled by the inverse power law. For the diffuse reflection boundary condition, based on an $L^2$ argument and its interplay with intricate $L^\infty$ analysis for the linearized Boltzmann equation, we first establish the global existence and then obtain the exponential decay in $L^\infty$ space for the nonlinear Boltzmann equation in general classes of bounded domain. It turns out that the zero lower bound of the collision frequency and the singularity of the collision kernel lead to some new difficulties for achieving the {\it a priori} $L^\infty$ estimates and time decay rates of the solution. In the course of the proof, we capture some new properties of the probability integrals along the stochastic cycles and improve the $L^2-L^\infty$ theory to give a more direct approach to overcome those difficulties.  As to the specular reflection condition, our key contribution is to develop a new time-velocity weighted $L^\infty$ theory so that we could deal with the greater difficulties stemmed from the complicated velocity relations among the specular cycles and the zero lower bound of the collision frequency. From this new point, we are also able to prove the solutions of the linearized Boltzmann equation tend to equilibrium exponentially in $L^\infty$ space with the aid of the $L^2$ theory and a bootstrap argument. These methods in the latter case can be applied to the Boltzmann equation with soft potential for all other types of boundary condition.
\end{abstract}
\maketitle

\thispagestyle{empty}
\tableofcontents

\section{Introduction}\label{In}

\subsection{The problem and background}
Boundary effects should been taken into account when we study the dynamics of rarefied gas governed by the Boltzmann equation in a bounded domain.
There are several standard classes of boundary conditions for Boltzmann equation, cf. \cite[pp.716]{Guo-2010}.
In this paper, we consider the the Boltzmann equation
\begin{equation}\label{be}
\pa_tF+v\cdot\na_xF=Q(F,F),\ \ (x,v)\in\Omega\times\R^3,\ t>0,
\end{equation}
with initial data
\begin{equation}\label{id}
F(0,x,v)=F_0(x,v),\ \ (x,v)\in\Omega\times\R^3,
\end{equation}
and either of the following boundary conditions:
\begin{itemize}
\item
The diffuse reflection boundary condition
\begin{equation}\label{dbd}
F(t,x,v)|_{n(x)\cdot v<0}=\mu(v)\int_{n(x)\cdot v'>0}F(t,x,v')(n(x)\cdot v')dv',\ x\in\pa\Omega,\ t\geq0;
\end{equation}
\item The specular reflection boundary condition
\begin{equation}\label{sbd}
F(t,x,v)|_{n(x)\cdot v<0}=F(t,x,R_xv),\ R_xv=v-2(v\cdot n(x))n(x),\ x\in\pa\Omega,\ t\geq0.
\end{equation}
\end{itemize}
Here, $F(t,x,v)\geq0$ denotes the density distribution function of the gas particles at time $t\geq0$, position $x\in\Omega$, and velocity $v\in\R^3$,
$\Omega$ is a bounded domain in $\R^3$, $n(x)$ is the outward pointing unit norm vector at boundary $x\in\pa\Omega$ and $\mu(v)$ stands for the global Maxwellian which is normalized as
\begin{equation*}
\mu(v)=\frac{1}{2\pi}e^{-\frac{|v|^2}{2}},
\end{equation*}
so that
\begin{equation}\label{mu.n}
\int_{n(x)\cdot v>0}\mu(v)(n(x)\cdot v)dv=1.
\end{equation}
Let $(u,v)$ and $(u', v')$ be the velocities of the particles before and after the collision, which satisfy
\begin{eqnarray}\label{v.re}
\left\{\begin{array}{lll}
&v'=v+[(u-v)\cdot\omega]\om,\ \ u'=u-[(u-v)\cdot\omega]\om,\\[2mm]
&|u|^2+|v|^2=|u'|^2+|v'|^2.\end{array}\right.
\end{eqnarray}
The Boltzmann collision operator $Q(\cdot,\cdot)$ is given as the following non-symmetric form
\begin{equation*}
\begin{split}
Q(F_1,F_2)=&\int_{\R^3\times\S^2}|u-v|^{\varrho}b_0(\ta)[F_1(u')F_2(v')-F_1(u)F_2(v)]dud\omega\\
=&Q^{\text{gain}}(F_1,F_2)-Q^{\text{loss}}(F_1,F_2),
\end{split}
\end{equation*}
where the exponent is $\varrho=1-\frac{4}{s}$ with inverse power $1<s<4$ and $\cos\ta=\omega\cdot\frac{u-v}{|u-v|}$. Through the paper, we assume
\begin{equation}\label{soft}
-3<\varrho<0,\ 0<b_0(\ta)\leq C\cos\ta,
\end{equation}
which is so-called soft potentials with Grad's angular cutoff. Traditionally, one calls the hard potentials case when $\varrho\in(0, 1]$,
the Maxwellian molecules case when $\varrho= 0$, as well as the soft potentials case
when $\varrho\in(-3, 0)$.

The boundary condition \eqref{dbd} says the incoming particles are a probability average of the outgoing particles, while the boundary condition \eqref{sbd} reveals that the gas particles elastically collide against
the wall like billiard balls.

The boundary effects in kinetic equations are fundamental to the dynamics of gas, for instance, the phenomena of slip boundary layer, thermal creep, curvature effects, and singularity of propagation due to the boundary \cite{So} can be understood only with the knowledge of the interaction mechanism of the particles with the boundary. Owing to the importance of the boundary effects, there have been many achievements in the mathematical study of different aspect of the Boltzmann boundary value problems, see \cite{Cer-92, CIP, CC, DL-VPB, ELM1, ELM2, Kim, LY-im, MS-03, Mis} and references therein. In what follows, we mention some works related to the current study of the paper.
Hamdache \cite{Ha-92}
constructed the global renormalization solution to the Boltzmann equation in the case of hard potential with isothermal Maxwell boundary condition which in fact extends the pioneering work \cite{DL} for the Cauchy problem to the initial boundary value problem. Later on, Arkeryd-Cercinagani \cite{AC-93} generalized the results in \cite{Ha-92} to more extensive situations including the case when the
boundaries are not isothermal and the velocity is bounded. Arkeryd-Maslova \cite{AM-94} then removed the restriction to the bounded velocity introduced in \cite{AC-93} to study the similar issue for the Boltzmann equation and the BGK model. Except for the topic concerning the existence of the weak solution to the Botlzmann equation with initial boundary value problem mentioned above, another interesting problem in the Boltzmann study is to prove existence and uniqueness of the solution, as well as their time decay toward an absolute Maxwllian, at the appearance of compatible physical boundary conditions in a general domain, cf. \cite{Gr-58,GrII}. Compared with the study for the Cauchy problem in the whole space, to our best knowledgement, there are much less rigorous mathematical results of uniqueness, regularity or time-decay for the Boltzmann solutions toward a Maxwellian in bounded domain.
Although it was announced in \cite{SA-77} that the solutions to the Boltzmann equation near a Maxwellian would tend exponentially to the same equilibrium in a smooth bounded convex domain with specular reflection boundary condition, there is no complete rigorous proof.
Ukai \cite{Ukai-86} made a rough outline for proving the existence and time convergence to a global Maxwellian for the initial boundary value problem with hard potential. Golse-Perthame-Sulem \cite{GPS} investigated the boundary layer of stationary Boltzmann equation in one spatial dimension with specular reflection boundary condition in the case of hard spheres model $(\varrho=1)$. Liu-Yu \cite{LY-IBP-07,LY-DCDS}
studied the stationary boundary layers and the propagating
fluid waves of the initial boundary value problem for the Botlzmann equation in half space by means of the Green's function introduced in \cite{LY-GF-04}. Based on an elementary energy method, Yang-Zhao \cite{YZ} proved the stability of the rarefaction waves for the one dimensional Botlzmann equation in half space with specular reflection boundary condition. Under the assumption that a priori strong Sobolev estimates can be
verified, Desvillettes and Villani \cite{Des-90, DV-05, V2} recently established an almost exponential decay rate for
Boltzmann solutions with large amplitude for general collision kernels and general
boundary conditions.
It should be pointed out that many of the natural
physical boundary conditions create singularities in general domains \cite{Kim-11}, for which the Sobolev estimates break down in the crucial
elliptic estimates for the macroscopic part \cite{G-IUMJ,G06}. 
A new $L^2-L^\infty$ theory was developed in \cite{Guo-2010} to obtain the global existence and the exponential decay rates of the
 solution around a global Maxwellian in the case of hard potentials for four basic types of boundary conditions: in flow, bounce back reflection, specular reflection and diffuse refection, we refer to \cite{EGKM-13, BG-2015, EGM} for the latest advancement on this topic. Different $L^{2}-L^{\infty }$ methods have also been used in
\cite{EGM, UY-AA}.  Thanks to the work of \cite{Guo-2010}, the regularity \cite{GKTT-12, GKTT-14} and hydrodynamic limits \cite{EGKM-15} for the Botlzmann equation in general classes of bounded domain were further pondered. All of those works are focused on the case of the hard potential.
A natural challenge is to extend $L^2-L^\infty$ analysis  developed in \cite{Guo-2010} to the case of soft potential. This is the
goal of the present paper. Namely, we will
investigate the the global existence and the large time behaviors of the initial boundary value problem of
\eqref{be}, \eqref{id}, \eqref{dbd} or \eqref{sbd} with the condition \eqref{soft}.

\subsection{Domain, Characteristics and Perturbation}
Throughout this paper, $\Omega $ is a connected and bounded domain in $%
\R^3$ and defined by the open set $\{x~|~\xi(x)<0\}$ with $\xi(x)$ being a smooth function. Let $\na\xi(x)\neq0$ at boundary $\xi(x)=0$. The outward pointing unit normal vector at every point $x\in\pa\Omega$ is given by
$$
n(x)=\frac{\na\xi(x)}{|\na\xi(x)|}.
$$
We say $\Om$ is strictly convex if there exists $c_\xi>0$, for any $\zeta=(\zeta^1,\zeta^2,\zeta^3) \in \R^3$, it satiafies
\begin{equation}\label{scon}
\pa_{ij}\xi(x)\zeta^i\zeta^j\geq c_\xi|\zeta|^2.
\end{equation}
We say that $\Omega $ has a rotational symmetry, if there are vectors $x_{0}$
and $\varpi ,$ such that for all $x\in \partial \Omega $
\begin{equation}
\{(x-x_{0})\times \varpi \}\cdot n(x)\equiv 0.  \label{axis}
\end{equation}

For convenience, the phase boundary in the phase space $\Omega \times \R
^{3} $ is denoted by $\gamma =\partial \Omega \times \R^{3}$, and we further split it into
the following three kinds:
\begin{eqnarray*}
\textrm{outgoing boundary}: \gamma _{+} &=&\{(x,v)\in \partial \Omega \times \R^{3}\ :\
n(x)\cdot v> 0\}, \\
\textrm{incoming boundary}: \gamma _{-} &=&\{(x,v)\in \partial \Omega \times \R^{3}\ :\
n(x)\cdot v< 0\}, \\
\textrm{grazing boundary}: \gamma _{0} &=&\{(x,v)\in \partial \Omega \times \R^{3}\ :\
n(x)\cdot v=0\}.
\end{eqnarray*}%
As it is shown \cite[pp.715]{Guo-2010}, the \textit{backward exit time} which plays a crucial role in the study of boundary value problem of the Botlzmann equation can be well-defined via the backward characteristic trajectory.
Given $(t,x,v)$, we let $[X(s),V(s)]$ satisfy
\begin{equation}\label{ch.eq}
\frac{dX(s)}{ds}=V(s),\ \ \frac{dV(s)}{ds}=0,
\end{equation}
with the initial data $[X(t;t,x,v),V(t;t,x,v)]=[x,v]$. Then
$$[X(s;t,x,v),V(s;t,x,v)]=[x-(t-s)v,v]=[X(s),V(s)],$$
 which is called as the backward characteristic trajectory for the Boltzmann equation
\eqref{be}.

For $%
(x,v)\in \Om\times \R^{3}$, the \textit{backward exit time} $t_{\mathbf{b}}(x,v)>0$ is defined to be the first moment at which the backward characteristic line
$[X(s;0,x,v),V(s;0,x,v)]$ emerges from $\pa\Om$:
\begin{equation*}
t_{\mathbf{b}}(x,v)=\inf \{\ t> 0:x-tv\notin \partial \Omega \},
\end{equation*}
and we also define $x_{\mathbf{b}}(x,v)=x-t_{\mathbf{b}}(x,v)v\in \partial \Omega $. Note that for any $(x,v)$, we use $t_{\mathbf{b}}(x,v)$
whenever it is well-defined.

Set the perturbation in a standard way $F=\mu+\sqrt{\mu}f$, the initial boundary value problem \eqref{be}, \eqref{id}, \eqref{dbd} and \eqref{sbd}
can be reformulated as
\begin{equation}\label{BE}
\pa_tf+v\cdot\na_x f+Lf=\Ga(f,f),
\end{equation}
\begin{equation}\label{ID}
f(0,x,v)=f_0(x,v),
\end{equation}
with the boundary condition
\begin{equation}\label{DBD}
f(t,x,v)|_{\ga_-}=\sqrt{\mu}\int_{n(x)\cdot v'>0}f(t,x,v')\sqrt{\mu(v')}n(x)\cdot v'dv',
\end{equation}
and
\begin{equation}\label{SBD}
f(t,x,v)|_{\ga_-}=f(t,x,R_xv),
\end{equation}
respectively.
The nonlinear operator $\Ga(\cdot,\cdot)$ and linear operator $L$ in \eqref{BE} are defined as
$$
\Ga(f_1,f_2)=\frac{1}{\sqrt{\mu}}Q(\sqrt{\mu}f_1,\sqrt{\mu}f_2),
$$
and
\begin{equation}\label{L.def}
Lf=-\frac{1}{\sqrt{\mu}}\{Q(\mu,\sqrt{\mu}f)+Q(\sqrt{\mu}f,\mu)\},
\end{equation}
respectively. $L$ can be further split into $L=\nu -K$
with $K$
a suitable integral kernel defined by \eqref{def.k}
in Section \ref{pre}, and
the collision frequency
$\nu (v)\equiv \int_{\R^{3}\times
\S^{2}}b_0(\ta)|u-v|^\varrho\mu(u)dud\omega$
for $-3<\varrho<0$, moreover there exists constant $C_\varrho>0$ such that
\begin{equation}\label{confre}
\frac{1}{C_\varrho}\{1+|v|^2\}^{\varrho/2 }\leq\nu (v)\leq C_\varrho \{1+|v|^2\}^{\varrho/2 }.
\end{equation}

Under the conditions \eqref{DBD} or \eqref{SBD}, it is straightforward to check that
$$
\int_{\ga_+}f(t,x,v)\sqrt{\mu(v)}|n(x)\cdot v|dS_xdv=\int_{\ga_-}f(t,x,v)\sqrt{\mu(v)}|n(x)\cdot v|dS_xdv,
$$
where $dS_x$ is the surface element.

Hence, in terms of perturbation $f(t,x,v)$, the mass conservation
\begin{equation}\label{mass.c}
\int_{\Om\times\R^3}f(t,x,v)\sqrt{\mu(v)}dxdv=0
\end{equation}
holds true for either of boundary conditions \eqref{DBD} and \eqref{SBD} by further assuming initially \eqref{be} has the same mass as
the Maxwellian $\mu$.

For the specular reflection condition \eqref{SBD}, in addition to the mass conservation \eqref{mass.c},
the energy conservation law also holds for $t\geq0$, that is
\begin{equation}\label{eng.c}
\int_{\Om\times\R^3}|v|^2f(t,x,v)\sqrt{\mu(v)}dxdv=0.
\end{equation}
Moreover, if the domain $\Omega $ has any axis of rotation symmetry
(\ref{axis}), then we further assume the corresponding conservation of
angular momentum is valid for all $t\geq 0$
\begin{equation}
\int_{\Omega \times \R^{3}}\{(x-x_{0})\times \varpi \}\cdot vf(t,x,v)
\sqrt{\mu }dxdv=0.  \label{axiscon}
\end{equation}

\subsection{Main results}
We introduce a weight function
\begin{equation}\label{wfun}
w_{q,\ta,\vth}=\exp\left\{\frac{q|v|^\ta}{8}+\frac{q|v|^\ta}{8(1+t)^\vth}\right\}, \ (q,\ta)\in \CA_{q,\ta},\ 0\leq\vth< -\frac{\ta}{\varrho},
\end{equation}
where
\begin{equation*}\label{qta}
\CA_{q,\ta}=\{(q,\ta)|q>0, \textrm{if}\ 0<\ta<2, \textrm{and}\ 0<q<1,\ \textrm{if}\ \ta=2\}.
\end{equation*}
For the sake of simplicity, we denote $w_{q,\ta,0}=w_{q,\ta}=\exp(\frac{q|v|^\ta}{4})$ throughout the paper.

We now state our main results as follows
\begin{theorem}\label{ms1}
Let $-3<\varrho<0$ and $(q,\ta)\in \CA_{q,\ta}$. Assume the mass conservation \eqref{mass.c} holds for $f_0(x,v)$. Then there exists small constant $\vps_0>0$ such that
if $F_0(x,v)=\mu+\sqrt{\mu}f_0(x,v)\geq0$ and $\|w_{q,\ta}f_0\|_{\infty}\leq\vps_0$, there exists a unique solution $F(t,x,v)=\mu+\sqrt{\mu}f(t,x,v)\geq0$ for the Boltzmann equation
\eqref{be} and \eqref{id} with the diffuse reflection boundary condition \eqref{dbd}. Moreover, there is some $C>0$ such that
\begin{equation}\label{u.bd}
\sup\limits_{0\leq t\leq+\infty}\|w_{q,\ta}f(t)\|_{\infty}\leq C\|w_{q,\ta}f_0\|_{\infty}.
\end{equation}
 Furthermore, we assume $\Omega $ is strictly convex and $f_0(x,v)$ is continuous away from the set $\gamma _{0}$ and
\begin{equation*}\label{id.BD}
f_0(x,v)|_{\ga_-}=\sqrt{\mu}\int_{n(x)\cdot v'>0}f_0(x,v')\sqrt{\mu(v')}n(x)\cdot v'dv'.
\end{equation*}
Then, $f(t,x,v)$ is continuous in $[0,+\infty)\times\{\overline{\Om}\times\R^3\setminus \ga_0\}$.
Moreover, let $\rho_0=\frac{\ta}{\ta-\varrho}$, there exist
$C>0$ and $\la_0>0$ independent of $t$ such that
\begin{equation}\label{decay}
\begin{split}
\|f(t)\|_\infty\leq& Ce^{-\la_0t^{\rho_0}}\|w_{q,\ta}f_0\|_\infty.
\end{split}
\end{equation}
\end{theorem}
\begin{theorem}
\label{specularnl}
Let $0<\vth< -\frac{\ta}{\varrho}$ with $-3<\varrho<0$ and $(q,\ta)\in\CA_{q,\ta}.$
Assume that $\xi $ is both strictly convex \eqref{scon} and analytic,
and the mass \eqref{mass.c} and energy \eqref{eng.c} are conserved for $f_{0}$%
. In the case of $\Omega $ has any rotational symmetry \eqref{axis}, we
further require the corresponding angular momentum \eqref{axiscon} is conserved
for $f_{0}$. Then there exists $\vps_0 >0$ such that if $F_{0}(x,v)=\mu +%
\sqrt{\mu }f_{0}(x,v)\geq 0$ and $\|w_{q,\ta,\vth}f_{0}\|_{\infty }\leq \vps_0 ,$ there
exists a unique solution $F(t,x,v)=\mu +\sqrt{\mu }f(t,x,v)\geq 0$ to the Boltzmann equation
\eqref{be} and \eqref{id} with the
specular reflection boundary condition \eqref{sbd}. Moreover, let $\rho_1=\frac{\ta+\vth\varrho}{\ta-\varrho}$, there exist $\la_0>0$ and $C>0$
  such that
\begin{equation*}
\|w_{q,\ta,\vth}f(t)\|_{\infty }\leq
Ce^{-\la_0t^{\rho_1}}\|w_{q,\ta,\vth}f_{0}\|_{\infty }.
\end{equation*}%
Furthermore, if $f_{0}(x,v)$ is continuous except on the set $\gamma _{0}$ and
\begin{equation*}
f_{0}(x,v)=f_{0}(x,R_xv)\text{ on }\partial \Omega,
\end{equation*}%
then $f(t,x,v)$ is continuous in $[0,\infty )\times \{\bar{\Omega}\times
\R^{3}\setminus \gamma _{0}\}.$
\end{theorem}

\begin{remark}
It should be pointed out that the method developed in Theorem \ref{specularnl} can be applied to verify Theorem \ref{ms1}, and it can also be used to handle the other two kinds of boundary conditions: in flow and bounce back reflection. Moreover, one can see that the approach developed in the proof of Theorem \ref{ms1} is more direct and constructive while the method used in the proof of Theorem \ref{specularnl} is more simpler, both of them have their merits. 
In addition, it is straightforward to know that the decay exponents 
satisfy $\rho_0=\lim\limits_{\vth\rightarrow 0^+}\rho_1.$
It is quite interesting to improve $\rho_1$ to $\rho_0$ which coincides with the decay rate for the periodic boundary condition \cite{SG-08}.
\end{remark}

Let us now give some comments on the difficulties associated with Theorems \ref{ms1} and \ref{specularnl}. Compared with the previous works such as \cite{Guo-2010}, \cite{BG-2015}, \cite{EGKM-13} and \cite{Yu}, a remarkable feature of our problems is that the collision frequency $\nu$ has no positive lower bound so that the Boltzmann solution could not be expected to decay exponentially in $L^\infty$ immediately. However, an instead decay rate plays a key role to establish the global existence of the Botlzmann equation in bounded domain, see Lemma 19 in \cite[pp.761]{Guo-2010} and also \cite[pp.81]{Ukai-86}. This time decay rate is essentially applied to eliminate the possible growth created by the $k-$times bounce-back reflection ($k$ is large). Our strategies to overcome this difficulty are briefly stated as follows:

For the diffuse reflection boundary condition, one needs some careful estimates on the integrals along the stochastic cycles so as to obtain the global existence by using only the $L^2$ decay. One of the key points in this paper is to develop a direct and unified approach to establish the global existence of the linearized Botlzmann equation with the diffuse reflection boundary condition. More specifically, instead of applying the time decay in $L^\infty$  to obtain the global existence cf. \cite{Guo-2010, BG-2015, EGKM-13}, we first construct a local solution via an iteration method, then directly deduce the {\it a priori} estimate which is uniform in time by means of a
refined estimates on the integrals defined on the stochastic cycles and the $L^2$ time decay for linearized equation. Finally, the global solution is obtained with the aid of the standard continuity argument. Among those, the main step is to establish the following type of uniform estimates
\begin{equation}\label{key1}
\begin{split}
\int_{\prod_{j=1}^{k-1}
\mathcal{V}_{j}}&\sum_{l=1}^{k-1}\mathbf{1}_{\{t_{l+1}\leq
0<t_{l}\}}\int_{0}^{t_l}\iint dv^{\prime }dv^{\prime \prime }\int_{\prod_{j=1}^{k-1}\mathcal{V
}_{j}^{\prime }}\frac{e^{-\nu (v^{\prime })(s-t_{1}^{\prime })}}
{\tilde{w}_{q,\ta}(v^{\prime })} \\&\quad\times\sum_{l^{\prime }=1}^{k-1}\mathbf{1}_{\{t_{l^{\prime}+1}^{\prime }>0\}}
\left\{\int_{t_{l^{\prime}+1}^{\prime }}^{t_{l^{\prime }}^{\prime }-\frac{1}{k^2(s)}}+\int_{t_{l^{\prime }}^{\prime }-\frac{1}{k^2(s)}}^{t_{l^{\prime }}^{\prime }}\right\}\mathbf{k}^\chi_{w}(v_{l},v^{\prime })
\mathbf{k}^\chi_{w}(v_{l^{\prime }}^{\prime },v^{\prime \prime
})
 \\&\quad\times |h^j(s_{1,}x_{l^{\prime }}^{\prime
}+(s_{1}-t_{l^{\prime }}^{\prime })v_{l^{\prime }}^{\prime },v^{\prime
\prime })|d\Sigma^w_{l^{\prime }}(s_{1})ds_{1}d\Sigma^w_{l}(s)ds\\
\leq&C_{q,\ta}\left(\frac{1}{T_0^{5/4}}+\frac{1}{N}\right)\sup\limits_{0\leq s\leq t_1}\|h^j(s)\|_{\infty}+C_N\sup\limits_{0\leq s\leq t_1}\left\{e^{\frac{\la_0}{2}s^{\rho_0}}
\left\Vert \frac{h^j(s)}{w_{q,\ta }(v)}\right\Vert _{2}\right\},
\end{split}
\end{equation}
here $k(s)=k=C'_{1}[\al(s)]^{5/4}\geq C'_{1}T_0^{5/4}$ and $C'_{1}>0$ is a constant.
To derive \eqref{key1}, the following key observation is used
$$
\sum\limits_{l=1}^{k-1}\int_{t_{l+1}}^{t_l}e^{\nu(v_m)(s-t_1)}\nu(v_m)ds\leq \int_{t_{k}}^{t_1}e^{\nu(v_m)(s-t_1)}\nu(v_m)ds\leq C,
$$
where $v_m$ is defined to satisfy $|v_m|=\max\{|v_1|,|v_2|,\cdots,|v_{k-1}|\}$ for $\max\{|v_1|,|v_2|,$ $\cdots,|v_{k-1}|\}\leq k$.

In addition, a delicate Banach space $\FX_\de(t)$ is designed to capture the properties of the solution in $L^\infty\cap L^2$ space so that the global existence and the exponential decay in $L^2-$norm can be simultaneously obtained.
The rapid time decay $e^{-\frac{\la_0}{2}s^{\rho_0}}$ in $L^2-$norm is adopted to control the Jacobian determinate when we convert the $L^\infty-$norm to $L^2-$norm.

It is also interesting to note that the estimate \eqref{decay} is a consequence of the interpolation technique basing on the $L^2$ energy estimate and the weighted $L^\infty$ estimate for the global solution as well as Young's inequality
\begin{equation}\label{you}
e^{-\nu(v)t}w^{-1}_{q/2,\ta}(v)\leq e^{-\la_0t^{\rho_0}},\ \rho_0=\frac{\ta}{\ta-\varrho}.
\end{equation}
\eqref{you} means that
one has to trade between the exponential decay rates and the additional exponential momentum weight on the solution itself
in order to obtain the rapid time decay rates, this also reveals that the additional velocity weight imposed on the initial data in \eqref{decay} is seen as a compensation for the exponential decay rates.

As to the specular reflection boundary condition, we can not expect to obtain the similar estimate as \eqref{key1}. There are two mathematical difficulties: one is the times of bounce back reflection $k$ and $k'$ in this situation both grow exponentially in time according to Velocity Lemma \ref{velocity}, hence the summation of the integral is out of control. The other is that it is impractical to compute the Jacobian
    determinate
    \begin{equation*}
\det \left( \frac{%
\partial \{x_{k^{\prime }}^{\prime }+( s _1-t_{k^{\prime }}^{\prime
})v_{k^{\prime }}^{\prime }\}}{\partial v^{\prime }}\right),
\end{equation*}
which depends on $t,x,v$,$k$ and $k'$. In this sense, the method developed in the case of diffuse reflection boundary condition
can not be applied to the case of specular reflection boundary condition. Precisely speaking,
one can not first obtain the global existence of \eqref{BE}, \eqref{ID} and \eqref{SBD} in some higher weighted $L^\infty$ space and therefore is not able to deduce the time decay rates in lower weighted $L^\infty$ space. As a sequence, we are forced to resort the bootstrap argument as that of \cite[Lemma 19, pp.761]{Guo-2010}. As mentioned before, to apply the bootstrap argument, the key point is to obtain the rapid time decay rates without any growth. Nevertheless, it seems impossible to achieve it due to the zero lower bound of the collision frequency.
To deal with this difficulty, we introduce a time-velocity weight
$$
w_{q,\ta,\vth}=\exp\left\{\frac{q|v|^\ta}{8}+\frac{q|v|^\ta}{8(1+t)^\vth}\right\}
$$
which has been used in \cite{D-im, DLYZ-VMB, DYZ-h, DYZ-s} to handle the non-hard sphere Boltzmann equations with self-consistent forces, by using this weight, we are able to deduce a time-dependent lower bound for a revised collision frequency, say
\begin{equation*}
\widetilde{\nu}(v,t)=\nu +\frac{\vth q|v|^\ta}{8(1+t)^{\vth+1}}
\geq C_{\varrho,q,\vth}(1+t)^{\frac{(1+\vth)\varrho}{\ta-\varrho}}.
\end{equation*}
So, the desired time decay rates will be naturally obtained. This is another key contribution of the present paper.



Due to the singularity of the collision kernel, the integral operator $K$ raises another difficulty when we carry out $L^\infty$ estimates for the linearized equation. Similar to the study of the Cauchy problem of Boltzmann equation on torus \cite{G-soft,S,SG-08}, we introduce a cutoff function $\chi$ to split $K=K^{\chi}+K^{1-\chi}$. With this decomposition, we only need to iterate $K^\chi$ twice \cite{Vi} to obtain the desired estimates since $K^{1-\chi}$ is small and can be controlled directly.

The estimates of the nonlinear operator $\Ga(\cdot,\cdot)$ in terms of the exponential weighted norm $\|w_{q,\ta,\vth}\cdot\|_\infty$ are subtle.
To avoid additional weight,
we estimate $w_{q/2,\ta,\vth}(v)$ as
$$
w_{q/2,\ta,\vth}(v)\leq\frac{1}{2} (w_{q,\ta,\vth}(v')+w_{q,\ta,\vth}(u'))
$$
instead of $
w_{q/2,\ta,\vth}(v)\leq w_{q/2,\ta,\vth}(v')w_{q/2,\ta,\vth}(u').
$

\medskip
The organization of the paper is as follows. In Section \ref{pre}, we collect some significant estimates for the later use. Section \ref{DVP}
is devoted to the study of the Boltzmann equation with diffuse reflection boundary condition.
The global existence and exponential time decay for Boltzmann equation with specular reflection boundary condition are presented in Section \ref{SVP}.

\medskip
\subsection{Notations and Norms}
We now list some notations and norms used in the paper.
 \begin{itemize}
 \item
 Throughout this paper,  $C$ denotes some generic positive (generally large) constant and $\la,\la_1,\la_2$ as well as $\la_0$ denote some generic positive (generally small) constants, where $C$, $\la,\la_1,\la_2$ and $\la_0$  may take different values in different places. $D\lesssim E$ means that  there is a generic constant $C>0$
such that $D\leq CE$. $D\sim E$
means $D\lesssim E$ and $E\lesssim D$.
\item Let $1\leq p\leq \infty$, we denote $\Vert \,\cdot \,\Vert _{p }$ the $L^{p }(\Omega
\times \R^{3})-$norm or the $L^{p }(\Omega )-$norm or $L^{p }(\Omega\cup\ga )-$norm,
while $|\,\cdot \,|_{\infty }$ is either the $L^{\infty }(\partial \Omega
\times \R^{3})-$norm or the $L^{\infty }(\partial \Omega )-$norm at
the boundary. Moreover we denote $\|\cdot \|_{\nu}\equiv \|\nu^{1/2}\cdot
\|_2$, and
$(\cdot,\cdot)$ denotes the $L^{2}$ inner product in
$\Omega\times {\R}^{3}$  with
the $L^{2}$ norm $\|\cdot\|_2$.

\item As to
the phase boundary integration, we denote $d\gamma = |n(x)\cdot v|dS(x)dv$,
where $dS(x)$ is the surface element and for $1\leq p<+\infty$, we define $|f|_p^p = \int_{\gamma}
|f(x,v)|^p d\gamma \equiv\int_{\gamma} |f(x,v)|^p $ and the corresponding
space as $L^p(\partial\Omega\times\R^3;d\gamma)=L^p(\partial\Omega%
\times\R^3)$. Furthermore $|f|_{p,\pm}= |f \mathbf{1}_{\gamma_{\pm}}|_p$
and $|f|_{\infty,\pm}= |f \mathbf{1}_{\gamma_{\pm}}|_{\infty}$. For simplicity,  we use $%
|f|_p^p = \int_{\partial\Omega} |f(x)|^p dS(x)\equiv\int_{\partial\Omega}
|f(x)|^p $. We also denote $f_{\pm}=f_{\gamma_{\pm}}=f\mathbf{1}_{\gamma_{\pm}}$.

\item Finally, we define
\begin{equation*}
P_{\gamma }f(x,v)=\sqrt{\mu (v)}\int_{n(x)\cdot v^{\prime }>0}f(x,v^{\prime
})\sqrt{\mu (v^{\prime })}(n(x)\cdot v^{\prime })dv^{\prime }, \ \ x\in\pa\Om.
\label{pgamma}
\end{equation*}%
Thanks to \eqref{mu.n}, $P_{\gamma }f$ defined on $\pa\Om\times\R^3$, 
is an $L_{v}^{2}$-projection with respect to the measure $|n(x)\cdot v|$ for
any boundary function $f$ defined on $\gamma _{+}$. We also denote $\{I-P_\ga\}f=f-P_\ga f$.
\end{itemize}

\medskip

\section{Preliminary}\label{pre}
In this section, we collect some basic definitions and estimates for the later proof.
We start with the analysis of $K$, from \eqref{L.def}, a standard decomposition for $K$ is the following
\begin{equation}\label{def.k}
\begin{split}
Kf=&\int_{\R^3\times \S^2}|u-v|^{\varrho}b_0(\cos \ta)\sqrt{\mu(u)}\left\{f(u')\sqrt{\mu(v')}+f(v')\sqrt{\mu(u')}\right\}dud\omega\\
&-\sqrt{\mu(v)}\int_{\R^3\times \S^2}|u-v|^{\varrho}b_0(\cos \ta)\sqrt{\mu(u)}f(u)dud\omega\eqdef K_2-K_1.
\end{split}
\end{equation}
To treat the singularity in $K$, we introduce
a smooth cutoff function $0\leq\chi\leq1$ such that
\begin{eqnarray*}
\chi(s)=\left\{
\begin{array}{rll}
1,&\  \ s\geq2\eps,\\
0,&\ \ s\leq\eps.
\end{array}
\right.
\end{eqnarray*}
Use $\chi$ to split $K_2=K_2^\chi+K_2^{1-\chi}$ where
\begin{equation*}
\begin{split}
K_2^{\chi}f=&\int_{\R^3\times \S^2}\chi(|u-v|)|u-v|^{\varrho}b_0(\cos \ta)\sqrt{\mu(u)}\\&\times\left\{f(u')\sqrt{\mu(v')}+f(v')\sqrt{\mu(u')}\right\}dud\omega.
\end{split}
\end{equation*}
With this, it follows from \cite[pp.294]{SG-08} that
$$
K_2^{\chi}f=\int_{\R^3}{\bf k}_2^\chi(v,u)f(u)du,
$$
where
\begin{equation*}
|{\bf k}_2^\chi(v,u)|\leq C\eps^{\varrho-1}\frac{\exp\left(-\frac{1}{8}|u-v|^2-\frac{1}{8}\frac{(|v|^2-|u|^2)^2}{|v-u|^2}\right)}{|v-u|},
\end{equation*}
or
\begin{equation}\label{k2.ke}
|{\bf k}_2^\chi(v,u)|\leq C\frac{\exp\left(-\frac{s_2}{8}|u-v|^2-\frac{s_1}{8}\frac{(|v|^2-|u|^2)^2}{|v-u|^2}\right)}{|v-u|(1+|v|+|u|)^{1-\varrho}},
\end{equation}
for any $0<s_1<s_2<1.$
As to $K_1$, it is obvious to see
$$
K_1f=\int_{\R^3}{\bf k}_1(v,u)f(u)du,
$$
with ${\bf k}_1(v,u)=\int_{\S^2}|u-v|^{\varrho}b_0(\cos \ta)\sqrt{\mu(u)}\sqrt{\mu(v)}d\omega.$
Analogously, we also denote $K^\chi=K_2^\chi-K_1^\chi$ and  $K^{1-\chi}=K_2^{1-\chi}-K_1^{1-\chi}.$

Prior to the study of the property of the operators $K$ and $\Ga$, we present the following elementary inequality:
\begin{lemma}\label{el.ine}
If $0< p\leq1$, for any $x,y\geq0$, it holds that
\begin{equation}\label{p.ine}
(x+y)^p\leq x^p+y^p.
\end{equation}
If $p>1$, for any $x,y\geq0$, it holds
\begin{equation}\label{p.ine2}
(x+y)^p\leq 2^{p-1}(x^p+y^p).
\end{equation}
\end{lemma}
\begin{proof}
If $y=0$, \eqref{p.ine} is obviously true. If $y>0$, \eqref{p.ine} is then equivalent to
$$
\left(1+\frac{x}{y}\right)^p-\left(\frac{x}{y}\right)^p-1\leq0.
$$
It is easy to check the function $g(t)=(1+t)^p-t^p-1$ is monotonically decreasing for $0< p\leq1$, and moreover $g(0)=0$, therefore \eqref{p.ine} is also valid for $y>0$. If $p>1$, \eqref{p.ine2} directly follows from the convexity of $t^p$.
This completes the proof of Lemma \ref{el.ine}.

\end{proof}
We now summarize the properties of $K$ as follows:
\begin{lemma}\label{es.k}
Assume $-3<\varrho<0$, $(q,\ta)\in \CA(q,\ta)$ and $\vth\geq0$. It holds that for $\eta>0$
\begin{equation}\label{K.ip1}
(Kf_1,w^2_{q,\ta,\vth}f_2)\leq \left\{\eta\|w_{q,\ta,\vth}f_1\|_\nu+C(\eta)\|{\bf 1}_{|v|\leq C(\eta)}f_1\|\right\}\|w_{q,\ta,\vth}f_2\|_\nu,
\end{equation}
especially,
\begin{equation}\label{K.ip2}
(Kf_1,w^2_{q,\ta,\vth}f_2)\leq C\|w_{q,\ta,\vth}f_1\|_\nu\|w_{q,\ta,\vth}f_2\|_\nu,\ \ (Kf_1,f_2)\leq C\|f_1\|_\nu\|f_2\|_\nu.
\end{equation}
In addition, for any $l\geq0$, one has
\begin{equation}\label{K.ip3}
\langle v\rangle^lw_{q,\ta,\vth}K^{1-\chi}\left(\frac{|h|}{\langle v\rangle^lw_{q,\ta,\vth}}\right)\leq C(\mu(v))^{\min\{1/8q,\frac{|1-q|}{8}\}}\eps^{\varrho+3}\|h\|_{\infty},
\end{equation}
and
\begin{equation}\label{K.ip4}
\langle v\rangle^lw_{q,\ta,\vth}\int_{\R^3}{\bf k}^{\chi}(v,u)\left(\frac{e^{\varepsilon|v-u|^2}|h(u)|}{\langle u\rangle^lw_{q,\ta,\vth}(u)}\right)du\leq C_{q,\ta}\langle v\rangle^{\varrho-2}\|h\|_{\infty},
\end{equation}
where $\vps>0$ and sufficiently small and $\langle v\rangle=\sqrt{1+|v|^2}$. 

\end{lemma}
\begin{proof}
We only detail the proof for \eqref{K.ip3} and \eqref{K.ip4}, since the strategy to prove \eqref{K.ip1} is basically the same as Lemma 2 of \cite[pp.296]{SG-08}, and \eqref{K.ip2} directly follows from \eqref{K.ip1}. Notice that $K^{1-\chi}=K_2^{1-\chi}-K_1^{1-\chi}$, we first consider the estimates for $K_1^{1-\chi}$. Recall
\begin{equation*}
w_{q,\ta,\vth}=\exp\left\{\frac{q|v|^\ta}{8}+\frac{q|v|^\ta}{8(1+t)^\vth}\right\}
=\exp\left\{\frac{q}{8}(1+(1+t)^{-\vth})|v|^\ta\right\}, \ (q,\ta)\in \CA_{q,\ta}.
\end{equation*}
Let $\widetilde{q}=\frac{q}{2}(1+(1+t)^{-\vth})$, then $q/2< \widetilde{q}\leq q$. Direct calculation yields
$$\langle v\rangle^lw_{q,\ta,\vth}\sqrt{\mu(v)}\leq C_{q,\ta}(\mu(v))^{\min\{1/8q,\frac{|1-q|}{8}\}},$$
then it is easy to obtain
\begin{equation*}
\begin{split}
\langle v\rangle^l&w_{q,\ta,\vth}K_1^{1-\chi}\left(\frac{|h|}{\langle v\rangle^lw_{q,\ta,\vth}}\right)\\
=&\langle v\rangle^lw_{q,\ta,\vth}\sqrt{\mu(v)}\int_{\R^3\times \S^2}(1-\chi(|u-v|))|u-v|^{\varrho}b_0(\cos \ta)\sqrt{\mu(u)}\\&\quad\times\left(\frac{|h(u)|}{\langle u\rangle^lw_{q,\ta,\vth}(u)}\right)d\omega du\\
\leq& C_{q,\ta,\vth}(\mu(v))^{\min\{1/8q,\frac{|1-q|}{8}\}}\int_{|v-u|\leq 2\eps}|u-v|^{\varrho}du\|h\|_{\infty}\\
\leq& C(\mu(v))^{\min\{1/8q,\frac{|1-q|}{8}\}}\eps^{\varrho+3}\|h\|_{\infty}.
\end{split}
\end{equation*}
For the contribution of ${\bf k}_1^{\chi}$ in \eqref{K.ip4}, it follows that
\begin{equation*}
\begin{split}
\langle v\rangle^l&w_{q,\ta,\vth}\int_{\R^3}{\bf k}_1^{\chi}(v,u)\left(\frac{e^{\varepsilon|v-u|^2}|h(u)|}{\langle u\rangle^lw_{q,\ta,\vth}(u)}\right)du\\
=&\langle v\rangle^lw_{q,\ta,\vth}\sqrt{\mu(v)}\int_{\R^3\times \S^2}\chi(|u-v|)|u-v|^{\varrho}b_0(\cos \ta)\sqrt{\mu(u)}\\&\quad\times\left(\frac{e^{\varepsilon|v-u|^2}|h(u)|}{\langle u\rangle^lw_{q,\ta,\vth}(u)}\right)d\omega du\\
\leq& C_{q,\ta}\int_{\R^3}|u-v|^{\varrho}(\mu(v)\mu(u))^{\min\{1/8q,\frac{|1-q|}{8}\}}e^{\varepsilon|v-u|^2}du\|h\|_{\infty}\\
\leq& C\langle v\rangle^{\varrho}(\mu(v))^{\min\{1/16q,\frac{|1-q|}{16}\}}\|h\|_{\infty},
\end{split}
\end{equation*}
where the last inequality is due to $\int_{\R^3}|u-v|^{\varrho}(\mu(u))^{\min\{1/16q,\frac{|1-q|}{16}\}}du\leq C\langle v\rangle^{\varrho}.$

We now turn to derive the contributions of $K^{1-\chi}_2$ in \eqref{K.ip3}. In light of \eqref{def.k}, on the one hand, we have
\begin{equation}\label{K21}
\begin{split}
\langle v\rangle^l&w_{q,\ta,\vth}K_2^{1-\chi}\left(\frac{|h|}{\langle v\rangle^lw_{q,\ta,\vth}}\right)\\
=&\langle v\rangle^lw_{q,\ta,\vth}\int_{\R^3\times \S^2}(1-\chi)(|u-v|)|u-v|^{\varrho}b_0(\cos \ta)\sqrt{\mu(u)}\\
&\times\left\{\frac{|h(u')|}{\langle u'\rangle^lw_{q,\ta,\vth}(u')}\sqrt{\mu(v')}+\frac{|h(v')|}{\langle v'\rangle^lw_{q,\ta,\vth}(v')}\sqrt{\mu(u')}\right\}d\omega du.
\end{split}
\end{equation}
On the other hand, \eqref{v.re} and $|v-u|\leq 2\eps$ imply
\begin{eqnarray}\label{v.re2}
\left\{\begin{array}{rll}
|v'|=&|v+[(u-v)\cdot\omega]\om|\geq|v|-|v-u|\geq |v|-2\eps,\\[2mm]
|u'|=&|v+u-v-[(u-v)\cdot\omega]\om|\geq|v|-2|v-u|\geq|v|-4\eps.
\end{array}\right.
\end{eqnarray}
Using $\langle v\rangle^lw_{q,\ta,\vth}\sqrt{\mu(v)}\leq C_{q,\ta}(\mu(v))^{\min\{1/8q,\frac{|1-q|}{8}\}}$ again, we get from \eqref{K21} and \eqref{v.re2} that
\begin{equation*}
\begin{split}
\langle v\rangle^lw_{q,\ta,\vth}K_2^{1-\chi}\left(\frac{|h|}{\langle v\rangle^lw_{q,\ta,\vth}}\right)
\leq&  C_{q,\ta}(\mu(v))^{\min\{1/8q,\frac{|1-q|}{8}\}}\int_{|v-u|\leq 2\eps}|u-v|^{\varrho}du\|h\|_{\infty}\\
\leq& C(\mu(v))^{\min\{1/8q,\frac{|1-q|}{8}\}}\eps^{\varrho+3}\|h\|_{\infty}.
\end{split}
\end{equation*}
It remains now to deduce the contribution of ${\bf k}^\chi_2$ in \eqref{K.ip4}. Recall \eqref{k2.ke}, take $s_0=\min\{s_1,s_2\}$ to obtain
\begin{equation*}
\begin{split}
\langle v\rangle^l&w_{q,\ta,\vth}\int_{\R^3}{\bf k}_2^{\chi}(v,u)\left(\frac{e^{\varepsilon|v-u|^2}|h(u)|}{\langle u\rangle^lw_{q,\ta,\vth}(u)}\right)du\\
\leq&C\|h\|_{\infty}\langle v\rangle^{\varrho-1}\langle v\rangle^lw_{q,\ta,\vth}\int_{\R^3}\frac{\exp\left(-\frac{s_0}{8}|u-v|^2-\frac{s_0}{8}\frac{(|v|^2-|u|^2)^2}{|v-u|^2}\right)}{|v-u|}
\\&\quad\times\left(\frac{e^{\varepsilon|v-u|^2}}{\langle u\rangle^lw_{q,\ta,\vth}(u)}\right)du \eqdef\CK_0.
\end{split}
\end{equation*}
Next, from \eqref{wfun}, we notice that for some $C_{l}>0$ and $\ta=2$
\begin{equation*}
\left\vert \frac{w_{q,\ta,\vth}(v)}{w_{q,\ta,\vth}(u)}\right\vert \leq C_l[1+|v-u|^{2}]^{l}e^{-\widetilde{q} \{|u|^{2}-|v|^{2}\}}.
\end{equation*}%
Let $v-u=\eta $ and $u=v-\eta $ in the integral of $\CK_0$. We then compute the total exponent in
$\mathbf{k}_2^{\chi}(v,u)\frac{w_{q,2,\vth}(v)}{w_{q,2,\vth}(u)}$ as:
\begin{equation*}
\begin{split}
-\frac{s_0}{8}|\eta |^{2}&-\frac{s_0}{8}\frac{||\eta |^{2}-2v\cdot \eta |^{2}}{%
|\eta |^{2}}-\frac{\widetilde{q}}{4} \{|v-\eta |^{2}-|v|^{2}\} \\
=&-\frac{s_0}{4}|\eta |^{2}+\frac{s_0}{2}v\cdot \eta -\frac{s_0}{2}\frac{|v\cdot
\eta |^{2}}{|\eta |^{2}}-\frac{\widetilde{q}}{4} \{|\eta |^{2}-2v\cdot \eta \} \\
=&-\frac{1}{4}(\widetilde{q}+s_0)|\eta |^{2}+\frac{1}{2}(s_0+\widetilde{q})v\cdot \eta -%
\frac{s_0}{2}\frac{\{v\cdot \eta \}^{2}}{|\eta |^{2}}.
\end{split}
\end{equation*}%
Let $\widetilde{q}\leq q <s_0,$ the discriminant of the above quadratic form of
$|\eta |$ and $\frac{v\cdot \eta }{|\eta |}$ is
\begin{equation*}
\Delta =\frac{1}{4}(s_0+\widetilde{q})^2-(\widetilde{q}+s_0)\frac{s_0}{2}=\frac{1}{4}(\widetilde{q} ^{2}-s_0^2)<0.
\end{equation*}
Notice that $q/2< \widetilde{q}\leq q,$ we thus have, for $\varepsilon >0$ sufficiently small and $q<s_0$, there is $C_{q}>0$ independent of $\vth$ such that the
following perturbed quadratic form is still negative definite
\begin{equation}\label{k0ep1}
\begin{split}
-\frac{s_0-8\varepsilon }{8}|\eta |^{2}&-\frac{s_0-8\varepsilon }{8}\frac{||\eta
|^{2}-2v\cdot \eta |^{2}}{|\eta |^{2}}-\frac{\widetilde{q}}{4} \{|\eta |^{2}-2v\cdot \eta \}
\\
\leq &-C_{q}\left\{|\eta |^{2}+\frac{|v\cdot \eta |^{2}}{|\eta |^{2}}\right\}
=-C_{q}\left\{\frac{|\eta |^{2}}{2}
+\left(\frac{|\eta |^{2}}{2}+\frac{|v\cdot\eta |^{2}}{|\eta |^{2}}\right)\right\} \\
\leq &-C_{q}\left\{\frac{|\eta |^{2}}{2}+|v\cdot \eta |\right\}.
\end{split}
\end{equation}%
%
If $0<\ta<2$, Lemma \ref{el.ine} yields
$$
|v|^\ta-|u|^\ta\leq C_\ta|\eta|^\ta.
$$
Therefore, one also has
\begin{equation}\label{k0ep2}
\begin{split}
-\frac{s_0-8\vps}{8}|\eta|^2&-\frac{s_0}{8}
\frac{||\eta|^2-2v\cdot\eta|^2}{|\eta|^2}+\frac{\widetilde{q}C_\ta}{4}\eta^\ta\\
\leq&-\frac{s_0-9\vps}{8}|\eta|^2-\frac{s_0-9\vps}{8}
\frac{||\eta|^2-2v\cdot\eta|^2}{|\eta|^2}+C_{q,\ta}
\\
\leq&-C_{s_0}
\left\{|\eta |^{2}+\frac{|v\cdot \eta |^{2}}{|\eta |^{2}}\right\}+C_{q,\ta}
\leq-C_{s_0}
\left\{\frac{|\eta |^{2}}{2}+|v\cdot \eta |\right\}+C_{q,\ta}.
\end{split}
\end{equation}
Plugging \eqref{k0ep1} or \eqref{k0ep2} into $\CK_0$, we obtain
\begin{equation*}
\begin{split}
\CK_0
\leq&C_{q,\ta}\langle v\rangle^{\varrho-1}\|h\|_{\infty}\int_{\R^3}\frac{\langle\eta\rangle^l}{|\eta|}\exp\left\{-C_{q}
\left\{\frac{|\eta |^{2}}{2}+|v\cdot \eta |\right\}\right\}d\eta.
\end{split}
\end{equation*}
Next, we make another change of variable $\eta_\parallel=(\eta\cdot\frac{v}{|v|})\frac{v}{|v|}$ and $\eta_\perp=\eta-\eta_\parallel$
so that $v\cdot \eta=|v||\eta_\parallel|$, this leads us to
\begin{equation*}
\begin{split}
\CK_0
\leq&C_{q,\ta}\langle v\rangle^{\varrho-1}\|h\|_{\infty}\int_{\R^2}\frac{1}{|\eta_\perp|}\exp\left\{-\frac{C_{q}}{4}
|\eta_\perp|^2\right\}\int_\R\exp\left\{
-\frac{C_{q}}{4}|v||\eta_\parallel|\right\}d\eta_\parallel d\eta_\perp\\
\leq&C_{q,\ta}\langle v\rangle^{\varrho-2}\|h\|_{\infty}.
\end{split}
\end{equation*}
This finishes the proof of Lemma \ref{es.k}.
\end{proof}

The following lemma is concerned with the estimates on nonlinear operator $\Ga$.
\begin{lemma}\label{es.nop}
It holds that
\begin{equation}\label{Ga.lif}
\|\nu^{-1}w_{q,\ta,\vth}\Ga(f_1,f_2)\|_\infty\leq C\|w_{q,\ta,\vth}f_1\|_{\infty}\|w_{q,\ta,\vth}f_2\|_{\infty},
\end{equation}
\begin{equation}\label{Ga.lif1}
\|w_{q/2,\ta,\vth}\Ga(f_1,f_2)\|_\infty\leq C\left\{\|f_1\|_{\infty}\|w_{q,\ta,\vth}f_2\|_{\infty}+\|w_{q,\ta,\vth}f_1\|_{\infty}\|f_2\|_{\infty}\right\},
\end{equation}
\begin{equation}\label{Ga.l2}
\begin{split}
\|\nu^{-1/2}&w_{q/2,\ta,\vth}\Ga(f_1,f_2)\|^2_{2}\\ \leq & C\|w_{q,\ta,\vth}f_1\|^2_{\infty}\|w_{q/2,\ta,\vth}f_2\|^2_{\nu}+C\|w_{q,\ta,\vth}f_2\|^2_{\infty}\|w_{q/2,\ta,\vth}f_1\|^2_{\nu},
\end{split}
\end{equation}
and
\begin{equation}\label{Ga.l22}
\|\nu^{-1/2}\Ga(f_1,f_2)\|^2_{2}\leq C\|w_{q/2,\ta,\vth}f_1\|^2_{\infty}\|f_2\|^2_{\nu}
+C\|w_{q/2,\ta,\vth}f_2\|^2_{\infty}\|f_1\|^2_{\nu}.
\end{equation}

\end{lemma}
\begin{proof}
The proof of \eqref{Ga.lif} is the same as that of Lemma 5 in \cite[pp.730]{Guo-2010}, we omit the details for brevity.
To prove \eqref{Ga.lif1}, we rewrite
\begin{equation}\label{Ga.sp}
\begin{split}
\Ga(f_1,f_2)=&\frac{1}{\sqrt{\mu}}Q(\sqrt{\mu}f_1,\sqrt{\mu}f_2)\\
=&\int_{\R^3\times\S^2}|v-u|^{\varrho}
b_0(\cos \ta)\sqrt{\mu(u)}f_1(u')f_2(v')dud\om\\&-\int_{\R^3\times\S^2}|v-u|^{\varrho}
b_0(\cos \ta)\sqrt{\mu(u)}f_1(u)f_2(v)dud\om\\
=&\Ga^{\textrm{gain}}(f_1,f_2)-\Ga^{\textrm{loss}}(f_1,f_2).
\end{split}
\end{equation}
For the loss term, a simple calculation directly gives
\begin{equation*}
\begin{split}
\|w_{q/2,\ta,\vth}\Ga^{\textrm{loss}}(f_1,f_2)\|_{\infty}\leq C\|w_{q/2,\ta,\vth}f_2\nu\|_{\infty}\| f_1\|_{\infty}
\leq C\|w_{q/2,\ta,\vth}f_2\|_{\infty}\|f_1\|_{\infty}.
\end{split}
\end{equation*}
Next, since $|v|^2\leq |v'|^2+|u'|^2$, by virtue of \eqref{p.ine}, 
one has
$$
w_{q/2,\ta,\vth}(v)\leq w_{q/2,\ta,\vth}(u')w_{q/2,\ta,\vth}(v')\leq \frac{1}{2}(w_{q,\ta,\vth}(u')+w_{q,\ta,\vth}(v')).
$$
With this, we present the corresponding computation for the gain term as follows
\begin{equation*}
\begin{split}
|&w_{q/2,\ta,\vth}\Ga^{\textrm{gain}}(f_1,f_2)|\\
\leq& \frac{1}{2}(w_{q,\ta,\vth}(u')+w_{q,\ta,\vth}(v'))\int_{\R^3\times\S^2}|v-u|^{\varrho}
b_0(\cos \ta)\sqrt{\mu(u)}|f_1(u')f_2(v')|dud\om\\
\leq&C\left\{\|f_1\|_{\infty}\|w_{q,\ta,\vth}f_2\|_{\infty}+\|w_{q,\ta,\vth}f_1\|_{\infty}\|f_2\|_{\infty}\right\}.
\end{split}
\end{equation*}
This ends the proof for \eqref{Ga.lif1}.
In what follows, we only prove \eqref{Ga.l2}, since \eqref{Ga.l22} can be obtained in a similar fashion. Recall \eqref{Ga.sp},
for the loss term, one has
\begin{equation*}
\begin{split}
&\|\nu^{-1/2}w_{q/2,\ta,\vth}\Ga^{\textrm{loss}}(f_1,f_2)\|^2_{2}\\
=&\int_{\R^3\times\Omega}\nu^{-1}(v)w^2_{q/2,\ta,\vth}\left\{\int_{\R^3\times\S^2}|v-u|^{\varrho}
b_0(\cos \ta)\sqrt{\mu(u)}f_1(u)f_2(v)dud\om\right\}^2dvdx\\
\leq&\|f_1\|^2_{\infty}\int_{\R^3\times\Omega}\nu^{-1}(v)w^2_{q/2,\ta,\vth}|f_2(v)|^2\left\{\int_{\R^3\times\S^2}|v-u|^{\varrho}
b_0(\cos \ta)\sqrt{\mu(u)}d\om\right\}^2dvdx\\
\leq& C\|f_1\|^2_{\infty}\|w_{q/2,\ta,\vth}f_2\|^2_{\nu}.
\end{split}
\end{equation*}
As to the gain term, let us denote
\begin{equation*}
\begin{split}
\CI_0=&\|\nu^{-1/2}w_{q/2,\ta,\vth}\Ga^{\textrm{gain}}(f_1,f_2)\|^2_{2}\\=&\int_{\R^3\times\Omega}\nu^{-1}(v)
w^2_{q/2,\ta,\vth}\left\{\int_{\R^3\times\S^2}|v-u|^{\varrho}
b_0(\cos \ta)\sqrt{\mu(u)}f_1(u')f_2(v')du\om\right\}^2dvdx.
\end{split}
\end{equation*}
The calculation for $\CI_0$ is a little more delicate, we divide it into following three cases:

\noindent {\it Case 1, $|u|\geq|v|/2.$} In this case, $\mu^{1/2}(u)\leq\mu^{1/4}(u)\mu^{1/16}(v)$.
By H\"{o}lder's inequality and a change of variable $(u,v)\rightarrow (u',v')$, we have
\begin{equation*}
\begin{split}
\CI_0
\leq &C\int_{\R^3\times\Omega}\nu^{-1}(v)w^2_{q/2,\ta,\vth}(v)\int_{\R^3}|v-u|^{\varrho}
\sqrt{\mu(u)}f^2_1(u')f^2_2(v')du\\
\times&\int_{\R^3}|v-u|^{\varrho}
\sqrt{\mu(u)}dudvdx\\
\leq &C\int_{\R^3\times\R^3\times\Omega}w^2_{q/2,\ta,\vth}(u)w^2_{q/2,\ta,\vth}(v)|v-u|^{\varrho}
\mu^{1/16}(u)\mu^{1/16}(v)f^2_1(u)f^2_2(v)dudvdx\\
\leq& C\|w_{q/2,\ta,\vth}f_1\|^2_{\infty}\|w_{q/2,\ta,\vth}f_2\|^2_{\nu},
\end{split}
\end{equation*}
where we also used the fact that $\max\{|v|, |u|\}\leq |u'|+|v'|$.

\noindent {\it Case 2, $|u|\leq|v|/2$ and $|v|\leq1$.} In this situation, $|u-v|\geq |v|-|u|\geq |v|/2$ and $|u|\leq 1/2,$ moreover
$|u'|+|v'|\leq 2(|u|+|v|)\leq 3|v|\leq 3,$ consequently, when $(u,v)\in\{(u,v)||u|\geq|v|/2,|v|\leq1\}$,
we have by H\"{o}lder's inequality and a change of variable $(u,v)\rightarrow (u',v')$ that
\begin{equation*}
\begin{split}
\CI_0
\leq &C\int_{\{|v|\leq1\}\times\Omega}|v|^{\varrho}\int_{\{|u|\leq1/2\}}|v-u|^{\varrho}
\mu(u)f^2_1(u')f^2_2(v')dudvdx\\
\leq &C\int_{\{|v|\leq 1, |u|\leq1/2\}\times\Omega}\min\{|u'|^\varrho,|v'|^\varrho\}f^2_1(u')f^2_2(v')dudvdx\\
\leq& C\int_{\{|v|\leq 3, |u|\leq 3\}\times\Omega}\min\{|u|^\varrho,|v|^\varrho\}f^2_1(u)f^2_2(v)dudvdx\leq C\|f_1\|^2_{\infty}\|f_2\|^2_{\nu}.
\end{split}
\end{equation*}
\noindent {\it Case 3, $|u|\leq|v|/2$ and $|v|\geq1$.} One has $\max\{|u'|,|v'|\}\leq 5|v|/2$
on this occasion, hence $\nu(v)\lesssim \nu(v')+\nu(u')$, moreover, it follows $|u-v|\geq|v|-|u|\geq|v|/2\geq 1/2$.
Notice that $w^2_{q/2,\ta,\vth}(v)\leq w^2_{q/2,\ta,\vth}(u')w^2_{q/2,\ta,\vth}(v')$,
apply H\"{o}lder's inequality and a change of variable $(u,v)\rightarrow (u',v')$ again to obtain
\begin{equation*}
\begin{split}
\CI_0
\leq &\int_{\R^3\times\R^3\times\Omega}\nu^{-1}(v)w^2_{q/2,\ta,\vth}(1+|v|)^{2\vho}f^2_1(u')f^2_2(v')dudvdx\\
\leq &C\int_{\R^3\times\R^3\times\Omega}w^2_{q/2,\ta,\vth}(u)w^2_{q/2,\ta,\vth}(v)(\nu(v)+\nu(u))f^2_1(u)f^2_2(v)dudvdx\\
\leq& C\|w_{q,\ta,\vth}f_1\|^2_{\infty}\|w_{q/2,\ta,\vth}f_2\|^2_{\nu}+C\|w_{q,\ta,\vth}f_2\|^2_{\infty}\|w_{q/2,\ta,\vth}f_1\|^2_{\nu},
\end{split}
\end{equation*}
where the fact $\int_{\R^3}w_{-q/2,\ta,\vth} dv<+\infty$ was used.

Combing all the estimates above, we see that \eqref{Ga.l2} holds true, this ends the proof of Lemma \ref{es.nop}.

Next, we address the following Ukai's trace theorem whose proof can be found in Lemma 2.1 of \cite[pp. 187]{EGKM-13}.
\begin{lemma}
\label{ukai}Let $\vps>0$, define the near-grazing set of $\gamma_+$ or $\gamma_-$ as
\begin{equation*}
\gamma ^{\varepsilon }_{\pm}\ \equiv \ \left\{(x,v)\in \gamma_{\pm} : |n(x)\cdot
v|\leq \varepsilon \ \text{or} \ |v|\geq \frac{1}{\varepsilon } \ \text{or}
\ |v|\leq \varepsilon\right\}.
\end{equation*}%
There exists constant $C_{\varepsilon ,\Omega }>0$ depends only on $\vps$ and $\Omega$ such that
\begin{equation}
\begin{split}
\int_{s}^{t}|f\mathbf{1}_{\gamma_+\backslash\gamma ^{\varepsilon }_+ }(\tau
)|_{1}d\tau &\leq  C_{\varepsilon ,\Omega } \left\{|| f(s) ||_1 +\int_{s}^{t}\Big[%
\Vert f(\tau )\Vert _{1}+\Vert \{\partial _{t}+v\cdot \nabla _{x}\}f(\tau
)\Vert _{1}\Big]d\tau\right\},
\end{split}\notag
\end{equation}
for any $0\leq s\leq t.$
\end{lemma}

\end{proof}

The following lemma quoted from \cite[pp.723]{Guo-2010} is concerned with property of the kinetic distance function.
\begin{lemma}
\label{velocity}Let $\Omega $ be strictly convex defined in \eqref{scon}. Define the functional along the trajectories $\frac{dX(s)}{ds}=V(s),\frac{%
dV(s)}{ds}=0$ in \eqref{ch.eq} as:
\begin{equation}
\alpha (s)\equiv \xi ^{2}(X(s))+[V(s)\cdot \nabla \xi
(X(s))]^{2}-2\{V(s)\cdot \nabla ^{2}\xi (X(s))\cdot V(s)\}\xi (X(s)).
\label{alpha}
\end{equation}%
Let $X(s)\in \bar{\Omega}$ for $t_{1}\leq s\leq t_{2}$. Then there exists
constant $C_{\xi }>0$ such that
\begin{eqnarray*}
e^{C_{\xi }(|V(t_{1})|+1)t_{1}}\alpha (t_{1}) &\leq &e^{C_{\xi
}(|V(t_{1})|+1)t_{2}}\alpha (t_{2}),  \label{velocitybound} \\
e^{-C_{\xi }(|V(t_{1})|+1)t_{1}}\alpha (t_{1}) &\geq &e^{-C_{\xi
}(|V(t_{1})|+1)t_{2}}\alpha (t_{2}).  \notag
\end{eqnarray*}
\end{lemma}
Finally, we state the following significant lemma which gives a lower bound of the backward exist time $t_{\mathbf{b}}(x,v)$.
\begin{lemma}\label{huang}\cite[pp.724]{Guo-2010}
Let $x_{i}\in \partial \Omega ,$ for $i=1,2,$ and let $(t_{1},x_{1},v)$
and $(t_{2},x_{2},v)$ be connected with the trajectory $\frac{dX(s)}{ds}%
=V(s),\frac{dV(s)}{ds}=0$ which lies inside $\bar{\Omega}$. Then there
exists a constant $C_{\xi }>0$ such that
\begin{equation}
|t_{1}-t_{2}|\geq \frac{|n(x_{1})\cdot v|}{C_{\xi }|v|^{2}}.  \label{tlower}
\end{equation}

\end{lemma}
\section{Diffuse reflection boundary value problem}\label{DVP}

\subsection{$L^2$ existence and decay for the linearized equation}\label{L2th}
As mentioned in Section \ref{In}, we mainly employ the $L^2\cap L^\infty$ argument to solve the initial boundary value problem of \eqref{BE}, \eqref{ID} and \eqref{DBD}. To obtain the time decay rates in $L^\infty$ space, an $L^2-$ time decay theory must be established at first cf. \cite{Guo-2010}.
However, one can not directly obtain the time decay of \eqref{BE}, \eqref{ID} and \eqref{DBD} by an $L^2-$ energy method, since the positive operator $L$ is degenerated in the sense that the inner product $(Lf,f)$ has no positive lower bound in the large velocity domain. To overcome this difficulty, we first construct the global existence in some weighted $L^2$ space, then tend to deduce the time decay rates in lower order weighed energy space via an interpolation technique. We remark that the main idea used here is similar as treating the Cauchy problem of Boltzmann equation with soft potential \cite{SG-06, SG-08}. And it should be pointed out that it is necessary to derive
the $L^2$ time decay rates even only considering the global existence of the initial boundary value problem of \eqref{BE}, \eqref{ID} and \eqref{DBD} in the case of soft potential.

Notice that the null space of the linear operator $L$ is generated by $\{1,v,\frac{1}{2}(|v^2|-3)\}\sqrt{\mu}$, we define
$$
\FP f=\left\{a+b\cdot v+\frac{1}{2}(|v^2|-3)c\right\}\sqrt{\mu}, \ (t,x,v)\in[0,+\infty)\times\Omega\times\R^3,
$$
which is called the macroscopic part of $f$. The microscopic part of $f$ is further denoted by $\{\FI-\FP\}f=f-\FP f$.
It is well-known that  there exists $\de_0>0$ such that
$$
(Lf,f)\geq\de_0\|\{\FI-\FP\}f\|^2_\nu.
$$
We consider the following initial boundary value problem of the linearized Boltzmann equation with soft potential
\begin{equation}
\partial _{t}f+v\cdot \nabla _{x}f+Lf=g,\text{ \ \ }f(0)=f_{0},\quad \text{
in }(0,+\infty)\times \Omega \times \R^{3},  \label{dlinear}
\end{equation}%
with
\begin{equation}\label{lbd}
f_{-}=P_{\gamma }f,\ \ \textrm{on}\  \R_{+}\times\gamma _{-},
\end{equation}
and $g$ is given.

In that follows in this subsection we will prove the following
\begin{proposition}\label{l2-lqn}
Let $-3<\varrho<0$ and $(q,\ta)\in\CA_{q,\ta}$.
Assume that for all $t>0$
\begin{equation}
\int_{\Omega \times \R^{3}}g(t,x,v)\sqrt{\mu }dvdx=0, \ \ \FP g=0. \label{dlinearcondition}
\end{equation}%
There exists $\varepsilon_0>$ such that if
$$\|w_{q/2,\ta}f_0\|^2_2+|f_0|_{2,+}^2+\int_{0}^t\left\|\nu^{-1/2}w_{q/2,\ta}g(s)\right\|_2^2ds\leq \varepsilon^2_0,$$
then there exists a unique solution to the problem
\eqref{dlinear} and \eqref{lbd} such that
for all $t\geq \>0$,
\begin{equation}
\int_{\Omega \times \R^{3}}f(t,x,v)\sqrt{\mu }dxdv=0,
\label{dlinearcondition1}
\end{equation}
\begin{equation*}
\begin{split}
\sup\limits_{0\leq s\leq t}&\|f(s)\|^2_2+\int_{0}^t\|f(s)\|^2_{\nu}ds
\leq C\|f_0\|^2_2+C\int_{0}^t\left\|\nu^{-1/2}g(s)\right\|_2^2ds,
\end{split}
\end{equation*}%
and
\begin{equation}\label{wl2}
\begin{split}
\sup\limits_{0\leq s\leq t}&\|w_{q/2,\ta}f(s)\|^2_2+\int_{0}^t\|w_{q/2,\ta}f(s)\|^2_{\nu}ds\\
\leq& C\|w_{q/2,\ta}f_0\|^2_2+C\int_{0}^t\left\|\nu^{-1/2}w_{q/2,\ta}g(s)\right\|_2^2ds.
\end{split}
\end{equation}%
Moreover, let $\rho_0=\frac{\ta}{\ta-\varrho}$, there exists $\lambda_1 >0$ depends on $q$ and $\rho_0$ such that
\begin{equation}\label{l-decay}
\begin{split}
\Vert& f(t)\Vert _{2}^{2}+e^{-\lambda_1 t^{\rho_0}}\int_0^te^{\la_1s^{\rho_0}}\|f(s)\|_2^2ds\\
\lesssim& e^{-\lambda_1 t^{\rho_0}}\left\{\|w_{q/2,\ta}f_0\|^2_2+\int_{0}^{t}e^{\lambda_1 s^{\rho_0}}\Vert \nu^{-1/2}g(s)\Vert
_{2}^{2}ds+\int_{0}^{t}\Vert\nu^{-1/2} w_{q/2,\ta}g(s)\Vert
_{2}^{2}ds\right\}.
\end{split}
\end{equation}
\end{proposition}
In order to construct the global existence of \eqref{dlinear} and \eqref{lbd}, we first deduce the global solvability of the equation \eqref{dlinear}
with an approximation boundary condition and then we show that such an approximate solution sequence converges in $L^2$
for any $t\geq0$.
Once the global existence is obtained, the time-decay estimate \eqref{l-decay} follows from an $L^2$ energy estimate and an interpolation technique. Along this line,
Proposition \ref{l2-lqn} is a easy consequence of the following two lemmas, the first one is concerned with
 {\it a priori} estimates for
the macroscopic part of the solution of \eqref{dlinear} and \eqref{lbd}.
\begin{lemma}\label{dabc}
Assume that $g$ satisfies \eqref{dlinearcondition} and $f$
satisfies \eqref{dlinear}, \eqref{lbd} and \eqref{dlinearcondition1}. Then there exists
a function $G(t)$ such that, for all $t\geq0$, $G(t)\lesssim
\|f(t)\|_{2}^{2}$ and
\begin{equation}\label{mm}
\begin{split}
\|\mathbf{P}f\|_{\nu }^{2} \lesssim
\frac{d}{dt}G(t)+\left\|g\right\|_{2}^{2}+\|\{\mathbf{I}-\mathbf{P}\}f\|_{\nu }^{2}+|\{1-P_{\gamma }\}f|_{2,+}^{2}.
\end{split}
\end{equation}
\end{lemma}
\begin{proof}
The proof of Lemma \ref{dabc} is much similar as that of Lemma 6.1 in \cite[pp.221]{EGKM-13}, we omit the details for brevity.
\end{proof}

\begin{lemma}\label{lg.ex}
Assume $g$ satisfies \eqref{dlinearcondition}.
There is a constant $\varepsilon_0>0$ such that for any $t>0$ if
$$\Vert f_{0}\Vert_2^{2}+|f_0|_{2,+}^2+\int_{0}^t\left\|\nu^{-1/2}g(s)\right\|_2^2ds\leq \varepsilon^2_0,$$
then \eqref{dlinear} and \eqref{lbd} admits a strong solution $f(t,x,v)$ in $[0,+\infty)\times\Omega\times\R^3$ satisfying
\begin{equation}\label{g.nw}
\begin{split}
\Vert f(t)\Vert _{2}^{2}&+\int_{0}^{t}\Vert
f(s)\Vert _{\nu }^{2}ds+\int_{0}^{t}|(I-P_{\gamma
})f(s)|_{2,+}^{2}ds\\
\leq& C\int_{0}^{t}\Vert \nu^{-1/2}g(s)\Vert_2^{2}ds+C\Vert
f_{0}\Vert _{2}^{2}.
\end{split}
\end{equation}
\end{lemma}
\begin{proof}
We establish a solution of \eqref{dlinear} and \eqref{lbd} via the following approximate boundary value problem
\begin{equation}
\partial _{t}f+v\cdot \nabla _{x}f+Lf=g,\text{ \ \ }f(0,x,v)=f_{0}, \label{j.eq}
\end{equation}%
with
\begin{equation}\label{j.bd}
f_{-}=(1-\frac{1}{j})P_{\gamma }f,\ j=2,3,\cdots.
\end{equation}
The proof is then divided into two steps.

\vskip.3cm \noindent \textit{Step 1.}\textit{Global existence of} \eqref{j.eq} \textit{and} \eqref{j.bd}.
We start with constructing the local existence of \eqref{j.eq} and \eqref{j.bd} through the
following sequence of iterating approximate solutions:
\begin{equation}
\partial _{t}f^{\ell +1}+v\cdot \nabla _{x}f^{\ell +1}+\nu f^{\ell
+1}-Kf^{\ell }=g,\text{ \ \ }f^{\ell+1}(0)=f_{0}, \ \ \ell\geq0,\label{daproximatenew}
\end{equation}%
with
\begin{equation}\label{fl.bd}
f_{-}^{\ell +1}=(1-\frac{1}{j})P_{\gamma }f^{\ell },\ j=2,3,\cdots,
\end{equation}
and $f^{0}\equiv f_{0}$.
Let us now define
$$
M(f)(t)=\Vert f(t)\Vert _{2}^{2}+\int_{0}^{t}|f(s)|_{2,+}^{2}ds.
$$
We {\it claim} that there exists a small $T_\ast>0$ such that if $\sum\limits_{0\leq t\leq T\ast}M(f^{\ell})(t)\leq M_1$ for $M_1>0$  then $\sum\limits_{0\leq t\leq T\ast}M(f^{\ell+1})(t)\leq M_1.$
Take an inner product
of (\ref{daproximatenew}) with $f^{\ell +1}$ and use Green's identity as well as Lemma \ref{es.k}, to deduce
\begin{equation}\label{fll2}
\begin{split}
\Vert &f^{\ell+1}(t)\Vert _{2}^{2}+(1-\varepsilon )\int_{0}^{t}\Vert
f^{\ell +1}(s)\Vert _{\nu }^{2}+\int_{0}^{t}|f^{\ell +1}(s)|_{2,+}^{2}ds
\\
\leq &
(1-\frac{1}{j})^{2} \int_{0}^{t}|P_{\gamma }f^{\ell
}|_{2,-}^{2}ds+C_{\varepsilon }\int_{0}^{t}\Vert f^{\ell}(s)\Vert _{\nu}ds+\int_{0}^{t}\Vert\nu^{-1/2} g(s)\Vert_2
^{2}ds+\Vert f_{0}\Vert _{2}^{2}.
\end{split}
\end{equation}%
Since
$$\Vert f_{0}\Vert_2^{2}+|f_0|_{2,+}^2+\int_{0}^{t}\Vert\nu^{-1/2} g(s)\Vert_2^{2}ds<\vps_0,\ \ |P_{\gamma }f^{\ell
}|_{2,-}^{2}\leq |f^{\ell
}|_{2,+}^{2},$$
and
$$\nu(v)\sim (1+|v|^2)^{\varrho/2}, -3<\varrho<0,$$
we see that
$$
M(f^{\ell+1})(t)\leq \max\{1,C_\varepsilon\} tM_1+\varepsilon_0.
$$
Taking $T_\ast>0$ suitably small and letting $\vps_0<M_1$, one obtains $\sum\limits_{0\leq t\leq T\ast}M(f^{\ell+1})(t)\leq M_1$.
This completes the proof of the {\it claim}.

Next we get from the difference of the equation \eqref{daproximatenew} for $\ell+1$ and $\ell$ that
\begin{equation*}
\partial _{t}[f^{\ell +1}-f^{\ell }]+v\cdot \nabla _{x}[f^{\ell +1}-f^{\ell
}]+\nu \lbrack f^{\ell +1}-f^{\ell }]=K[f^{\ell }-f^{\ell -1}], \ \ell\geq1,
\end{equation*}%
with \ $[f^{\ell +1}-f^{\ell }](0)\equiv 0$ and $f_{-}^{\ell +1}-f_{-}^{\ell
}=(1-\frac{1}{j})P_{\gamma }[f^{\ell }-f^{\ell -1}]$. Performing the similar calculations as for obtaining \eqref{fll2},
one has
\begin{equation*}
\begin{split}
\Vert f^{\ell +1}(t)&-f^{\ell }(t)\Vert _{2}+\int_{0}^{t}\Vert f^{\ell
+1}(s)-f^{\ell}(s)\Vert _{\nu }^{2}+\int_{0}^{t}|f^{\ell +1}(s)-f^{\ell
}(s)|_{2,+}^{2}ds \\
&\leq (1-\frac{1}{j})^2\int_{0}^{t}|f^{\ell}-f^{\ell-1}|_{2,+}^{2}+C_\vps\int_{0}^{t}\Vert
f^{\ell}(s)-f^{\ell-1}(s)\Vert _{\nu }^{2}ds,
\end{split}
\end{equation*}%
from which, we obtain
\begin{equation*}
\begin{split}
&\sup\limits_{0\leq t\leq T_\ast}\Vert f^{\ell +1}(t)-f^{\ell }(t)\Vert _{2}+\int_{0}^{T_\ast}|f^{\ell +1}(s)-f^{\ell
}(s)|_{2,+}^{2}ds \\
&\leq \max\left\{(1-\frac{1}{j})^2,T_\ast C_\vps\right\}\left\{\sup\limits_{0\leq t\leq T_\ast}\Vert
f^{\ell}(t)-f^{\ell-1}(t)\Vert_2^{2}+\int_{0}^{T_\ast}|f^{\ell}-f^{\ell-1}|_{2,+}^{2}\right\},
\end{split}
\end{equation*}
for $\ell\geq1$.
Thus, if $T_\ast C_\vps<1$, we also show that $f^{\ell }(t)$ is a Cauchy sequence in $L^2$ for $t\in[0,T_\ast]$ and $j\geq2$. That is $f^{\ell }\rightarrow f^{j}$
and $f^{j}$ is a strong solution of
\begin{equation}
\partial _{t}f+v\cdot \nabla _{x}f+Lf=g,\text{ \ \ }f(0)=f_{0},\text{ \ \ \
\ \ }f_{-}=(1-\frac{1}{j})P_{\gamma }f.  \label{dlinearj}
\end{equation}
Furthermore, for any given $j\geq2$, assume $f^j$ is a strong solution of \eqref{daproximatenew} and \eqref{fl.bd}, by using Green's identity and $\FP g=0$, one obtains the following the {\it a priori} estimate:
\begin{equation}\label{ape1}
\begin{split}
\Vert f^{j}(t)\Vert _{2}^{2}&+\la\int_{0}^{t}\Vert (\mathbf{I}-\mathbf{P}%
)f^{j}(s)\Vert _{\nu }^{2}ds+\int_{0}^{t}|(1-P_{\gamma
})f^{j}(s)|_{2,+}^{2}ds\\& +\left(\frac{2}{j}-\frac{1}{j^2}\right)\int_{0}^{t}|P_{\gamma
}f^{j}(s)|_{2,+}^{2}ds
\leq \int_{0}^{t}\Vert \nu^{-1/2}g(s)\Vert_2^{2}ds+\Vert
f_{0}\Vert _{2}^{2}.
\end{split}
\end{equation}
Then the global existence of \eqref{j.eq} and \eqref{j.bd} follows from the standard continuation argument.

\vskip.3cm \noindent \textit{Step 2.} \textit{For any} $t>0$, $\{f^j\}_{j=2}^{+\infty}$ \textit{is convergent} \textit{in} $L^2$.
Notice that $f^j$ enjoys the bound \eqref{ape1},
by taking a weak limit, we obtain a weak solution $f$ to (\ref{dlinear}) and \eqref{lbd}.
Taking difference, we further have
\begin{equation}\label{fj-f.eq}
\partial _{t}[f^{j}-f]+v\cdot \nabla _{x}[f^{j}-f]+L[f^{j}-f]=0,\text{ \ \ \
\ }[f^{j}-f]_{-}=P_{\gamma }[f^{j}-f]+\frac{1}{j}P_{\gamma }f^{j},
\end{equation}%
with $[f^{j}-f](0)=0$. Utilizing a standard $L^2$ energy estimates as for deriving \eqref{ape1}, we obtain for $\eta>0$
\begin{equation}\label{fj-f}
\begin{split}
\Vert f^{j}(t)-f(t)\Vert _{2}^{2}&+\int_{0}^{t}\Vert\{\FI-\FP\} [f^{j}(s)-f(s)]\Vert _{\nu
}^{2}ds\\&+\int_{0}^{t}|\{I-P_\ga\}[f^{j}(s)-f(s)]|_{2,+}^{2}ds
\\ \lesssim& \eta \int_{0}^{t}|P_\ga[f^{j}(s)-f(s)]|_{2,+}^{2}ds+
\frac{C_\eta}{j^2}
\int_{0}^{t}|P_{\gamma }f^{j}|_{2,+}^{2}ds.
\end{split}
\end{equation}%
Since $\left(\frac{2}{j}-\frac{1}{j^2}\right)\int_{0}^{t}|P_{\gamma
}f^{j}(s)|_{2,+}^{2}ds$ is bounded by \eqref{ape1}, one can see that
$$
\frac{C_\eta}{j^2}
\int_{0}^{t}|P_{\gamma }f^{j}|_{2,+}^{2}ds\rightarrow0,\ \textrm{as}\ j\rightarrow\infty.
$$
To handle the small term $\eta \int_{0}^{t}|P_\ga[f^{j}(s)-f(s)]|_{2,+}^{2}ds$, we resort to Ukai's trace theorem. Recalling the boundary norm
\begin{equation}
\int_{0}^{t}|P_{\gamma }[f^{j}-f](s)|_{2,\pm }^{2} =\int_{0}^{t}\int_{\gamma
_{\pm }}\left[ \int_{\{u:n\cdot u>0\}}[f^{j}-f](s,x,u)\sqrt{\mu }\{n\cdot u\}du%
\right] ^{2}\mu (v) d\gamma ds.
\notag
\end{equation}%
Now we split the domain of inner integration as
\begin{equation}\nonumber
\begin{split}
\{u\in \R^{3}:n(x)\cdot u>0\}& =\{u\in \R^{3}:0<n(x)\cdot
u<\varepsilon \ \text{or}\ |u|>1/\varepsilon \ \text{or}\ |u|<\varepsilon \}
\\
& \ \ \ \cup \{u\in \R^{3}:\varepsilon \leq n(x)\cdot u\ \text{and}\
|u|\leq 1/\varepsilon \ \text{and}\ |u|\geq \varepsilon \}.
\end{split}
\end{equation}%
The first set's contribution(grazing part) of $\int_{0}^{t}|P_{\gamma
}f^{j}(s)|_{2,\pm }^{2}ds$ is bounded by the H\"{o}lder inequality,
\begin{equation}
\begin{split}
 C\Bigg(\int_{\/_{\substack{ 0<n\cdot u<\varepsilon  \\ \text{or}%
|u|>1/\varepsilon  \\ \text{or}|u|<\varepsilon }}}\mu (u)\{n\cdot u\}du\Bigg)&
\int_{0}^{t}\int_{\partial \Omega }\int_{\{u:n\cdot
u>0\}}|[f^{j}-f](s)|^{2}\{n\cdot u\}dS_{x}duds \\
& \lesssim \ \varepsilon\int_{0}^{t}\int_{\gamma
_{+}}|[f^{j}-f](s)|^{2}d\gamma ds.
\end{split}
\label{firstgamma-}
\end{equation}%
For the second term, we use Lemma \ref{ukai} and \eqref{fj-f.eq} to bound
the second set's contribution(non-grazing part) of $\int_{0}^{t}|P_{\gamma
}[f^{j}-f](s)|_{2,\pm }^{2}ds$ by
\begin{equation}
\begin{split}
C\int_{0}^{t}&|[f^{j}-f](s)\mathbf{1}_{\gamma _{+}\backslash \gamma
_{+}^{\varepsilon }}|_{2}^{2}ds\\
 \lesssim C&\int_{0}^{t}\Vert
[f^{j}-f](s)\Vert _{2}^{2}ds+C\int_{0}^{t}\Vert \partial _{t}[f^{j}-f]^{2}+v\cdot
\nabla _{x}[f^{j}-f]^{2}\Vert _{1}ds \\
 \lesssim C&\int_{0}^{t}\Vert
[f^{j}-f](s)\Vert _{2}^{2}ds+C\int_{0}^{t}|(L[f^{j}-f],[f^{j}-f])|ds\\
 \lesssim C&\int_{0}^{t}\Vert
[f^{j}-f](s)\Vert _{2}^{2}ds+C\int_{0}^{t}\Vert \{\mathbf{I}-\mathbf{P}\}[f^{j}-f](s)\Vert _{\nu }^{2}ds.
\end{split}
\label{dukai}
\end{equation}%
From \eqref{firstgamma-} and \eqref{dukai}, we have on the one hand
\begin{equation}
\begin{split}
\int_{0}^{t}|P_{\gamma }[f^{j}-f](s)|_{2,\pm }^{2}ds\leq& \varepsilon
\int_{0}^{t}\int_{\gamma _{+}}|[f^{j}-f](s)|^{2}d\gamma ds
\\&+C_{\varepsilon}\left\{\int_{0}^{t}\Big[
\|[f^{j}-f](s)\|_{2}^{2}+\|(\mathbf{I}-\mathbf{P})[f^{j}-f](s)\|_{\nu
}^{2}\Big]ds\right\} .
\end{split}
\label{gamma-}
\end{equation}%
On the other hand,
we get by integrating (\ref{fj-f}) from $0$ to $t$
\begin{equation}
\int_{0}^{t}||\mathbf{P}[f^{j}-f](s)||_{2}^{2}ds\lesssim \eta C_t\int_{0}^{t}|P_\ga[f^{j}(s)-f(s)]|_{2,+}^{2}ds+
\frac{C_tC_\eta}{j^2}
\int_{0}^{t}|P_{\gamma }f^{j}|_{2,+}^{2}ds.
\label{Pf}
\end{equation}
Letting
$\varepsilon >0$ and $\eta>0$ suitably small and taking a appropriate linear combination of \eqref{fj-f}, \eqref{gamma-} and \eqref{Pf}, we improve \eqref{fj-f} as
\begin{equation*}
\Vert f^{j}(t)-f(t)\Vert _{2}^{2}+\int_{0}^{t}\Vert f^{j}(s)-f(s)\Vert _{\nu
}^{2}ds+\int_{0}^{t}|f^{j}(s)-f(s)|_{2}^{2}ds\lesssim \frac{C_{t}}{j^2}%
\int_{0}^{t}|P_{\gamma }f^{j}|^{2}ds \rightarrow 0,
\end{equation*}%
which implies $f^j\rightarrow f$ strongly in $L^2$ for any given $t\geq0$. Moreover we can also show that such a solution is unique by a similar $L^2$ energy estimates used above.
As a consequent, we construct $f(t,x,v)$ as an $L^{2}$ strong solution to (\ref{dlinear}) and \eqref{lbd} for any $t\geq0$.
Finally, by taking the inner product of \eqref{dlinear} with $f$ over $\Om\times\R^3$ and applying Green's identity again, one has
\begin{equation}\label{en.1}
\begin{split}
\frac{d}{dt}\Vert f\Vert _{2}^{2}+\la\Vert\{\FI-\FP\}
f\Vert _{\nu }^{2}+|(I-P_{\gamma
})f|_{2,+}^{2}
\leq \Vert \nu^{-1/2}g\Vert^{2}.
\end{split}
\end{equation}
Letting $0<\ka_1\ll 1$, taking the summation of \eqref{en.1} and $\ka_1\times\eqref{mm}$, we obtain
\begin{equation}\label{en.2}
\begin{split}
\frac{d}{dt}\left\{\Vert f\Vert _{2}^{2}-\ka_1G(t)\right\}+\la\Vert
f\Vert _{\nu }^{2}+\la|(I-P_{\gamma
})f|_{2,+}^{2}
\leq \Vert \nu^{-1/2}g\Vert^{2}.
\end{split}
\end{equation}
\eqref{g.nw} follows from \eqref{en.2}.
 This completes the proof of Lemma \ref{lg.ex}.
\end{proof}

With Lemmas \ref{lg.ex} and \ref{dabc} in hand, we now turn to complete

\begin{proof}[The proof of Proposition \ref{l2-lqn}]
Let $h=w_{q/2,\ta}f$, 
then \eqref{dlinear} and \eqref{lbd} is equivalent to
\begin{equation}\label{h.eqn}
\pa_th+v\cdot \na h+\nu h-w_{q/2,\ta} K\left(\frac{h}{w_{q/2,\ta}}\right)=w_{q/2,\ta}g,
h(0,x,v)=h_0(x,v)=w_{q/2,\ta}f_0(x,v),
\end{equation}
and
\begin{equation}\label{h.bd}
h_-=w_{q/2,\ta}\sqrt{\mu}\int_{\mathcal {V}(x)}h(t,x,v')\frac{1}{w_{q/2,\ta}(v')\sqrt{\mu}(v')}d\si\eqdef P_\ga^w h,
\end{equation}
where
$$
\mathcal {V}(x)=\{v'\in \R^3, v'\cdot n(x)>0\},\ \ d\si=\mu(v')|n(x)\cdot v'|dv'.$$
Proceeding similarly to obtain the global existence of \eqref{dlinear} and \eqref{lbd}, one can show that \eqref{h.eqn}
and \eqref{h.bd} possesses a unique solution $h(t,x,v)$. We now turn to prove \eqref{wl2} and \eqref{l-decay}.
Taking the inner product of \eqref{h.eqn} with $h$ over $\Om\times\R^3$ and applying Lemma \ref{es.k}, one has
\begin{equation}\label{h.inner1}
\frac{d}{dt}\|h\|_2^2+|\{I-P_\ga^w\}h|_{2,+}^2+\|h\|^2_\nu\leq \eta\|h\|^2_\nu+C_\eta\|f\|_\nu^2+C\|\nu^{-1/2}w_{q/2,\ta}g\|_2^2,
\end{equation}
where we have also used the fact that $|P_\ga^wh|_{2,+}^2=|P_\ga^wh|_{2,-}^2.$
Integrating \eqref{h.inner1} with respect to the time variable over $[0,t]$ and combing \eqref{g.nw}, we obtain
\begin{equation*}\label{h.inner2}
\begin{split}
\|h(t)\|_2^2&+\int_0^t|\{I-P_\ga^w\}h|_{2,+}^2dt+\int_0^t\|h\|^2_\nu ds
+\Vert f(t)\Vert _{2}^{2}\\&+\int_{0}^{t}\Vert
f(s)\Vert _{\nu }^{2}ds+\int_{0}^{t}|(I-P_{\gamma})f(s)|_{2,+}^{2}ds\\
\leq& C\Vert
w_{q/2,\ta}f_{0}\Vert _{2}^{2}+C\int_0^t\|\nu^{-1/2}w_{q/2,\ta}g(s)\|_2^2ds,
\end{split}
\end{equation*}
which implies \eqref{wl2}. It remains now to prove the time decay \eqref{l-decay}.
Take constants $\la_1>0$ and $0<\rho_0<1$ whose specific values will be determined later on,
multiply $e^{\la_1t^{\rho_0}}$ to \eqref{en.2} and integrate the resulting inequality with respect to time variable over $[0,t]$ to obtain
\begin{equation}\label{en.3}
\begin{split}
e^{\la_1t^\rho}\Vert f\Vert _{2}^{2}&+\int_0^te^{\la_1s^\rho}\Vert
f\Vert_{\nu }^{2}ds+\int_0^t e^{\la_1s^{\rho_0}}|(I-P_{\gamma
})f|_{2,+}^{2}ds\\
\leq& C\|f_0\|_2^2+C\la_1{\rho_0} \int_0^t s^{\rho-1}e^{\la_1s^{\rho_0}}\Vert f\Vert_{2}^{2}ds+C \int_0^te^{\la_1s^{\rho_0}}\Vert \nu^{-1/2}g\Vert_2^{2}ds.
\end{split}
\end{equation}
To take care of the delicate term $s^{{\rho_0}-1}e^{\la_1s^{\rho_0}}\Vert f\Vert_{2}^{2}$, we decompose the $v$ integration domain into
two parts:
$$
E:\{v|s^{\rho-1}\leq \ka_0(1+|v|^2)^{\varrho/2}\},\ \ E^c:\{v|s^{{\rho_0}-1}\geq \ka_0(1+|v|^2)^{\varrho/2}\},
$$
where $\ka_0>0$ and small enough. On $E$, it is straightforward to see that
\begin{equation}\label{E.es}
s^{\rho_0-1}e^{\la_1s^{\rho_0}}\Vert f{\bf 1}_E\Vert_{2}^{2}\leq C_\varrho\ka_0e^{\la_1s^{\rho_0}}\Vert
f\Vert_{\nu }^{2},
\end{equation}
here
$$
{\bf 1}_E=\left\{\begin{array}{rll}
1,&\ \ v\in E,\\
0,&\ \ v\notin E.\end{array}\right.
$$
and $C_\varrho$ is determined by \eqref{confre}.
While on $E^c$, notice that $0<\rho_0<1$, one obtains
$$
2\la_1 s^{\rho_0}\leq 2\la_1 \ka_0^{\frac{\rho_0}{\rho_0-1}}(1+|v|^2)^{\frac{\varrho\rho_0}{2(\rho_0-1)}}.
$$
With this, we further have by letting $\la_1=\frac{q}{8} \ka_0^{\frac{\rho_0}{1-\rho_0}}$ and $\rho_0=\frac{\ta}{\ta-\varrho}$
\begin{equation}\label{Ec.es}
\begin{split}
\int_0^t s^{\rho_0-1}e^{\la_1s^{\rho_0}}\Vert f{\bf 1}_{E^c}\Vert_{2}^{2}ds=&\int_0^t s^{\rho_0-1}e^{-\la_1s^{\rho_0}}e^{2\la_1s^{\rho_0}}\Vert f{\bf 1}_{E^c}\Vert_{2}^{2}ds\\
\leq& C_\ta\int_0^t s^{\rho_0-1}e^{-\la_1s^{\rho_0}}\Vert w_{q/2,\ta}f\Vert_{2}^{2}ds\\
\leq& C\Vert w_{q/2,\ta}f_0\Vert_{2}^{2}
+C\int_0^t\|w_{q/2,\ta}\nu^{-1/2}g(s)\|_2^2ds,
\end{split}
\end{equation}
here we have used \eqref{wl2} to derive the last inequality.
Plugging \eqref{E.es} and \eqref{Ec.es} into \eqref{en.3} and dividing the resulting inequality by $e^{\la_1t^{\rho_0}}$, we then show that
\eqref{l-decay} holds true.
Thereby concluding the proof of
Proposition \ref{l2-lqn}.

\end{proof}

\subsection{$L^\infty$ existence for the linearized equation}\label{lifth}
In this subsection, we still consider the following initial boundary value problem
\begin{equation}
\partial_{t}f+v\cdot \nabla _{x}f+Lf=g,\ f(0)=f_{0},\ \text{
in }(0,+\infty)\times \Omega \times \R^{3},  \label{lineq}
\end{equation}%
with
\begin{equation}\label{linbd}
f_{-}=P_{\gamma }f,\ \ \textrm{on}\  \R_{+}\times\gamma _{-},
\end{equation}
and $g$ is given.
Our purpose is to establish the global existence for \eqref{lineq} and \eqref{linbd} in weighted $L^\infty$ space.
A key point is that we develop some new iterated integral estimates so that one can construct the $L^\infty$ existence without using the
time-decay of the solution in $L^\infty-$norm. We stress that it is very difficult to obtain the global existence and the time-decay of the solution in $L^\infty$ space at the same time due to the fact that the collision frequency $\nu$ has zero lower bound. The main result of this subsection is the following
\begin{proposition}\label{lLif}
Let $(q,\ta)\in\CA_{q,\ta}$, assume \eqref{dlinearcondition} holds true.
Then the initial boundary value problem \eqref{dlinear} and \eqref{lbd}
admits a unique solution satisfying
\begin{equation}\label{es.fif}
\begin{split}
\Vert w_{q,\ta}f\Vert _{\infty }+|w_{q,\ta}f|_{\infty }\lesssim& \Vert w_{q,\ta} f_0\Vert _{\infty }+\sup_{0\leq s\leq t}\Vert \nu^{-1} w_{q,\ta}g(s)\Vert
_{\infty }\\
&+\sqrt{\int_{0}^{t}e^{\lambda_1 s^{\rho_0}}\Vert \nu^{-1/2}g(s)\Vert
_{2}^{2}ds}\\&+\sqrt{\int_{0}^{t}\Vert\nu^{-1/2} w_{q/2,\ta}g(s)\Vert
_{2}^{2}ds}.
\end{split}
\end{equation}
\end{proposition}
Our proof for Proposition \ref{lLif} relies heavily upon the estimates for the iterated integral defined on stochastic cycles.
The stochastic cycles are defined as follows
\begin{definition}[Stochastic Cycles]
\label{diffusecycles}Fixed any point $(t,x,v)$ with $(x,v)\notin \gamma _{0}$%
, let $(t_{0},x_{0},v_{0})$ $=(t,x,v)$. For $v_{k+1}$such
that $v_{k+1}\cdot n(x_{k+1})>0$, define the $(k+1)$-component of the
back-time cycle as
\begin{equation}
(t_{k+1},x_{k+1},v_{k+1})=(t_{k}-t_{\mathbf{b}}(x_{k},v_{k}),x_{%
\mathbf{b}}(x_{k},v_{k}),v_{k+1}).  \label{diffusecycle}
\end{equation}%
Set
\begin{eqnarray*}
X_{\mathbf{cl}}(s;t,x,v) &=&\sum_{k}\mathbf{1}_{[t_{k+1},t_{k})}(s)%
\{x_{k}+(s-t_{k})v_{k}\}, \\
V_{\mathbf{cl}}(s;t,x,v) &=&\sum_{k}\mathbf{1}%
_{[t_{k+1},t_{k})}(s)v_{k}.
\end{eqnarray*}%
Define $\mathcal{V}_{k+1}=\{v\in \R^{3}\ |\ v\cdot
n(x_{k+1})>0\}$, and let the iterated integral for $k\geq 2$ be defined as
\begin{equation*}
\int_{\Pi_{j=1}^{k-1}\mathcal{V}_{j}}\Pi _{j=1}^{k-1}d\sigma
_{j}\equiv \int_{\mathcal{V}_{1}}\cdots\left\{ \int_{\mathcal{V}%
_{k-1}}d\sigma_{k-1}\right\} d\sigma_{1},
\end{equation*}%
where $d\sigma _{j}=\mu (v)(n(x_{j})\cdot v)dv$ is a probability
measure.
\end{definition}

\begin{lemma}\label{k}
Let $T_{0}>0$ and large enough, denote $\al(t)=\max\{t,T_0\}$, there exist constants $%
C_{1},C_{2}>0$ independent of $\al(t)$, such that for $k=C_{1}[\al(t)]^{5/4}$,
and $(t,x,v)\in \lbrack 0,\infty)\times \overline{\Omega }\times \R^{3},$
\begin{equation}
\int_{\Pi _{j=1}^{k-1}\mathcal{V}_{j}}\mathbf{1}_{\{t_{k}(t,x,v,v%
_{1},v_{2},\cdots ,v_{k-1})>0\}}\Pi _{j=1}^{k-1}d\sigma _{j}\leq
\left\{ \frac{1}{2}\right\} ^{C_{2}[\al(t)]^{5/4}}.  \label{largek}
\end{equation}%
We also have, for $(q,\ta)\in\CA_{q,\ta}$, there exist constants $C_3,\ C_4>0$ independent of $k$ such that
\begin{equation}\label{ktildew1}
\begin{split}
\int_{\Pi _{j=1}^{k-1}\mathcal{V}_{j}}\sum_{l=1}^{k-1}\mathbf{1}_{\{t_{l+1}\leq 0<t_{l}\}}
\int_0^{t_l} d\Sigma^w_l(s)ds\leq C_3,
\end{split}
\end{equation}
and
\begin{equation}\label{ktildew2}
\begin{split}
\int_{\Pi _{j=1}^{k-1}\mathcal{V}_{j}}\sum_{l=1}^{k-1}\mathbf{1}_{\{t_{l+1}>0\}}\int_{t_{l+1}}^{t_l} d\Sigma^w_l(s)ds\leq
C_4,
\end{split}
\end{equation}
where
\begin{equation}\label{measure}
d\Sigma^w_{l}(s) =\{\Pi _{j=l+1}^{k-1}d\sigma_{j}\}\times\{e^{\nu
(v_{l})(s-t_{l})}\tilde{w}_{q,\ta}(v_{l})d\sigma_{l}\}\times \Pi _{j=1}^{l-1}\{{{%
e^{\nu (v_{j})(t_{j+1}-t_{j})} d\sigma_{j}}}\},
\end{equation}
and
$
\widetilde{w}_{q,\ta}=\frac{1}{w_{q,\ta}\sqrt{\mu}}.
$
\end{lemma}
\begin{proof}
If $0<t\leq T_0$, then $\al(t)=T_0$, the proof of \eqref{largek} has already been given by Lemma 23 in \cite[pp. 781]{Guo-2010}.
For the case that $T_0<t<+\infty$, setting $T_0=t$ in Lemma 23 of \cite[pp.781]{Guo-2010} and performing the same computations as its proof,
one sees that
\eqref{largek} is also true for $\al(t)=t$. 
In what follows,
we mainly prove \eqref{ktildew2}, the proof for \eqref{ktildew1} will only be briefly sketched.
For any $k>0$, we first split the the left hand side of \eqref{ktildew2} as
\begin{equation}\label{sp.k1}
\begin{split}
\int_{\Pi_{j=1}^{k-1}\mathcal{V}_j}& \sum_{l=1}^{k-1}\mathbf{1}_{\{t_{l+1}>0\}}\int_{t_{l+1}}^{t_l}\tilde{w}_{q,\ta}(v_l)\nu^{-1}(v_l)
\{\Pi _{j=l+1}^{k-1}d\sigma_{j}\}\{e^{\nu
(v_{l})(s-t_{l})}\nu(v_l)d\sigma_{l}\}\\
&\times \Pi _{j=1}^{l-1}\{{{%
e^{\nu (v_{j})(t_{j+1}-t_{j})} d\sigma_{j}}}\}ds\\
=&\int_{\Pi_{j=1}^{k-1}\mathcal{V}_j\atop{\max\{|v_1|,|v_2|,\cdots,|v_{k-1}|\}\leq k}}\sum_{l=1}^{k-1}\mathbf{1}_{\{t_{l+1}>0\}}\int_{t_{l+1}}^{t_l}
\tilde{w}_{q,\ta}(v_l)\nu^{-1}(v_l)
\{\Pi _{j=l+1}^{k-1}d\sigma_{j}\}\\[2mm]
&\quad\times\{e^{\nu
(v_{l})(s-t_{l})}\nu(v_l)d\sigma_{l}\} \Pi _{j=1}^{l-1}\{{{%
e^{\nu (v_{j})(t_{j+1}-t_{j})} d\sigma_{j}}}\}ds\\
&+\int_{\Pi_{j=1}^{k-1}\mathcal{V}_j\atop{\max\{|v_1|,|v_2|,\cdots,|v_{k-1}|\}>k}}\sum_{l=1}^{k-1}\mathbf{1}_{\{t_{l+1}>0\}}\int_{t_{l+1}}^{t_l}
\tilde{w}_{q,\ta}(v_l)\nu^{-1}(v_l)
\{\Pi _{j=l+1}^{k-1}d\sigma_{j}\}\\[2mm]
&\quad\times\{e^{\nu
(v_{l})(s-t_{l})}\nu(v_l)d\sigma_{l}\} \Pi _{j=1}^{l-1}\{{{%
e^{\nu (v_{j})(t_{j+1}-t_{j})} d\sigma_{j}}}\}ds=\mathcal {K}_1+\mathcal {K}_2.
\end{split}
\end{equation}
For $\mathcal {K}_1$, denote $\max\{|v_1|,|v_2|,\cdots,|v_{k-1}|\}=|v_m|$ , one has
\begin{equation*}
\begin{split}
\mathcal {K}_1\leq&C_{q,\ta}\int_{\Pi_{j=1}^{k-1}\mathcal{V}_j}\sum\limits_{l=1}^{k-1}\mathbf{1}_{\{t_{l+1}>0\}}\int_{t_{l+1}}^{t_l}
e^{\nu(v_m)(s-t_1)}\nu(v_m) ds\tilde{w}_{q,\ta}(v_m)\nu^{-1}(v_m)\Pi_{j=1}^{k-1} d\sigma_j\\
\leq&
C_{q,\ta}\int_{\Pi_{j=1}^{k-1}\mathcal{V}_j}\int_{t_{k}}^{t_1}e^{\nu(v_m)(s-t_1)}\nu(v_m)ds \tilde{w}_{q,\ta}(v_m)\nu^{-1}(v_m)\Pi_{j=1}^{k-1} d\sigma_j
\\
\leq&\frac{C_{q,\ta}}{\sqrt{2\pi}} \int_{n(x_m)\cdot v_{m}>0}(n(x_m)\cdot v_{m}) e^{-\frac{1}{4}|v_m|^2-\frac{q}{4}|v_m|^{\ta}}\nu^{-1}(v_m)dv_m\\
\leq& \frac{C_{q,\ta}}{\sqrt{2\pi}} \int_{u_{m1}>0}u_{m1} e^{-\frac{1}{4}|u_m|^2}du_m
\leq C_{q,\ta},
\end{split}
\end{equation*}
here we have used the key observation
$$
\sum\limits_{l=1}^{k-1}\int_{t_{l+1}}^{t_l}
e^{\nu(v_m)(s-t_1)}\nu(v_m) ds\leq \int_{t_{k}}^{t_1}
e^{\nu(v_m)(s-t_1)}\nu(v_m) ds\leq2.
$$

As to $\mathcal {K}_2$, without loss of generality, we may assume $|v_i|> k$ for some $i\in\{1,2,\cdots,k-1\}$, then
\begin{eqnarray*}
\mathcal {K}_2\leq&&C\sum\limits_{l=1}^{k-1}\int_{\Pi_{j=1}^{k-1}\mathcal{V}_j}\tilde{w}_{q,\ta}(v_l)\nu^{-1}(v_l)\Pi_{j=1}^{k-1} d\sigma_j\\
 \leq&&C\int_{\Pi_{j=1}^{i-1}\mathcal{V}_j}\Pi_{j=1}^{i-1} d\sigma_j\int_{n(x_i)\cdot v_{i}>0\atop{|v_i|> k}}(n(x_i)\cdot v_{i}) e^{-\frac{1}{4}|v_i|^2-\frac{q}{4}|v_i|^{\ta}}\nu^{-1}(v_i)dv_i
 \\&&+C\sum\limits_{l=1}^{i-1}\int_{\Pi_{j=1}^{l-1}\mathcal{V}_j}\Pi_{j=1}^{l-1} d\sigma_j\int_{n(x_l)\cdot v_{l}>0}(n(x_l)\cdot v_{l}) e^{-\frac{1}{4}|v_l|^2-\frac{q}{4}|v_l|^{\ta}}\nu^{-1}(v_l)dv_l
  \\&&\quad\times\int_{\Pi_{j=l+1}^{i-1}\mathcal{V}_j}\Pi_{j=l+1}^{i-1} d\sigma_j
\int_{n(x_i)\cdot v_{i}>0\atop{|v_i|> k}}e^{-\frac{|v_i|^2}{2}}(n(x_i)\cdot v_i)dv_i\\
 \\&&+C\sum\limits_{l=i+1}^{k-1}\int_{\Pi_{j=1}^{i-1}\mathcal{V}_j}\Pi_{j=1}^{i-1} d\sigma_j\int_{n(x_i)\cdot v_{i}>0\atop{|v_i|> k}}e^{-\frac{|v_i|^2}{2}}(n(x_i)\cdot v_i)dv_i
 \\&&\quad\times\int_{\Pi_{j=i+1}^{l-1}\mathcal{V}_j}\Pi_{j=i+1}^{l-1} d\sigma_j\int_{n(x_l)\cdot v_{l}>0}(n(x_l)\cdot v_{l}) e^{-\frac{1}{4}|v_l|^2-\frac{q}{4}|v_l|^{\ta}}\nu^{-1}(v_l)dv_l
 \\
\leq&& C_{q,\ta}(k-1)e^{-\frac{k^2}{8}}\leq C_{q,\ta}.
\end{eqnarray*}
Substituting the above estimates for $\mathcal {K}_1$ and $\mathcal {K}_2$ into \eqref{sp.k1}, we see that \eqref{ktildew2} is true.

The proof for \eqref{ktildew1} is much similar as that of \eqref{ktildew2}, the only difference is the following
\begin{equation*}
\begin{split}
\int_{\Pi_{j=1}^{k-1}\mathcal{V}_j\atop{\max\{|v_1|,|v_2|,\cdots,|v_{k-1}|\}\leq k}}&\sum_{l=1}^{k-1}\mathbf{1}_{\{t_{l+1}\leq0<t_l\}}\int_{0}^{t_l} \tilde{w}_{q,\ta}(v_l)\nu^{-1}(v_l)
\{\Pi _{j=l+1}^{k-1}d\sigma_{j}\}\\
&\times \{e^{\nu
(v_{l})(s-t_{l})}\nu(v_l)d\sigma_{l}\}\Pi _{j=1}^{l-1}\{{{e^{\nu (v_{j})(t_{j+1}-t_{j})} d\sigma_{j}}}\}ds\\
\leq&\int_{\Pi_{j=1}^{k-1}\mathcal{V}_j}\sum_{l=1}^{k-1}\mathbf{1}
_{\{t_{l+1}\leq 0<t_{l}\}}\tilde{w}_{q,\ta}(v_m)\nu^{-1}(v_m)\Pi_{j=1}^{k-1} d\sigma_j\\
\leq&
\int_{\Pi_{j=1}^{k-1}\mathcal{V}_j}\tilde{w}_{q,\ta}(v_m)\nu^{-1}(v_m)\Pi_{j=1}^{k-1} d\sigma_j
\leq C_{q,\ta},
\end{split}
\end{equation*}
here the second inequality follows due to $\sum_{l=1}^{k-1}\mathbf{1}_{\{t_{l+1}\leq 0<t_{l}\}}=\mathbf{1}_{\{t_{l+1}\leq 0\}}$.
This finishes the proof of Lemma \ref{k}.
\end{proof}

\begin{remark}
Since $\al(t)\leq T_0$, the upper bound on the right hand side of \eqref{largek} can be relaxed to $\left\{ \frac{1}{2}\right\} ^{C_{2}T_0^{5/4}}$.
\end{remark}
Priori to proving Proposition \ref{lLif}, we first show the following global solvability of \eqref{dlinear} and \eqref{lbd} in $L^\infty$ space without weight.
\begin{lemma}\label{lex.ow}
There exists $\vps_0>0$ such that if $\FP g=0$ and
\begin{equation*}
\begin{split}
\|f_0&\|_{L^\infty(\Om\cup \ga_+)}+\|w_{q/2,\ta}f_0\|_2+\sup\limits_{0\leq s\leq t}\|\nu^{-1}g(s)\|_{\infty}
\\&+\sqrt{\int_{0}^{t}e^{\lambda_1 s^{\rho_0}}\Vert \nu^{-1/2}g(s)\Vert_{2}^{2}ds}+\sqrt{\int_{0}^{t}\Vert\nu^{-1/2} w_{q/2,\ta}g(s)\Vert_{2}}\leq \varepsilon_0,
\end{split}
\end{equation*}
then \eqref{dlinear} and \eqref{lbd} admits a unique solution $f(t,x,v)$ for which it holds that
\begin{equation}\label{fif.nw}
\begin{split}
\sup\limits_{0\leq s\leq t}\Vert f\Vert _{\infty }+|f|_{\infty }\lesssim& \|f_0\|_{L^\infty(\Om\cup \ga_+)}
+\|w_{q/2,\ta}f_0\|_2+\sup\limits_{0\leq s\leq t}\|\nu^{-1}g(s)\|_{\infty}\\&+\sqrt{\int_{0}^{t}e^{\lambda_1 s^{\rho_0}}\Vert \nu^{-1/2}g(s)\Vert
_{2}^{2}ds}+\sqrt{\int_{0}^{t}\Vert\nu^{-1/2} w_{q/2,\ta}g(s)\Vert
_{2}^{2}ds}.
\end{split}
\end{equation}
\end{lemma}

\begin{proof}
As in the proof of Lemma \ref{lg.ex},  we use the approximate form
\begin{eqnarray}\label{j0.eq}
\left\{
\begin{array}{rll}
&\partial_{t}f+v\cdot \nabla _{x}f+Lf=g,\text{ \ \ }f(0,x,v)=f_{0}, \\[2mm]
&f_{-}=(1-\frac{1}{j})P_{\gamma }f,\ j=2,3,\cdots
\end{array}\right.
\end{eqnarray}
to construct the global existence of \eqref{dlinear} and \eqref{lbd}, while
the global solution (denoted by $f^j$) of \eqref{j0.eq} is further established by the following iteration scheme
\begin{eqnarray*}
\left\{
\begin{array}{rll}
&\partial _{t}f^{\ell +1}+v\cdot \nabla _{x}f^{\ell +1}+\nu f^{\ell
+1}-Kf^{\ell }=g,\text{ \ \ }f^{\ell+1}(0)=f_{0}, \ \ \ell\geq0,\ f^{0}\equiv f_{0},\\[2mm]
&f_{-}^{\ell +1}=(1-\frac{1}{j})P_{\gamma }f^{\ell },\ j=2,3,\cdots.
\end{array}\right.
\end{eqnarray*}
To do so,
performing a similar calculation as for deriving (199) in Lemma 24 of \cite[pp.783]{Guo-2010}, 
we find
\begin{equation*}
\begin{split}
  |f^{\ell+1}(t,x,v)| \leq&
\underbrace{\left\{\mathbf{1}_{t_{1}\leq 0}
\int_{0}^{t}+\mathbf{1}_{t_{1}>0}
\int_{t_1}^{t}\right\}
e^{-\nu
(v)(t-s)}|Kf^{\ell}(s,x-(t-s){v},v)|ds}_{I_1}
\\
&
\underbrace{+\left\{\mathbf{1}_{t_{1}\leq 0}
\int_{0}^{t}+\mathbf{1}_{t_{1}>0}
\int_{t_1}^{t}\right\}e^{-\nu(v)(t-s)}|g(s,x-(t-s){v},v)|ds}_{I_3}\\
&+\mathbf{1}_{t_{1}\leq 0}e^{-\nu
(v)t}|f^{\ell+1}(0,x-t{v},v)|
\\&+\mathbf{1}_{t_{1}> 0}\left(1-\frac{1}{j}\right)e^{-\nu (v)(t-t_{1})}\int_{\mathcal {V}_1}|f^{\ell}|d\si_1,
\end{split}
\end{equation*}
where the last line follows from the boundary condition. 
A direct calculation leads us to
\begin{equation*}
\begin{split}
\|f^{\ell+1}\|_{L^\infty(\Om\cup \ga_+)}\leq& tC\|f^{\ell}\|_{\infty}+\left(1-\frac{1}{j}\right)|f^{\ell}|_{\infty,+}
+\|f_0\|_\infty+\sup\limits_{0\leq s\leq t}\|\nu^{-1}g(s)\|_{\infty}\\
\leq&tC\|f^{\ell}\|_{\infty}+\left(1-\frac{1}{j}\right)|f^{\ell}|_{\infty,+}
+\vps_0.
\end{split}
\end{equation*}
With this, one can show that
there exists $T^{\ast\ast}>0$ $(CT^{\ast\ast}<1)$ such that if $\sup\limits_{0\leq t\leq T^{\ast\ast}}\|f^{\ell}\|_{L^\infty(\Om\cup \ga_+)}\leq 2\vps_0$ then
$$\sup\limits_{0\leq t\leq T^{\ast\ast}}\|f^{\ell+1}\|_{L^\infty(\Om\cup \ga_+)}\leq 2\vps_0,$$
thus $\{\|f^{\ell}\|_{L^\infty(\Om\cup \ga_+)}\}$ is uniformly bounded with respect to $\ell$ in a short time interval $[0,T^{\ast\ast}]$. In fact, we can further prove that $\{f^\ell\}$ is also a Cauchy sequence in $L^\infty(\Om\cup \ga_+)$ provided $CT^{\ast\ast}<1$, thus we obtain a local solution $f^j$ for \eqref{dlinearj}. To construct the global existence, it suffices to obtain the following {\it a priori estimates}
\begin{equation}\label{lif.esow}
\begin{split}
\sup\limits_{0\leq s\leq t}&\{\Vert f^j\Vert _{\infty}+|f^j|_{\infty,+}\}\\
\lesssim& \|f_0\|_{L^\infty(\Om\cup \ga_+)}
+\sup\limits_{0\leq s\leq t}\|\nu^{-1}g(s)\|_{\infty}+\|w_{q/2,\ta}f_0\|_2\\
&+\sqrt{\int_{0}^{t}e^{\lambda_1 s^{\rho_0}}\Vert \nu^{-1/2}g(s)\Vert
_{2}^{2}ds}+\sqrt{\int_{0}^{t}\Vert\nu^{-1/2} w_{q/2,\ta}g(s)\Vert
_{2}^{2}ds},
\end{split}
\end{equation}
for all $j\geq2.$
In fact, \eqref{lif.esow} follows from a tedious calculation for the following inequality
\begin{equation}\label{iteration1}
\begin{split}
  |f^j(t,x,v)| \leq&
\left\{\mathbf{1}_{t_{1}\leq 0}
\int_{0}^{t}+\mathbf{1}_{t_{1}>0}
\int_{t_1}^{t}\right\}
e^{-\nu
(v)(t-s)}|K^{1-\chi}f^j(s,x-(t-s){v},v)|ds
\\
&+\left\{\mathbf{1}_{t_{1}\leq 0}
\int_{0}^{t}+\mathbf{1}_{t_{1}>0}
\int_{t_1}^{t}\right\}
e^{-\nu
(v)(t-s)}|K^{\chi}f^j(s,x-(t-s){v},v)|ds
\\
&
+\left\{\mathbf{1}_{t_{1}\leq 0}
\int_{0}^{t}+\mathbf{1}_{t_{1}>0}
\int_{t_1}^{t}\right\}e^{-\nu
(v)(t-s)}|g(s,x-(t-s){v},v)|ds
+\sum\limits_{n=1}^{5}I_n,
\end{split}
\end{equation}%
with
\begin{equation*}
\begin{split}
I_1=&\mathbf{1}_{t_{1}\leq 0}e^{-\nu
(v)t}|f(0,x-t{v},v)|
\\&+e^{-\nu (v)(t-t_{1})}\int_{\prod_{j=1}^{k-1}%
\mathcal{V}_{j}}\sum_{l=1}^{k-1}\mathbf{1}_{\{t_{l+1}\leq
0<t_{l}\}}|f(0,x_{l}-t_{l}{v}_{l},v_{l})|d\Sigma _{l}(0),
\end{split}
\end{equation*}
\begin{equation*}
\begin{split}
I_2=&e^{-\nu (v)(t-t_{1})}\sqrt{\mu}\bigg\{\int_{\prod_{j=1}^{k-1}%
\mathcal{V}_{j}}\mathbf{1}_{\{t_{l+1}\leq
0<t_{l}\}}\sum_{l=1}^{k-1}\int_{0}^{t_l}\\&\quad\times|[K^{1-\chi} f^j](s,x_{l}-(t_{l}-s){v}_{l},v_{l})|d\Sigma _{l}(s)ds
\\&+\int_{\prod_{j=1}^{k-1}%
\mathcal{V}_{j}}\sum_{l=1}^{k-1}\mathbf{1}_{\{0<t_{l+1}\}}\int_{t_{l+1}}^{t_{l}}
|[K^{1-\chi} f^j](s,x_{l}-(t_{l}-s){v}_{l},v_{l})|d\Sigma _{l}(s)ds\bigg\},
\end{split}
\end{equation*}
\begin{equation*}
\begin{split}
I_3=&e^{-\nu (v)(t-t_{1})}\sqrt{\mu}\bigg\{\int_{\prod_{j=1}^{k-1}%
\mathcal{V}_{j}}\mathbf{1}_{\{t_{l+1}\leq
0<t_{l}\}}\sum_{l=1}^{k-1}\int_{0}^{t_l}\\&\quad\times|[K^\chi f^j](s,x_{l}-(t_{l}-s){v}_{l},v_{l})|d\Sigma _{l}(s)ds
\\&+\int_{\prod_{j=1}^{k-1}%
\mathcal{V}_{j}}\sum_{l=1}^{k-1}\mathbf{1}_{\{0<t_{l+1}\}}\int_{t_{l+1}}^{t_{l}}
|[K^\chi f^j](s,x_{l}-(t_{l}-s){v}_{l},v_{l})|d\Sigma _{l}(s)ds\bigg\},
\end{split}
\end{equation*}
\begin{equation*}
\begin{split}
I_4=&e^{-\nu (v)(t-t_{1})}\sqrt{\mu}\bigg\{\int_{\prod_{j=1}^{k-1}%
\mathcal{V}_{j}}\sum_{l=1}^{k-1}\mathbf{1}_{\{t_{l+1}\leq
0<t_{l}\}}\int_{0}^{t_l}|g(s,x_{l}-(t_{l}-s){v}_{l},v_{l})|d\Sigma _{l}(s)ds \\&+\int_{\prod_{j=1}^{k-1}%
\mathcal{V}_{j}}\sum_{l=1}^{k-1}\mathbf{1}_{\{0<t_{l+1}\}}\int_{t_{l+1}}^{t_{l}}|g(s,x_{l}-(t_{l}-s){v}_{l},v_{l})|d\Sigma _{l}(s)ds\bigg\},
\end{split}
\end{equation*}
\begin{equation*}
I_5=e^{-\nu (v)(t-t_{1})}\sqrt{\mu}\int_{\prod_{j=1}^{k-1}%
\mathcal{V}_{j}}\mathbf{1}_{\{0<t_{k}\}}|f^j(t_{k},x_{k},v_{k-1})|d\Sigma
_{k-1}(t_{k}),\ \ k\geq2,
\end{equation*}%
and
\begin{equation}\label{measure1}
d\Sigma _{l}(s)=\{\Pi _{j=l+1}^{k-1}d\sigma _{j}\}\times\{e^{\nu
(v_{l})(s-t_{l})}\mu^{-1/2}(v_l)d\sigma _{l}\}\times \Pi _{j=1}^{l-1}\{{{%
e^{\nu (v_{j})(t_{j+1}-t_{j})} d\sigma _{j}}}\}.
\end{equation}
We point out that \eqref{iteration1} is deduced from \eqref{dlinearj} by means of a similar argument as for obtaining (199) in \cite[pp.783]{Guo-2010}.
The estimates for the corresponding terms on the right hand side of \eqref{iteration1} are much similar as that of $\CI_{n}$ $(1\leq n\leq 8)$ in \eqref{iteration}.
To avoid needless repetition, we are not going to detail the computations here.
When \eqref{lif.esow} is derived, the global existence of (\ref{daproximatenew}) and \eqref{fl.bd} follows from a standard continuation argument. 
Notice that \eqref{lif.esow} is unform in $j$, $\{f^j\}_{j=1}^\infty$ possesses (up to a subsequence) a weak$-\ast$ limit $f$ which satisfies
(\ref{dlinear}) and \eqref{lbd} in the weak sense.
Again, by taking a difference one has
\begin{eqnarray*}
\left\{
\begin{array}{rll}
&\partial _{t}[f^{j}-f]+v\cdot \nabla _{x}[f^{j}-f]+L[f^{j}-f]=0,\ \ [f^{j}-f](0)=0,\\[2mm]
&[f^{j}-f]_{-}=P_{\gamma }[f^{j}-f]+\frac{1}{j}P_{\gamma }f^{j},
\end{array}\right.
\end{eqnarray*}
from which, we have by a similar argument as for obtaining \eqref{iteration1}
\begin{equation}\label{iteration2}
\begin{split}
|&[f^{j}-f](t,x,v)|\\ \leq&
\left\{\mathbf{1}_{t_{1}\leq 0}
\int_{0}^{t}+\mathbf{1}_{t_{1}>0}
\int_{t_1}^{t}\right\}
e^{-\nu
(v)(t-s)}|K^{1-\chi}[f^{j}-f](s,x-(t-s){v},v)|ds
\\
&+\left\{\mathbf{1}_{t_{1}\leq 0}
\int_{0}^{t}+\mathbf{1}_{t_{1}>0}
\int_{t_1}^{t}\right\}
e^{-\nu
(v)(t-s)}|K^{\chi}[f^{j}-f](s,x-(t-s){v},v)|ds
+\sum\limits_{n=6}^{9}I_n,
\end{split}
\end{equation}%
with
\begin{equation*}
\begin{split}
I_6=&\frac{1}{j}\mathbf{1}_{t_{1}> 0}e^{-\nu
(v)(t-t_1)}|(P_\ga f^j)(t_1,x_1,v)|
\\&+\frac{1}{j}e^{-\nu (v)(t-t_{1})}\sqrt{\mu}\int_{\prod_{j=1}^{k-1}%
\mathcal{V}_{j}}\sum_{l=1}^{k-1}\mathbf{1}_{\{t_{l+1}>0\}}|(P_\ga f^j)(t_{l+1},x_{l+1},v_{l})|d\Sigma _{l}(t_{l+1}),
\end{split}
\end{equation*}
\begin{equation*}
\begin{split}
I_7=&e^{-\nu (v)(t-t_{1})}\sqrt{\mu}\bigg\{\int_{\prod_{j=1}^{k-1}%
\mathcal{V}_{j}}\sum_{l=1}^{k-1}\mathbf{1}_{\{t_{l+1}\leq
0<t_{l}\}}\\&\times\int_{0}^{t_l}|[K^{1-\chi} [f^{j}-f]](s,x_{l}-(t_{l}-s){v}_{l},v_{l})|d\Sigma _{l}(s)ds
\\&+\int_{\prod_{j=1}^{k-1}%
\mathcal{V}_{j}}\sum_{l=1}^{k-1}\mathbf{1}_{\{0<t_{l+1}\}}\int_{t_{l+1}}^{t_{l}}
|[K^{1-\chi} [f^{j}-f]](s,x_{l}-(t_{l}-s){v}_{l},v_{l})|d\Sigma _{l}(s)ds\bigg\},
\end{split}
\end{equation*}
\begin{equation*}
\begin{split}
I_8=&e^{-\nu (v)(t-t_{1})}\sqrt{\mu}\bigg\{\int_{\prod_{j=1}^{k-1}%
\mathcal{V}_{j}}\sum_{l=1}^{k-1}\mathbf{1}_{\{t_{l+1}\leq
0<t_{l}\}}\\&\times\int_{0}^{t_l}|[K^\chi [f^{j}-f]](s,x_{l}-(t_{l}-s){v}_{l},v_{l})|d\Sigma _{l}(s)ds
\\&+\int_{\prod_{j=1}^{k-1}%
\mathcal{V}_{j}}\sum_{l=1}^{k-1}\mathbf{1}_{\{0<t_{l+1}\}}\int_{t_{l+1}}^{t_{l}}
|[K^\chi [f^{j}-f]](s,x_{l}-(t_{l}-s){v}_{l},v_{l})|d\Sigma _{l}(s)ds\bigg\},
\end{split}
\end{equation*}
\begin{equation*}
I_{9}=e^{-\nu (v)(t-t_{1})}\sqrt{\mu}\int_{\prod_{j=1}^{k-1}%
\mathcal{V}_{j}}\mathbf{1}_{\{0<t_{k}\}}|[f^{j}-f](t_{k},x_{k},v_{k-1})|d\Sigma
_{k-1}(t_{k}),\ \ k\geq2.
\end{equation*}
Comparing \eqref{iteration2} with \eqref{iteration1}, one obtains
\begin{equation}\label{fj-f.es}
\sup\limits_{0\leq s\leq t}\{\Vert [f^{j}-f]\Vert _{\infty}(s)+|[f^{j}-f]|_{\infty,+}(s)\}\lesssim C\sup\limits_{0\leq s\leq t}|I_6|.
\end{equation}
On the other hand, from Lemma \ref{k}, it follows
\begin{equation}\label{I61}
|I_6|\leq \frac{C}{j}|P_\ga f^j|_{\infty,-}\leq \frac{C}{j}|f^j|_{\infty,+}.
\end{equation}
\eqref{fj-f.es} and \eqref {I61} then lead us to
\begin{equation*}
\sup\limits_{0\leq s\leq t}\{\Vert [f^{j}-f](s) _{\infty}+|[f^{j}-f](s)|_{\infty,+}\}\lesssim \frac{C}{j}\sup\limits_{0\leq s\leq t}|f^j|_{\infty,+},
\end{equation*}
from which and the bound \eqref{lif.esow}, we see that $f^j$ converges to $f$ strongly in $L^\infty$ and $f$ satisfies \eqref{fif.nw}, this completes the proof of Lemma \ref{lex.ow}.
\end{proof}
With Lemma \ref{lex.ow} in hand, we are now ready to tackle
\begin{proof}[The proof of Proposition \ref{lLif}]

Similar to the analysis in Section \ref{L2th}, denote
$$h^{\ell}=w_{q,\ta}f^{\ell },\ \ \ell\geq0,\ \ \textrm{and}\ \ K_w(\cdot)=w_{q,\ta} K\left(\frac{\cdot}{w_{q,\ta}}\right),$$
where $f^{\ell }$ is determined by (\ref{daproximatenew}) and \eqref{fl.bd}.
The solution $h^j(t,x,v)=w_{q,\ta}f^j$ of the problem
\begin{eqnarray}\label{hj.eqn}
\pa_th+v\cdot \na_x h+\nu h-K_wh=w_{q,\ta}g,\ \
h(0,x,v)=h_0(x,v)=w_{q,\ta}f_0(x,v),
\end{eqnarray}
and
\begin{equation}\label{hj.bd}
h_-=\frac{1-\frac{1}{j}}{\widetilde{w}_{q,\ta}}\int_{\mathcal {V}(x)}h(t,x,v')\widetilde{w}_{q,\ta}(v')d\si
\end{equation}
will be constructed with the help of an \textit{abstract iteration}
scheme defined in the following way
\begin{eqnarray}\label{h.it}
\left\{
\begin{array}{rll}
&\pa_th^{\ell+1}+v\cdot \na_x h^{\ell+1}+\nu h^{\ell+1}-K_wh^{\ell}=w_{q,\ta}g,\\[2mm]
&h^{\ell+1}(0,x,v)=h^{\ell+1}_0(x,v)=w_{q,\ta}f_0(x,v), \ell\geq0,
\end{array}\right.
\end{eqnarray}
with $h^0=h_0=w_{q,\ta}f_0(x,v)$ and
\begin{equation}\label{h.bdit}
h^{\ell+1}_-=\frac{1-\frac{1}{j}}{\widetilde{w}_{q,\ta}}\int_{\mathcal {V}(x)}h^{\ell}(t,x,v')\widetilde{w}_{q,\ta}(v')d\si.
\end{equation}
From \eqref{h.it} and \eqref{h.bdit}, it is straightforward to check
\begin{equation}\label{iteration.s}
\begin{split}
  |h^{\ell+1}(t,x,v)| \leq&
\left\{\mathbf{1}_{t_{1}\leq 0}
\int_{0}^{t}+\mathbf{1}_{t_{1}>0}
\int_{t_1}^{t}\right\}
e^{-\nu
(v)(t-s)}|K_{w}h^{\ell}(s,x-(t-s){v},v)|ds
\\
&
+\left\{\mathbf{1}_{t_{1}\leq 0}
\int_{0}^{t}+\mathbf{1}_{t_{1}>0}
\int_{t_1}^{t}\right\}e^{-\nu
(v)(t-s)}|w_{q,\ta}g(s,x-(t-s){v},v)|ds\\
&+\mathbf{1}_{t_{1}\leq 0}e^{-\nu
(v)t}|h^{\ell+1}(0,x-t{v},v)|
\\&+\mathbf{1}_{t_{1}> 0}\left(1-\frac{1}{j}\right)\frac{e^{-\nu (v)(t-t_{1})}}{\tilde{w}_{q,\ta}(v)}\int_{\mathcal {V}_1}|h^{\ell}(t_1,x_1,v_1)|\tilde{w}_{q,\ta}(v_1)d\si_1.
\end{split}
\end{equation}
Since
$\frac{1}{\tilde{w}_{q,\ta}(v)}\leq C_{q,\ta},$ and $\int_{n\cdot v>0}\sqrt{\mu}(n\cdot v) dv<\infty,$
we get from \eqref{iteration.s} and Lemma \ref{es.k}
that
\begin{equation}\label{hlb1}
\begin{split}
  \|h^{\ell+1}(t)\|_{\infty} \leq&
Ct\|h^{\ell}(t,x,v)\|_{\infty}+C\|h_0\|_{\infty}+C\sup\limits_{0\leq s\leq t}\|\nu^{-1}w_{q,\ta}g(s)\|_\infty\\&+C_{q,\ta}\max\limits_{\ell}\sup\limits_{0\leq s\leq t}|f^\ell|_{\infty,+}.
\end{split}
\end{equation}
Recalling Lemma \ref{lex.ow}, we have shown that $f^\ell\rightarrow f^j$ and $f^j$ bears the bound \eqref{lif.esow},
therefore $\max\limits_{\ell}\sup\limits_{0\leq s\leq t}|f^\ell|_{\infty,+}<\infty$.
From this and \eqref{hlb1}, for any given $j\geq2,$ the existence of a local solution $h^j$ to
\eqref{hj.eqn} and \eqref{hj.bd} is guaranteed  by a similar argument as the proof of Lemma \ref{lex.ow}.

To obtain the global existence of \eqref{hj.eqn} and \eqref{hj.bd}, a central part of deduction is the following {\it a priori} estimates
\begin{equation}\label{es.hlp1}
\begin{split}
\sup_{0\leq s\leq t}\Vert h^j(s)\Vert
_{\infty } \leq&
C\Vert h_0\Vert _{\infty }+C\sup_{0\leq s\leq t}\Vert \nu^{-1} w_{q,\ta}g(s)\Vert
_{\infty }+C\|w_{q/2,\ta}f_0\|_2\\
&+C\sqrt{\int_{0}^{t}e^{\lambda_1 s^{\rho_0}}\Vert \nu^{-1/2}g(s)\Vert
_{2}^{2}ds}+C\sqrt{\int_{0}^{t}\Vert\nu^{-1/2} w_{q/2,\ta}g(s)\Vert
_{2}^{2}ds}.
\end{split}
\end{equation}
Upon using \eqref{hj.eqn} and \eqref{hj.bd}, once again, we proceed like for deducing (199) in \cite[pp. 783]{Guo-2010} to obtain
\begin{equation}\label{iteration}
\begin{split}
  |h^j(t,x,v)| \leq&
\underbrace{\left\{\mathbf{1}_{t_{1}\leq 0}
\int_{0}^{t}+\mathbf{1}_{t_{1}>0}
\int_{t_1}^{t}\right\}
e^{-\nu
(v)(t-s)}|K^{1-\chi}_{w}h^j(s,x-(t-s){v},v)|ds}_{\CI_1}
\\
&\underbrace{+\left\{\mathbf{1}_{t_{1}\leq 0}
\int_{0}^{t}+\mathbf{1}_{t_{1}>0}
\int_{t_1}^{t}\right\}
e^{-\nu
(v)(t-s)}|K^{\chi}_{w}h^j(s,x-(t-s){v},v)|ds}_{\CI_2}
\\
&
\underbrace{+\left\{\mathbf{1}_{t_{1}\leq 0}
\int_{0}^{t}+\mathbf{1}_{t_{1}>0}
\int_{t_1}^{t}\right\}e^{-\nu
(v)(t-s)}|w_{q,\ta}g(s,x-(t-s){v},v)|ds}_{\CI_3}
\\&+\sum\limits_{n=4}^{8}\CI_n,
\end{split}
\end{equation}%
with
\begin{equation*}
\begin{split}
\CI_4=&\mathbf{1}_{t_{1}\leq 0}e^{-\nu
(v)t}|h(0,x-t{v},v)|
\\&+\frac{e^{-\nu (v)(t-t_{1})}}{\tilde{w}_{q,\ta}(v)}\int_{\prod_{j=1}^{k-1}%
\mathcal{V}_{j}}\sum_{l=1}^{k-1}\mathbf{1}_{\{t_{l+1}\leq
0<t_{l}\}}|h(0,x_{l}-t_{l}{v}_{l},v_{l})|d\Sigma^w_{l}(0),
\end{split}
\end{equation*}
\begin{equation*}
\begin{split}
\CI_5=&\frac{e^{-\nu (v)(t-t_{1})}}{\tilde{w}_{q,\ta}(v)}\bigg\{\int_{\prod_{j=1}^{k-1}%
\mathcal{V}_{j}}\sum_{l=1}^{k-1}\mathbf{1}_{\{t_{l+1}\leq0<t_{l}\}}\int_{0}^{t_l}
\\&\quad\times|[K^{1-\chi}_{w}h^j](s,x_{l}-(t_{l}-s){v}_{l},v_{l})|d\Sigma^w_{l}(s)ds
\\&+\int_{\prod_{j=1}^{k-1}%
\mathcal{V}_{j}}\sum_{l=1}^{k-1}\mathbf{1}_{\{0<t_{l+1}\}}\int_{t_{l+1}}^{t_{l}}
|[K^{1-\chi}_{w}h^j](s,x_{l}-(t_{l}-s){v}_{l},v_{l})|d\Sigma^w_{l}(s)ds\bigg\},
\end{split}
\end{equation*}
\begin{equation*}
\begin{split}
\CI_6=&\frac{e^{-\nu (v)(t-t_{1})}}{\tilde{w}_{q,\ta}(v)}\bigg\{\int_{\prod_{j=1}^{k-1}%
\mathcal{V}_{j}}\sum_{l=1}^{k-1}\mathbf{1}_{\{t_{l+1}\leq0<t_{l}\}}\int_{0}^{t_l}
\\&\quad\times|[K^\chi_{w}h^j](s,x_{l}-(t_{l}-s){v}_{l},v_{l})|d\Sigma^w_{l}(s)ds
\\&+\int_{\prod_{j=1}^{k-1}%
\mathcal{V}_{j}}\sum_{l=1}^{k-1}\mathbf{1}_{\{0<t_{l+1}\}}\int_{t_{l+1}}^{t_{l}}
|[K^\chi_{w}h^j](s,x_{l}-(t_{l}-s){v}_{l},v_{l})|d\Sigma^w_{l}(s)ds\bigg\},
\end{split}
\end{equation*}
\begin{equation*}
\begin{split}
\CI_7=&\frac{e^{-\nu (v)(t-t_{1})}}{\tilde{w}_{q,\ta}(v)}\bigg\{\int_{\prod_{j=1}^{k-1}%
\mathcal{V}_{j}}\sum_{l=1}^{k-1}\mathbf{1}_{\{t_{l+1}\leq0<t_{l}\}}
\int_{0}^{t_l}\\&\quad\times|w_{q,\ta}g(s,x_{l}-(t_{l}-s){v}_{l},v_{l})|d\Sigma^w_{l}(s)ds \\&+\int_{\prod_{j=1}^{k-1}%
\mathcal{V}_{j}}\sum_{l=1}^{k-1}\mathbf{1}_{\{0<t_{l+1}\}}\int_{t_{l+1}}^{t_{l}}|w_{q,\ta}g(s,x_{l}-(t_{l}-s){v}_{l},v_{l})|
d\Sigma^w_{l}(s)ds\bigg\},
\end{split}
\end{equation*}
\begin{equation*}
\CI_8=\frac{e^{-\nu (v)(t-t_{1})}}{\tilde{w}_{q,\ta}(v)}\int_{\prod_{j=1}^{k-1}%
\mathcal{V}_{j}}\mathbf{1}_{\{0<t_{k}\}}|h^j(t_{k},x_{k},v_{k-1})|d\Sigma^w_{k-1}(t_{k}),\ \ k\geq2,
\end{equation*}%
and $d\Sigma^w_{l}(s)$ is given by \eqref{measure}.
The main difference between this proof and that of Lemma \ref{lex.ow}
is that we now have an additional velocity weight $w_{q,\ta}$.
We now turn to compute $\CI_{n}$ $(1\leq n\leq 8)$ in \eqref{iteration} term by term. 

\noindent{\it Estimates for $\CI_1$ and $\CI_5$.} Notice that
\begin{equation}\label{nu.int}
\int_0^te^{-\nu(v)(t-s)}\nu(v)ds<+\infty.
\end{equation}
From Lemma \ref{es.k}, it follows that
$$
\CI_1\leq C\eps^{\varrho+3}\sup_{0\leq s\leq t}\| h^j(s)\|_{\infty}.
$$
Since $\tilde{w}^{-1}_{q,\ta}(v)\leq C_{q,\ta}$, Lemma \ref{es.k} and \eqref{ktildew1} imply the first term in $\CI_5$ can be bounded by
\begin{equation*}
\begin{split}
C_{q,\ta}\eps^{\varrho+3}&\int_{\prod_{j=1}^{k-1}%
\mathcal{V}_{j}}\sum_{l=1}^{k-1}\mathbf{1}_{\{t_{l+1}\leq
0<t_{l}\}}\int_0^{t_l}\left\|h^j(s)\right\|_{\infty}d\Sigma _{l}(s)ds\\
 \leq&
C_{q,\ta}\eps^{\varrho+3}\sup_{0\leq s\leq t}\left\|h^j(s)\right\|_{\infty}.
\end{split}
\end{equation*}
As to the second term in $\CI_5$, by Lemma \ref{es.k} and \eqref{ktildew2}, we get its upper bound
\begin{equation*}
\begin{split}
C_{q,\ta}\eps^{\varrho+3}&\int_{\prod_{j=1}^{k-1}%
\mathcal{V}_{j}}\sum\limits_{l=1}^{k-1}{\bf 1}_{\{0<t_{l+1}\}}\int_{t_{l+1}}^{t_l}\left\|h^j(s)\right\|_{\infty}
d\Sigma _{l}(s)ds\\
\leq&
C_{q,\ta}\eps^{\varrho+3}\sup_{0\leq s\leq t}\left\|h^j(s)\right\|_{\infty}.
\end{split}
\end{equation*}

\noindent{\it Estimates for $\CI_3$ and $\CI_7$.}
From \eqref{nu.int}, it is straightforward to check
\begin{equation*}
\CI_3\leq C\sup_{0\leq s\leq t}\left\|\nu^{-1}(v)w_{q,\ta}g(s)\right\|_{\infty}.
\end{equation*}
In view of \eqref{ktildew1}, one sees that the first term in $\CI_7$ can be dominated by
\begin{equation}\label{I51}
\begin{split}
C_{q,\ta}\int_{\prod_{j=1}^{k-1}%
\mathcal{V}_{j}}\sum_{l=1}^{k-1}\mathbf{1}_{\{t_{l+1}\leq
0<t_{l}\}}\int_0^{t_l}\left\|w_{q,\ta}g(s)\right\|_{\infty} d\Sigma _{l}(s)ds
 \leq
C_{q,\ta}\sup_{0\leq s\leq t}\left\|w_{q,\ta}g(s)\right\|_{\infty}.
\end{split}
\end{equation}
Thanks to \eqref{nu.int} and \eqref{ktildew1}, we bound the second term in $\CI_7$ by
\begin{equation*}
\begin{split}
C_{q,\ta}&\int_{\prod_{j=1}^{k-1}%
\mathcal{V}_{j}}\mathbf{1}_{\{0<t_{l+1}\}}
\int_{t_{l+1}}^{t_l}\left\|w_{q,\ta}g(s)\right\|_{\infty}
d\Sigma _{l}(s)ds
\leq
C_{q,\ta}\sup_{0\leq s\leq t}\left\|w_{q,\ta}g(s)\right\|_{\infty}.
\end{split}
\end{equation*}
\noindent{\it Estimates for $\CI_4$.}
By a similar manner as for obtaining \eqref{I51}, we have
\begin{equation*}
\begin{split}
\CI_4\leq&\|h(0)\|_{\infty }+C_{q,\ta}\|h(0)\|_{\infty
}\int \sum_{l=1}^{k-1}\mathbf{1}_{\{t_{l+1}\leq 0<t_{l}\}}d\Sigma _{l}(0)
\leq C_{q,\ta}\|h(0)\|_{\infty }.
\end{split}
\end{equation*}

\noindent{\it Estimates for $\CI_8$.}
Since
\begin{equation*}
\begin{split}
\int_{\mathcal{V}_{k-1}}&\tilde{w}_{q,\ta}(v_{k-1})d\sigma _{k-1}\\
\leq&
\frac{1}{\sqrt{2\pi}} \int_{n(x_{k-1})\cdot v_{k-1}>0}(n(x_{k-1})\cdot v_{k-1})
e^{-\frac{1}{4}|v_{k-1}|^2-\frac{q}{4}|v_{k-1}|^{\ta}}dv_{k-1}\leq C_{q,\ta},
\end{split}
\end{equation*}
by applying \eqref{largek} in Lemma \ref{k}, we have
\begin{equation*}
\begin{split}
\CI_8\leq& C_{q,\ta}\int_{\prod_{j=1}^{k-2}%
\mathcal{V}_{j}}\mathbf{1}_{\{0<t_{k-1}\}}\Pi _{j=1}^{k-2} d\sigma _{j}
\sup_{0\leq s\leq t}\Vert h^j(s)\Vert _{\infty }\\
\leq&
C_{q,\ta}\left\{ \frac{1}{2}\right\} ^{C_{2}T_0^{5/4}}\sup_{0\leq s\leq t}\Vert h(s)\Vert _{\infty}.
\end{split}
\end{equation*}

We can not obtain the desired estimates for $\CI_2$ and $\CI_6$ for the time being and they will be treated by using iteration \eqref{iteration} for $h^j$ again.
To illustrate this more clearly, we first combine the above estimates for $\CI_1$, $\CI_3$, $\CI_4$, $\CI_5$, $\CI_7$ and $\CI_8$ to conclude that
\begin{equation}\label{hmain}
\begin{split}
|h^j(t,x,v)|
\leq& \left\{\mathbf{1}_{t_{1}\leq 0}
\int_{0}^{t}+\mathbf{1}_{t_{1}>0}
\int_{t_1}^{t}\right\}e^{-\nu
(v)(t-s)}|K^\chi_{w}h^j(s,x-(t-s)v,v)|ds \\
&+\frac{e^{-\nu (v)(t-t_{1})}}{\tilde{w}_{q,\ta}(v)}\times
\int_{\prod_{j=1}^{k-1}\mathcal{V}_{j}}\sum_{l=1}^{k-1}\bigg\{%
\int_{0}^{t_{l}}\mathbf{1}_{\{t_{l+1}\leq 0<t_{l}\}}|K^\chi_{w}h^j(s,X_{\mathbf{cl}}(s),v_{l})|   \\
&+\int_{t_{l+1}}^{t_{l}}\mathbf{1}_{\{0<t_{l+1}\}}|K^\chi_{w}h^j(s,X_{\mathbf{cl}}(s),v_{l})|\bigg\}d\Sigma _{l}(s)ds+A_{1 }(t)
\\&=\CI_2+\CI_6+A_{1 }(t),
\end{split}
\end{equation}%
where $A_{1 }(t)$ denotes
\begin{equation*}
\begin{split}
A_{1 }(t)=&
C_{q,\ta}\sup_{0\leq s\leq t}\left\Vert \nu^{-1}w_{q,\ta}g(s)\right\Vert
_{\infty }+C_{q,\ta}\Vert h(0)\Vert _{\infty}
\\&+C_{q,\ta}\left(\frac{1}{2}\right)^{C_2T_0^{5/4}}\sup_{0\leq s\leq t}\Vert h^j(s)\Vert _{\infty}
+C_{q,\ta}\eps^{3+\varrho}\sup_{0\leq s\leq t}\left\|h^j(s)\right\|_{\infty}.
\end{split}
\end{equation*}%
Recall the back-time cycles: $X_{\mathbf{cl}}(s)=\sum\limits_{l}{\bf 1}_{[t_{l+1},t_l)}(s)\{x_l-(t_l-s)v_l\}.$
Let $(t_{0}',x_{0}',v_{0}')=(s,X_{\mathbf{cl}}(s),v')$, for $v_{l'+1}'\in \mathcal {V}'_{l'+1}=\{v_{l'+1}'\cdot n(x_{l'+1}')>0\},$
we define a new back-time cycle as
$$
(t_{l'+1}',x_{l'+1}',v_{l'+1}')=(t_{l'}'-t_{\mathbf{b}}(x_{l'}',v_{l'}'),x_{\mathbf{b}}(x_{l'}',v_{l'}'),v_{l'+1}').
$$
We now
iterate (\ref{hmain}) to  get the representation for $K^\chi_{w}h^j(s,X_{\mathbf{cl}%
}(s),v_{l})$ as
\begin{equation}\label{diff3}
\begin{split}
K&^\chi_{w}h^j(s,X_{\mathbf{cl}}(s),v_{l})\\
\leq& \int_{
\mathbf{R}^{3}}\mathbf{k}^\chi_{w}(v_{l},v^{\prime })|h^j(s,X_{
\mathbf{cl}}(s),v^{\prime })|dv^{\prime }  \\
\leq& \iint \left\{\mathbf{1}_{t_{1}^{\prime }\leq 0}\int_{0}^{s}+\mathbf{1}_{t_{1}^{\prime }>0}\int_{t_{1}^{\prime }}^{s}\right\}e^{-\nu
(v^{\prime })(s-s_{1})} \mathbf{k}^\chi_{w}(v_{l},v^{\prime })
\mathbf{k}^\chi_{w}(v^{\prime },v^{\prime \prime })
\\&\qquad\qquad\times|h^j(s_{1},X_{\mathbf{cl}}(s)-(s-s_{1})v^{\prime },v^{\prime \prime
})|ds_{1}dv^{\prime }dv^{\prime \prime }\\
&+\iint dv^{\prime }dv^{\prime \prime }\int_{\prod_{j=1}^{k-1}\mathcal{V
}_{j}^{\prime }}\frac{e^{-\nu (v^{\prime })(s-t_{1}^{\prime })}}
{\tilde{w}_{q,\ta}(v^{\prime })}
\\&\quad\times\sum_{l^{\prime }=1}^{k-1}\int_{0}^{t_{l^{\prime
}}^{\prime }}ds_{1}\mathbf{1}_{\{t_{l^{\prime }+1}^{\prime }\leq
0<t_{l^{\prime }}^{\prime }\}}  \mathbf{k}^\chi_{w}(v_{l},v^{\prime })\\&\quad\times
\mathbf{k}^\chi_{w}(v_{l^{\prime }}^{\prime },v^{\prime \prime
})|h^j(s_{1,}x_{l^{\prime }}^{\prime
}+(s_{1}-t_{l^{\prime }}^{\prime })v_{l^{\prime }}^{\prime },v^{\prime
\prime })|d\Sigma^w_{l^{\prime }}(s_{1})  \\
&+\iint dv^{\prime }dv^{\prime \prime }\int_{\prod_{j=1}^{k-1}\mathcal{V
}_{j}^{\prime }}\frac{e^{-\nu (v^{\prime })(s-t_{1}^{\prime })}}
{\tilde{w}_{q,\ta}(v^{\prime })}
\\&\quad\times\sum_{l^{\prime }=1}^{k-1}\int_{t_{l^{\prime}+1}^{\prime }}^{t_{l^{\prime }}^{\prime }}ds_{1}
\mathbf{1}_{\{t_{l^{\prime}+1}^{\prime }>0\}}  \mathbf{k}^\chi_{w}(v_{l},v^{\prime })\\&\quad\times
\mathbf{k}^\chi_{w}(v_{l^{\prime }}^{\prime },v^{\prime \prime
})|h^j(s_{1},x_{l^{\prime }}^{\prime }+(s_{1}-t_{l^{\prime}}^{\prime })v_{l^{\prime }}^{\prime },v^{\prime \prime })|d\Sigma^w_{l^{\prime }}(s_{1}) \\
&+\int_{\R^{3}}
\mathbf{k}^\chi_{w}(v_{l},v^{\prime })dv^{\prime}A_{1}(s)=\sum\limits_{n=1}^4J_n,
\end{split}
\end{equation}
where $\mathbf{k}^\chi_{w}(\cdot)=w_{q,\ta}\mathbf{k}^\chi(\frac{\cdot}{w_{q,\ta}})$ and
$J_n$ $(1\leq n\leq4)$ denote the corresponding four terms on the right hand side of the last inequality.

In what follows, we only give an explicit computation for the delicate term $\CI_6$, the appropriate estimates for $\CI_2$ is similar and much easier and will be omitted for brevity.

\noindent{\it Estimates for $\CI_6$.} Substituting \eqref{diff3} into $\CI_6$, one has
\begin{equation}\label{I6}
\begin{split}
\CI_6\leq&C_{q,\ta}\int_{\prod_{j=1}^{k-1}
\mathcal{V}_{j}}\sum_{l=1}^{k-1}\left\{\mathbf{1}_{\{t_{l+1}\leq
0<t_{l}\}}\int_{0}^{t_l}+\mathbf{1}_{\{0<t_{l+1}\}}\int_{t_{l+1}}^{t_{l}}
\right\}\sum\limits_{n=1}^4~J_n~d\Sigma^w_{l}(s)ds.
\end{split}
\end{equation}
Continuing, we first consider the simpler terms involving $A_{1}(s)$ in $\CI_6$, that is, the terms containing $J_4$.
Since
$\int \mathbf{k}^\chi_{w}(v_{l},v ^{\prime })dv ^{\prime }<\infty,$
the
summation of all contributions of $A_{1}$'s leads to the bound
\begin{equation*}
\begin{split}
\int_{\prod_{j=1}^{k-1}\mathcal{V}_j}&
\left\{\sum_{l=1}^{k-1}\mathbf{1}_{\{t_{l+1}\leq
0<t_{l}\}} \int_0^{t_l}
A_{1}(s)
+\mathbf{1}_{\{0<t_{l}\}}\int_{t_{l+1}}^{t_l} A_{1}(s)
\right\}
d\Sigma^w_{l} ds
\leq CA_{1}(t),
\end{split}
\end{equation*}
according to Lemma \ref{k}.


Next we only compute the terms containing $J_2$ and $J_3$, because the estimates for the terms involving $J_1$ is similar and easier.
Let us first show that there exists constant $N>0$ such that
\begin{equation}\label{J3}
\begin{split}
&\int_{\prod_{j=1}^{k-1}
\mathcal{V}_{j}}\sum_{l=1}^{k-1}\int_{0}^{t_l}\mathbf{1}_{\{t_{l+1}\leq
0<t_{l}\}}~J_3~d\Sigma^w_{l}(s)ds\\
=&\int_{\prod_{j=1}^{k-1}
\mathcal{V}_{j}}\sum_{l=1}^{k-1}\mathbf{1}_{\{t_{l+1}\leq
0<t_{l}\}}\int_{0}^{t_l}\iint dv^{\prime }dv^{\prime \prime }\\&\quad\times\int_{\prod_{j=1}^{k-1}\mathcal{V
}_{j}^{\prime }}\frac{e^{-\nu (v^{\prime })(s-t_{1}^{\prime })}}
{\tilde{w}_{q,\ta}(v^{\prime })}\sum_{l^{\prime }=1}^{k-1}\int_{t_{l^{\prime}+1}^{\prime }}^{t_{l^{\prime }}^{\prime }}
\mathbf{1}_{\{t_{l^{\prime}+1}^{\prime }>0\}} \\&\quad\times \mathbf{k}^\chi_{w}(v_{l},v^{\prime })
\mathbf{k}^\chi_{w}(v_{l^{\prime }}^{\prime },v^{\prime \prime
})|h^j(s_{1,}x_{l^{\prime }}^{\prime
}+(s_{1}-t_{l^{\prime }}^{\prime })v_{l^{\prime }}^{\prime },v^{\prime
\prime })|d\Sigma^w_{l^{\prime }}(s_{1})ds_{1}d\Sigma^w_{l}(s)ds
\\ \leq&C_{q,\ta}\left(\frac{1}{T_0^{5/4}}+\frac{1}{N}\right)\sup\limits_{0\leq s\leq t}\|h^j(s)\|_\infty+C_N\sup\limits_{0\leq s\leq t_1}\left\{e^{\frac{\la_0}{2}s^{\rho_0}}
\left\Vert \frac{h^j(s)}{w_{q,\ta }(v)}\right\Vert _{2}\right\}.
\end{split}
\end{equation}
To prove \eqref{J3}, we decompose the velocity-time integration into several regions and treat them independently.
Recalling that $\{(t_{l'}', x_{l'}',v_{l'}')\}_{l'=1}^{k}$ start from $(s,X_{\mathbf{cl}},v')$, in order to avoid confusion, let us denote
\begin{equation}\label{ks}
k(s)=k=C'_{1}[\al(s)]^{5/4}.
\end{equation}
For any $1\leq l'\leq k-1$, we consider the following splitting
$$
\int_{t_{l^{\prime}+1}}^{t_{l'}^{\prime }}
=\int_{t_{l^{\prime}+1}}^{t_{l'}^{\prime }-\frac{1}{k^2(s)}}+\int_{t_{l'}^{\prime }-\frac{1}{k^2(s)}}^{t_{l'}^{\prime }},
$$
and
treat the second integral first, specifically, we intend to obtain
\begin{equation}\label{J33}
\begin{split}
\int_{\prod_{j=1}^{k-1}
\mathcal{V}_{j}}&\sum_{l=1}^{k-1}\mathbf{1}_{\{t_{l+1}\leq0<t_{l}\}}
\int_{0}^{t_l}\iint dv^{\prime }dv^{\prime \prime }\int_{\prod_{j=1}^{k-1}\mathcal{V
}_{j}^{\prime }}\frac{e^{-\nu (v^{\prime })(s-t_{1}^{\prime })}}
{\tilde{w}_{q,\ta}(v^{\prime })}\sum_{l^{\prime }=1}^{k-1}\int_{t_{l^{\prime}}^{\prime }-\frac{1}{k^2(s)}}^{t_{l^{\prime }}^{\prime }}
\\&\quad\times \mathbf{1}_{\{t_{l^{\prime}+1}^{\prime }>0\}} \mathbf{k}^\chi_{w}(v_{l},v^{\prime })
\mathbf{k}^\chi_{w}(v_{l^{\prime }}^{\prime },v^{\prime \prime
})
|h^j(s_{1,}x_{l^{\prime }}^{\prime
}+(s_{1}-t_{l^{\prime }}^{\prime })v_{l^{\prime }}^{\prime },v^{\prime
\prime })|\\&\quad\times  d\Sigma^w_{l^{\prime }}(s_{1})ds_{1}d\Sigma^w_{l}(s)ds\\
\leq &\frac{C_{q,\ta}}{T_0^{5/4}}\sup\limits_{0\leq s\leq t}\|h^j(s)\|_{L^\infty}.
\end{split}
\end{equation}
Indeed, since $t_{l'}'-s_1\leq 1/k^2(s)$, and $k(s)\geq C'_1T_0^{5/4}$, it follows from Lemma \ref{es.k} and \eqref{ktildew2} that the right hand side of \eqref{J33} is bounded by
\begin{equation*}
\begin{split}
C_{q,\ta}&\int_{\prod_{j=1}^{k-1}
\mathcal{V}_{j}}\sum_{l=1}^{k-1}\mathbf{1}_{\{t_{l+1}\leq
0<t_{l}\}}\int_{0}^{t_l}d\Sigma^w_{l}(s)\frac{1}{k^2(s)}\\&\times\sup_{0\leq s_1\leq s}
\int_{\prod_{j=1}^{k-1}\mathcal{V
}_{j}^{\prime }}\sum_{l^{\prime }=1}^{k-1}\mathbf{1}_{\{t_{l^{\prime}+1}^{\prime }>0\}}
d\Sigma^w_{l'}(s_1)ds\sup\limits_{0\leq s_1\leq t_1}
\|h^j(s_{1})\|_{L^\infty}\\
\leq&C_{q,\ta}\int_{\prod_{j=1}^{k-1}
\mathcal{V}_{j}}\sum_{l=1}^{k-1}\mathbf{1}_{\{t_{l+1}\leq
0<t_{l}\}}\int_{0}^{t_l}d\Sigma^w_{l}(s)\frac{1}{k(s)}ds\sup\limits_{0\leq s_1\leq t_1}
\|h^j(s_{1})\|_{L^\infty}\\
\leq&\frac{C_{q,\ta}}{T_0^{5/4}}\sup\limits_{0\leq s\leq t_1}
\|h^j(s)\|_{\infty},
\end{split}
\end{equation*}
where we also used the following significant estimate
\begin{equation*}
\begin{split}
\sup_{0\leq s_1\leq s}\int_{\prod_{j=1}^{k-1}\mathcal{V
}_{j}^{\prime }}\sum_{l^{\prime }=1}^{k-1}\mathbf{1}_{\{t_{l^{\prime}+1}^{\prime }>0\}}
d\Sigma^w_{l'}(s_1)
\leq&C_{q,\ta}k(s).
\end{split}
\end{equation*}
As to the first integral
\begin{equation}\label{J34}
\begin{split}
\int_{\prod_{j=1}^{k-1}
\mathcal{V}_{j}}&\sum_{l=1}^{k-1}\mathbf{1}_{\{t_{l+1}\leq
0<t_{l}\}}\int_{0}^{t_l}\iint dv^{\prime }dv^{\prime \prime }\int_{\prod_{j=1}^{k-1}\mathcal{V
}_{j}^{\prime }}\frac{e^{-\nu (v^{\prime })(s-t_{1}^{\prime })}}
{\tilde{w}_{q,\ta}(v^{\prime })}\sum_{l^{\prime }=1}^{k-1}\mathbf{1}_{\{t_{l^{\prime}+1}^{\prime }>0\}}\\&\quad\times\int_{t_{l^{\prime}+1}^{\prime }}^{t_{l^{\prime }}^{\prime }-\frac{1}{k^2(s)}}
  \mathbf{k}^\chi_{w}(v_{l},v^{\prime })
\mathbf{k}^\chi_{w}(v_{l^{\prime }}^{\prime },v^{\prime \prime
})|h^j(s_{1,}x_{l^{\prime }}^{\prime
}+(s_{1}-t_{l^{\prime }}^{\prime })v_{l^{\prime }}^{\prime },v^{\prime
\prime })|\\&\quad\times d\Sigma^w_{l^{\prime }}(s_{1})ds_{1}d\Sigma^w_{l}(s)ds,
\end{split}
\end{equation}
we divide our computations into following three cases:

\noindent{\bf Case 1.} $|v_l|\geq N$ or $|v_{l'}'|\geq N$ with $N$ suitably large.
From Lemma \ref{es.k}, it follows that
\begin{equation*}
\begin{split}
&\int\mathbf{k}^\chi_{w}(v_{l},v^{\prime })dv^{\prime}\leq \frac{C_\eps}{(1+|v_l|)^{-\varrho}}
\leq \frac{C_\eps}{N},\\
&\nu^{-1}(v'_{l'})\int\mathbf{k}^\chi_{w}(v_{l^{\prime }}^{\prime },v^{\prime \prime})dv^{\prime \prime }\leq \frac{C_\eps}{(1+|v_{l'}'|)^{-\varrho}}\leq \frac{C_\eps}{N}.
\end{split}
\end{equation*}
Using this, for $|v_l|\geq N$ or $|v_{l'}'|\geq N$, we know thanks to \eqref{ktildew1} and \eqref{ktildew2} that
\begin{equation}\label{J31}
\begin{split}
\eqref{J34}\leq& \frac{C_\eps C_{q,\ta}}{N}\int_{\prod_{j=1}^{k-1}
\mathcal{V}_{j}}\sum_{l=1}^{k-1}\mathbf{1}_{\{t_{l+1}\leq
0<t_{l}\}}\int_0^{t_l}d\Sigma^w_l(s)
\\&\times\int_{\prod_{j=1}^{k-1}\mathcal{V
}_{j}^{\prime }}\sum_{l^{\prime }=1}^{k-1}\mathbf{1}_{\{t_{l^{\prime}+1}^{\prime }>0\}}\int_{t_{l^{\prime}+1}^{\prime }}^{t_{l^{\prime }}^{\prime }}d\Sigma^w_{l'}(s_1)ds_{1}ds
\sup\limits_{0\leq s\leq t_1}
\|h^j(s)\|_{\infty}
\\ \leq&\frac{C_{\eps,q,\ta}}{N}\sup\limits_{0\leq s\leq t}
\|h^j(s)\|_{\infty}.
\end{split}
\end{equation}

\noindent{\bf Case 2.} $|v_l|\leq N$ and $|v'|\geq 2N$, or $|v'_{l'}|\leq N$ and $|v''|\geq 2N$. Notice that we have either
$|v_l-v'|\geq N$ or $|v'_{l'}-v''|\geq N$, and either of the following holds correspondingly
\begin{equation*}
\begin{split}
&\mathbf{k}^\chi_{w}(v_{l},v^{\prime })
\leq Ce^{-\frac{\vps N^2}{16}}\mathbf{k}^\chi_{w}(v_{l},v^{\prime })e^{\frac{\vps |v_l-v'|^2}{16}},\\
&\mathbf{k}^\chi_{w}(v_{l^{\prime }}^{\prime },v'')
\leq Ce^{-\frac{\vps N^2}{16}}\mathbf{k}^\chi_{w}(v'_{l'},v^{\prime\prime })e^{\frac{\vps |v'_{l'}-v''|^2}{16}}.
\end{split}
\end{equation*}
By virtue of Lemma \ref{es.k}, one sees that both of
$$\int\mathbf{k}^\chi_{w}(v_{l},v^{\prime })e^{\frac{\vps |v_l-v'|^2}{16}}\ \textrm{and}\ \int\mathbf{k}^\chi_{w}(v'_{l'},v^{\prime\prime })e^{\frac{\vps |v'_{l'}-v''|^2}{16}}$$
are still bounded. In this situation,
we have by a similar argument as for obtaining \eqref{J31} that
\begin{equation}\label{J32}
\begin{split}
\eqref{J34}\leq C_{q,\ta}e^{-\frac{\vps N^2}{16}}\sup\limits_{0\leq s\leq t}
\|h^j(s)\|_{\infty}.
\end{split}
\end{equation}
To obtain the final bound for \eqref{J34}, we are now in a position to handle the last case:

\noindent{\bf Case 3.} $|v_l|\leq N$, $|v'|\leq 2N$, $|v_{l'}'|\leq N$ and $|v''|\leq 2N$. 
Recall there is a lower bound $t_{l'}'-s_1\geq 1/k^2$, so that one can convert the bound in $L^\infty-$norm to the one in $L^2-$norm which has been well-established in Section \ref{L2th}. To do so, for any large $N>0$,
we choose a number $m(N)$ to define
\begin{equation}
\mathbf{k}^\chi_{m,w}(p,v')\equiv \mathbf{1}
_{|p-v^{\prime }|\geq \frac{1}{m},|v^{\prime}|\leq m}\mathbf{k}^\chi_{w}(p,v'),
\label{km}
\end{equation}%
such that $\sup_{p}\int_{\mathbf{R}^{3}}|\mathbf{k}^\chi
_{m}(p,v^{\prime})
-\mathbf{k}^\chi_{w}(p ,v^{\prime})|dv^{\prime}\leq
\frac{1}{N}.$ We split
\begin{equation*}
\begin{split}
\mathbf{k}^\chi_{w}(v_l,v')
\mathbf{k}_{w}^\chi(v_{l^{\prime }}^{\prime },v^{\prime \prime})=&\{\mathbf{k}^\chi_{w}(v_{l},v^{\prime})
-\mathbf{k}^\chi_{m,w}(v_{l},v^{\prime})\}\mathbf{k}^\chi_{w}(v_{l^{\prime }}^{\prime },v^{\prime \prime })
\\&+\{\mathbf{k}^\chi_{w}(v'_{l'},v'')
-\mathbf{k}^\chi_{m,w}(v'_{l'},v'')\}\mathbf{k}^\chi_{m,w}(v_{l},v^{\prime})
\\&+\mathbf{k}^\chi_{m,w}(v_{l},v^{\prime})\mathbf{k}^\chi_{m,w}(v_{l^{\prime }}^{\prime },v^{\prime \prime }),
\end{split}
\end{equation*}
and from Lemma \ref{k}, the first two difference leads to a small
contribution in \eqref{J34}
\begin{equation}
\frac{C_{q,\ta }}{N}\sup_{0\leq s\leq t}\Vert h^j(s)\Vert _{\infty }.  \label{epsilon2}
\end{equation}
For the remaining main contribution of $\mathbf{k}^\chi_{m,w}(v_{l},v^{\prime})\mathbf{k}^\chi_{m,w}(v_{l^{\prime }}^{\prime },v^{\prime \prime })$,
by a change of
variable $y=x_{l^{\prime }}^{\prime }+(s_{1}-t_{l^{\prime }}^{\prime })v_{l^{\prime }}$
and notice that $x_{l'}'$ is independent of
$v_{l}^{\prime }$, we see that $\left\vert \frac{dy}{dv_{l}^{\prime }}%
\right\vert \geq (k(s))^{-6}$. Consequently, as in Case 4 in the proof of Theorem 6 in \cite[pp. 754]{Guo-2010}, we obtain
\begin{equation}\label{ks-6}
\begin{split}
\eqref{J34}\leq&\frac{C_{q,\ta}}{N}\sup\limits_{0\leq s\leq t}
\|h^j(s)\|_{\infty}\\&+C_{N}\int_{\prod_{j=1}^{k-1}
\mathcal{V}_{j}}\sum_{l=1}^{k-1}\mathbf{1}_{\{t_{l+1}\leq
0<t_{l}\}}\int_{0}^{t_l}d\Sigma^w_l(s)(k(s))^6
\\&\quad\times\int_{\prod_{j=1}^{l'-1}\mathcal{V
}_{j}^{\prime }\prod_{j=l'+1}^{k-1}\mathcal{V
}_{j}^{\prime }}\sum_{l^{\prime }=1}^{k-1}\mathbf{1}_{\{t_{l^{\prime}+1}^{\prime }>0\}}\int_{t_{l^{\prime}+1}^{\prime }}^{t_{l^{\prime }}^{\prime }}\int_{\Omega\times \{|v''|\leq 2N
\}}
\left| \frac{h^j(s_1)}{w_{q,\ta }(v)}\right|dydv''\\&\quad\times e^{-\la_0(t'_{l'}-s_1)^{\rho_0}}\{\Pi _{j'=l'+1}^{k-1}d\sigma _{j'}\}\times \Pi _{j'=1}^{l'-1}\{{{%
e^{\nu (v'_{j'})(t'_{j'+1}-t'_{j'})} d\sigma _{j'}}}\}ds_1ds.
\end{split}
\end{equation}
In light of Lemma \ref{k.d} in Section \ref{non.de}, we see that
\eqref{ks-6} can be further dominated by
\begin{equation}\label{I634}
\begin{split}
\eqref{J34}\leq&\frac{C_{q,\ta}}{N}\sup\limits_{0\leq s\leq t}
\|h^j(s)\|_{\infty}\\&+C_{N}\int_{\prod_{j=1}^{k-1}
\mathcal{V}_{j}}\sum_{l=1}^{k-1}\mathbf{1}_{\{t_{l+1}\leq
0<t_{l}\}}\int_{0}^{t_l}d\Sigma^w_l(s)
(k(s))^7e^{-\frac{\la_0}{2}s^{\rho_0}}ds\\
&\quad\times\sup\limits_{0\leq s\leq t_1}\left\{e^{\frac{\la_0}{2}s^{\rho_0}}
\left\Vert \frac{h^j(s)}{w_{q,\ta }(v)}\right\Vert _{2}\right\}\\
\leq&\frac{C_{q,\ta}}{N}\sup\limits_{0\leq s\leq t}\|h^j(s)\|_{\infty}+C_N\sup\limits_{0\leq s\leq t_1}\left\{e^{\frac{\la_0}{2}s^{\rho_0}}
\left\Vert \frac{h^j(s)}{w_{q,\ta }(v)}\right\Vert _{2}\right\}.
\end{split}
\end{equation}
Putting all the above estimates \eqref{J33}, \eqref{J31}, \eqref{J32}, \eqref{epsilon2} and \eqref{I634} together, one sees that \eqref{J3} is valid.
Furthermore, by a similar argument as proving \eqref{J3}, we can also show that the remaining terms in \eqref{I6} and $\CI_2$ share the same bound as \eqref{J3},
we thus arrive at
\begin{equation*}
\begin{split}
\CI_{2},\ \CI_{6}\leq& C_{q,\ta}\left(\frac{1}{T_0^{5/4}}+\frac{1}{N}\right)\sup\limits_{0\leq s\leq t}
\|h^j(s)\|_{\infty}\\&+C_N\sup\limits_{0\leq s\leq t}\left\{e^{\frac{\la_0}{2}s^{\rho_0}}
\left\Vert \frac{h^j(s)}{w_{q,\ta }(v)}\right\Vert _{2}\right\}+CA_{1}(t).
\end{split}
\end{equation*}
Now choose $T_0, N>0$ large, and plug the estimates for $\CI_2$, $\CI_6$ and $A_{1}(t)$ into \eqref{hmain} to obtain
\begin{equation}\label{es.hlf1}
\begin{split}
\sup_{0\leq s\leq t}\Vert h^j(s)\Vert
_{\infty } \leq&
C\Vert h_0\Vert _{\infty }+C\sup_{0\leq s\leq t}\Vert \nu^{-1} w_{q,\ta}g(s)\Vert
_{\infty }\\&+C\sup\limits_{0\leq s\leq t}\left\{e^{\frac{\la_0}{2}s^{\rho_0}}
\left\Vert f^j(s)\right\Vert _{2}\right\}.
\end{split}
\end{equation}
On the other hand, from \eqref{l-decay} in Proposition \eqref{l2-lqn}, one has by taking $\la_0\leq\la_1$
\begin{equation}\label{fj2,bd}
\begin{split}
\sup\limits_{0\leq s\leq t}&\left\{e^{\frac{\la_0}{2}s^{\rho_0}}
\left\Vert f^j(s)\right\Vert _{2}\right\}\\
\leq&C\|w_{q/2,\ta}f_0\|_2+C\sqrt{\int_{0}^{t}e^{\lambda_1 s^{\rho_0}}\Vert \nu^{-1/2}g(s)\Vert
_{2}^{2}ds}\\&+C\sqrt{\int_{0}^{t}\Vert\nu^{-1/2} w_{q/2,\ta}g(s)\Vert
_{2}^{2}ds}.
\end{split}
\end{equation}
\eqref{es.hlp1} then follows from \eqref{es.hlf1}, \eqref{fj2,bd}
and \eqref{ape1}.
This allows us to construct a global solution $h^j(t,x,v)$ to \eqref{hj.eqn} and \eqref{hj.bd}.
Since \eqref{es.hlp1} is uniform in $j$, one can further show that $\{h^j\}_{j=2}^\infty$ converges to $h$ strongly in $L^\infty$ via the similar argument used in the end of the proof for Lemma \ref{lex.ow}. Finally, by \eqref{es.hlp1}, we also have
\begin{equation}\label{es.whif}
\begin{split}
\sup_{0\leq s\leq t}\left\{\Vert h(s)\Vert _{\infty }+|h(s)|_{\infty }\right\}\lesssim& \Vert w_{q,\ta} f_0\Vert _{\infty }+\sup_{0\leq s\leq t}\Vert \nu^{-1} w_{q,\ta}g(s)\Vert
_{\infty }\\
&+\|w_{q/2,\ta}f_0\|_2+\sqrt{\int_{0}^{t}e^{\lambda_1 s^{\rho_0}}\Vert \nu^{-1/2}g(s)\Vert
_{2}^{2}ds}\\&+\sqrt{\int_{0}^{t}\Vert\nu^{-1/2} w_{q/2,\ta}g(s)\Vert
_{2}^{2}ds}.
\end{split}
\end{equation}
\eqref{es.whif} and the inequality $\|w_{q/2,\ta}f_0\|_2\lesssim \Vert w_{q,\ta} f_0\Vert _{\infty }$ imply \eqref{es.fif}.
This completes the proof of Proposition
\ref{lLif}.
\end{proof}

\subsection{Nonlinear existence}\label{non.ex}
Our aim in this subsection is to prove
\begin{proof}[The global existence of \eqref{BE}, \eqref{ID} and \eqref{DBD}.]
Recall the initial boundary value problem for the linearized equation \eqref{dlinear} and \eqref{lbd}, we design
the following iteration sequence
\begin{equation}\label{nn.it}
\partial _{t}f^{\ell +1}+v\cdot \nabla _{x}f^{\ell +1}+Lf^{\ell +1}=\Gamma (f^{\ell },f^{\ell }),\ f^{\ell+1}(0,x,v)=f_0(x,v),
\end{equation}%
with $f_{{-}}^{\ell +1}=P_{\gamma }f^{\ell}$ and $f^{0}=f_0(x,v)$.
Clearly $\FP\{\Gamma (f^{\ell },f^{\ell })\}=0 $.

Note that the iteration scheme \eqref{nn.it} does not provide us the positivity of the solution of the original equation \eqref{be}, however it coincides with the linearized equation \eqref{dlinear} so that Propositions \ref{l2-lqn}
and \ref{lLif} can be directly used. Our strategy to prove the global existence \eqref{BE}, \eqref{ID} and \eqref{DBD} can be outlined as follows:
we first show that the sequence $\{f^{\ell}\}_{\ell=0}^{\infty}$ determined by \eqref{nn.it} is well-defined in a suitable Banach space via Propositions \ref{l2-lqn}
and \ref{lLif}, then we prove that such a sequence is in fact a Cauchy sequence and the limit is a desired global solution.
Let us now define the following energy functional
$$
\CE(f)(t)=\|w_{q,\ta}f\|^2_{\infty}+|w_{q,\ta}f|^2_{\infty,+}+e^{\la_1t^{\rho_0}}\|f\|_2^2+\|w_{q/2,\ta}f\|^2_{2},
$$
and dissipation rate
$$
\mathcal {D}(f)(t)=\|w_{q/2,\ta}f\|_{\nu}^2+e^{\la_1t^{\rho_0}}\|f\|_{\nu}^2.
$$
For later use, we also define a Banach space
$$
\FX_\de=\left\{f~|~\sup\limits_{0\leq s\leq t}\CE(f)(s)+\int_0^t\mathcal {D}(f)(s)ds<\de,\ \ \de>0\right\},
$$
associated with the norm
$$
\FX_\de(f)(t)=\sup\limits_{0\leq s\leq t}\CE(f)(s)+\int_0^t\mathcal {D}(f)(s)ds.
$$
We now show that $f^{\ell+1}\in\FX_\de$ if $f^{\ell}\in\FX_\de$. For this,
on the one hand, we know from \eqref{es.fif}, 
\eqref{wl2} and \eqref{l-decay} 
with $f=f^{\ell+1}$ and
$g=\Gamma (f^{\ell },f^{\ell })$ that \eqref{nn.it} admits a unique solution $f^{\ell+1}$ satisfying
\begin{equation}\label{pro24}
\begin{split}
\sup\limits_{0\leq s\leq t}&\CE(f^{\ell+1})(s)+\int_0^t\mathcal {D}(f^{\ell+1})(s)ds\\
\leq& C\CE(f)(0)+ C\sup\limits_{0\leq s\leq t}\|\nu^{-1}w_{q,\ta}\Gamma (f^{\ell },f^{\ell })(s)\|^2_\infty
\\
&+C\int_{0}^t\left\|\nu^{-1/2}w_{q/2,\ta}\Gamma (f^{\ell },f^{\ell })(s)\right\|_2^2ds+C\int_{0}^{t}e^{\lambda_1 s^{\rho_0}}\Vert \nu^{-1/2}\Gamma (f^{\ell },f^{\ell })(s)\Vert
_{2}^{2}ds,
\end{split}
\end{equation}
provided the right hand side is finite.
On the another hand, thanks to Lemma \ref{es.nop}, it follows
\begin{equation}\label{es.nopl1}
\begin{split}
\int_{0}^t\left\|\nu^{-1/2}w_{q/2,\ta}\Gamma (f^{\ell },f^{\ell })\right\|_2^2ds
\leq& C\sup\limits_{0\leq s\leq t}\|w_{q,\ta}f(s)\|^2_{\infty}\int_{0}^t\|w_{q/2,\ta}f(s)\|^2_{\nu}ds\\
\leq& C\sup\limits_{0\leq s\leq t}\CE(f^\ell)(s)\int_{0}^t\mathcal {D}(f^\ell)(s)ds,
\end{split}
\end{equation}
\begin{equation}\label{es.nopl2}
\begin{split}
\int_{0}^{t}e^{\lambda_1 s^{\rho_0}}\Vert \nu^{-1/2}\Gamma (f^{\ell },f^{\ell })(s)\Vert
_{2}^{2}ds \leq& C\sup\limits_{0\leq s\leq t}\|w_{q/2,\ta}f(s)\|^2_{\infty}\int_{0}^te^{\lambda_1 s^{\rho_0}}\|f(s)\|^2_{\nu}ds\\
\leq& C\sup\limits_{0\leq s\leq t}\CE(f^\ell)(s)\int_{0}^t\mathcal {D}(f^\ell)(s)ds,
\end{split}
\end{equation}
and
\begin{equation}\label{es.nopl3}
\sup\limits_{0\leq s\leq t}\left\|\nu^{-1}w_{q,\ta}\Gamma (f^{\ell },f^{\ell })\right\|_{\infty}\leq C\sup\limits_{0\leq s\leq t}\CE(f^\ell)(s).
\end{equation}
As a consequence, one has from \eqref{pro24}, \eqref{es.nopl1}, \eqref{es.nopl2} and \eqref{es.nopl3} that
\begin{equation}\label{X1}
\FX_\de(f^{\ell+1})(t)\leq C\CE(f)(0)+C\FX_\de^2(f^{\ell})(t),
\end{equation}
which further yields $\FX_\de(f^{\ell+1})(t)<\de$ supposing $f^{\ell}\in\FX_\de$ with $\de$ and $\CE(f)(0)$ small enough.

We now prove the strong convergence of the iteration sequence $\{f^{\ell}\}_{\ell=0}^{\infty}$ constructed above. To do this,
by taking difference of the equations that $f^{\ell +1}$ and $f^{\ell }$ satisfy, we deduce that%
\begin{eqnarray*}
\left\{\begin{array}{rll}
&\partial _{t}[f^{\ell +1}-f^{\ell }]+v\cdot \nabla _{x}[f^{\ell
+1}-f^{\ell }]+L[f^{\ell +1}-f^{\ell }] \\[2mm]&\qquad=\Gamma (f^{\ell
}-f^{\ell -1},f^{\ell })+\Gamma (f^{\ell -1},f^{\ell }-f^{\ell -1}), \\[2mm]
&[f^{\ell +1}-f^{\ell }]_{-}=P_{\gamma }[f^{\ell +1}-f^{\ell }],
\end{array}\right.
\end{eqnarray*}%
with $f^{\ell +1}-f^{\ell }=0$ initially. Repeating the same argument as for obtaining \eqref{X1}, we
obtain
\begin{equation*}\label{Xfmin}
\FX_\de(f^{\ell+1}-f^{\ell})(t)\leq C\left\{\FX_\de(f^\ell)+\FX_\de(f^{\ell-1})\right\}\FX_\de(f^{\ell}-f^{\ell-1})(t).
\end{equation*}
This implies that $\{f^{\ell}\}_{\ell=0}^{\infty}$ is a Cauchy sequence in $\FX_\de$ for $\de$ suitably small. Moreover, take $f$ as the limit of the sequence $\{f^{\ell}\}_{\ell=0}^{\infty}$ in $\FX_\de$, then $f$ satisfies
\begin{equation}\label{sol.es}
\sup\limits_{0\leq s\leq t}\CE(f)(s)+\int_0^t\mathcal {D}(f)(s)ds
\leq C\CE(f)(0)\leq C\|w_{q,\ta}f_0\|^2_{\infty}.
\end{equation}

Since we have $L^\infty$ convergence at each step, as \cite[pp.788]{Guo-2010}, we deduce that $w_{q,\ta}f$ is continuous away from $\ga_0$
when $\Omega$ is strictly convex.
The uniqueness is
standard. We now turn to prove the positivity of $\mu+\sqrt{\mu}f$. As mentioned at the beginning of this subsection, we need to design a different iterative sequence. We
use the following one:
\begin{eqnarray*}
\left\{\begin{array}{rll}
&\left\{\partial_t+v\cdot\nabla_x \right\}F^{\ell+1}+F^{\ell+1}(v)\nu(F^{\ell})\\&\qquad\qquad\qquad=\int_{\R^3\times \S^2}|v-u|^\varrho b_0(\theta)
F^\ell
(u')F^\ell(v')\,du d\omega=\Ga^{\textrm{gain}}(F^\ell,F^\ell),\\[2mm]
&F_{{-}}^{\ell +1}=\mu\int_{n(x)\cdot v>0}F^{\ell }(v)n(x)\cdot vdv,\\[2mm]
&F^{\ell+1}(0,x,v)=F_0(x,v),
\end{array}\right.
\end{eqnarray*}
starting with $F^0(t,x,v)=F_0(x,v)$, here $\nu(F^{\ell})=\int_{\R^3\times \S^2}|v-u|^\varrho b_0(\theta)
F^\ell(u)dud\omega$. By a similar procedure as the proof of Theorem 4 in \cite[pp.806-807]{Guo-2010}, one can easily verify that
such an iteration preserves the non-negativity. We now need to prove $F^{\ell }$
is convergent to conclude the non negativity of the limit $F(t)\geq 0$.
Noticing that $F^{\ell+1}=\mu+\mu^{1/2}f^{\ell+1}$, equivalently we need to solve $f^{\ell+1}$ such that
\begin{equation}\label{iterate.f}
\begin{split}
&\left\{\partial_t+v\cdot\nabla_x +\nu\right\}f^{\ell+1}-Kf^{\ell}=\Gamma^{\text{gain}}(f^{\ell},f^{\ell})-f^{\ell+1}(v)\nu(\sqrt{\mu}f^\ell),\\
&f_{{-}}^{\ell +1}=P_\ga f^\ell,\ \
f^{\ell+1}(0,x,v)=f_0(x,v).
\end{split}
\end{equation}
In fact, since $|\nu(\sqrt{\mu}f^\ell)|\leq C\vps_0 \nu$ for $\|w_{q,\ta}f^\ell\|_{\infty}\leq\vps_0$, one can rewrite \eqref{iterate.f} as
\begin{equation*}
\begin{split}
&\left\{\partial_t+v\cdot\nabla_x +\overline{\nu}\right\}f^{\ell+1}=Kf^{\ell}+\Gamma^{\text{gain}}(f^{\ell},f^{\ell}),\\
&f_{{-}}^{\ell +1}=P_\ga f^\ell,\ \
f^{\ell+1}(0,x,v)=f_0(x,v),
\end{split}
\end{equation*}
with $\overline{\nu}=\nu+\nu(\sqrt{\mu}f^\ell)$. 
As the proof of Lemma \ref{lg.ex},
it follows from a routine procedure to show that $\{h^{\ell+1}=w_{q,\ta}f^{\ell+1}\}_{\ell=0}^{\infty}$ is indeed convergent in $L^\infty$ local in time $[0,T_{\ast}]$.
This ends the proof of the first part of Theorem \ref{ms1}. We leave the second part to the next subsection.
\end{proof}

\subsection{Nonlinear $L^\infty$ exponential decay}\label{non.de}
In this subsection, we are going to deduce the $L^\infty$ exponential time decay rates for the initial boundary value problem \eqref{BE}, \eqref{ID} and \eqref{DBD} based on the global existence constructed in Section \ref{non.ex}. For this, let us first present the following refined estimates for integrals on the stochastic cycles given by Definition \ref{diffusecycles}.
\begin{lemma}\label{k.d}
Denote $\|\cdot\|_{\FY}=\|\cdot\|_{2}\ \textrm{or}\ \|\cdot\|_{\infty}$. Assume $(q,\ta)\in\CA_{q,\ta}$.
There exists constant $\la_0>0$ such that for $\rho_0=\frac{\ta}{\ta-\varrho}$
\begin{equation}\label{kd1}
\begin{split}
\int_{\Pi_{j=1}^{k-1}\mathcal{V}_{j}}&\sum_{l=1}^{k-1}\int_0^{t_l}
\mathbf{1}_{\{t_{l+1}\leq 0<t_{l}\}}\|f(s,\cdot,v_l)\|_{\FY}d\Sigma _{l}(s)ds\\
\leq& Ce^{-\frac{\la_0}{2} t_1^{\rho_0}}\sup\limits_{0\leq s\leq t_1}e^{\frac{\la_0}{2} s^{\rho_0}}\|f(s)\|_{\FY},
\end{split}
\end{equation}
and
\begin{equation}\label{kd2}
\begin{split}
\int_{\Pi _{j=1}^{k-1}\mathcal{V}_{j}}\sum_{l=1}^{k-1}\int_{t_{l+1}}^{t_l}\mathbf{1}
_{\{t_{l+1}>0\}}\|f(s,\cdot,v_l)\|_{\FY}d\Sigma _{l}(s)ds
\leq
Ce^{-\frac{\la_0}{2} t_1^{\rho_0}}\sup\limits_{0\leq s\leq t_1}e^{\frac{\la_0}{2} s^{\rho_0}}\|f(s)\|_\FY,
\end{split}
\end{equation}
where $C>0$ and independent of $k$.

Moreover, for any $\eps_0>0$, it holds that
\begin{equation}\label{kd11}
\begin{split}
\int_{\Pi _{j=1}^{k-1}\mathcal{V}_{j}}\int_{t_l-\eps_0}^{t_l}
\|f(s,\cdot,v_l)\|_{\FY}d\Sigma _{l}(s)ds
\leq C\eps_0e^{-\frac{\la_0}{2} t_1^{\rho_0}}\sup\limits_{0\leq s\leq t_1}e^{\frac{\la_0}{2} s^{\rho_0}}\|f(s,\cdot,v_l)\|_{\FY},
\end{split}
\end{equation}
\begin{equation}\label{kd21}
\begin{split}
\int_{\prod_{j=1}^{l-1}\mathcal{V
}_{j}\prod_{j=l+1}^{k-1}\mathcal{V
}_{j}}&\mathbf{1}_{\{t_{l+1}>0\}}\int_{t_{l+1}}^{t_{l}}\|f(s,\cdot,v_l)\|_{\FY}e^{-\la_0(t_{l}-s_1)^{\rho_0}}\{\Pi _{j=l+1}^{k-1}d\sigma _{j}\}\\&\times \Pi _{j=1}^{l-1}\{{{%
e^{\nu (v_{j})(t_{j+1}-t_{j})} d\sigma _{j}}}\}ds_1\\
\leq&
Ce^{-\frac{\la_0}{2} t_1^{\rho_0}}\sup\limits_{0\leq s\leq t_1}e^{\frac{\la_0}{2} s^{\rho_0}}\|f(s,\cdot,v_l)\|_{\FY},
\end{split}
\end{equation}
and
\begin{equation}\label{kd3}
\begin{split}
\int_{\prod_{j=1}^{k-1}%
\mathcal{V}_{j}}\mathbf{1}_{\{0<t_{k}\}}|f(t_k,\cdot,v_{k-1})|d\Sigma
_{k-1}(t_{k})\leq C\eps_0 e^{-\frac{\la_0}{2} t_1^{\rho_0}}\sup\limits_{0\leq s\leq t_1}e^{\frac{\la_0}{2} s^{\rho_0}}\|f(s)\|_\infty.
\end{split}
\end{equation}

\end{lemma}

\begin{proof}
We first prove \eqref{kd2}. Recall the decomposition \eqref{sp.k1}, we also
rewrite
\begin{equation*}
\begin{split}
\int_{\Pi_{j=1}^{k-1}\mathcal{V}_j}& \sum_{l=1}^{k-1}\int_{t_{l+1}}^{t_l}\mathbf{1}
_{\{t_{l+1}>0\}}\|f(s)\|_{\FY}
e^{\nu(v_{l})(s-t_{l})}\mu^{-1/2}(v_l)d\sigma _{l}ds\\
&\times\Pi _{j=1}^{l-1}\{e^{\nu (v_{j})(t_{j+1}-t_{j})} d\sigma _{j}\}\\
=&\int_{\Pi_{j=1}^{k-1}\mathcal{V}_j\atop{\max\{|v_1|,|v_2|,\cdots,|v_{k-1}|\}\leq k}} \sum_{l=1}^{k-1}\int_{t_{l+1}}^{t_l}\mathbf{1}
_{\{t_{l+1}>0\}}\|f(s)\|_{\FY}\\
&\quad\times\mu^{-1/2}(v_l)
e^{\nu(v_{l})(s-t_{l})}d\sigma _{l}ds\Pi _{j=1}^{l-1}\{e^{\nu (v_{j})(t_{j+1}-t_{j})} d\sigma _{j}\}\\
&+\int_{\Pi_{j=1}^{k-1}\mathcal{V}_j\atop{\max\{|v_1|,|v_2|,\cdots,|v_{k-1}|\}>k}} \sum_{l=1}^{k-1}\int_{t_{l+1}}^{t_l}\mathbf{1}
_{\{t_{l+1}>0\}}\|f(s)\|_{\FY}\\
&\quad\times\mu^{-1/2}(v_l)
e^{\nu(v_{l})(s-t_{l})}d\sigma _{l}ds\Pi _{j=1}^{l-1}\{e^{\nu (v_{j})(t_{j+1}-t_{j})} d\sigma _{j}\}\\
\eqdef&\mathcal {K}_3+\mathcal {K}_4.
\end{split}
\end{equation*}
To estimate $\CK_3$, as in the proof for \eqref{ktildew2}, we denote $\max\{|v_1|,|v_2|,\cdots,|v_{{k-1}}|\}=|v_m|$ again,
then it follows that
\begin{equation*}
\begin{split}
\CK_3\leq& \int_{\Pi_{j=1}^{k-1}\mathcal{V}_j\atop{\max\{|v_1|,|v_2|,\cdots,|v_{k-1}|\}\leq k}} \sum_{l=1}^{k-1}\int_{t_{l+1}}^{t_l}\mathbf{1}
_{\{t_{l+1}>0\}}\|f(s)\|_{\FY}\\&\times
e^{\nu(v_{m})(s-t_{1})}\mu^{-1/2}(v_m)d\sigma _{l}ds\Pi _{j=1}^{l-1} d\sigma _{j}.
\end{split}
\end{equation*}
Meanwhile, by Young's inequality, we find
\begin{equation}\label{in.de}
e^{-\nu(v)t}w^{-1}_{q/2,\ta}(v)\leq e^{-\la_0t^{\rho_0}},\ \rho_0=\frac{\ta}{\ta-\varrho},
\end{equation}
where $\la_0$ is given by
$$
0<\la_0\leq(C_\varrho\rho_0)^{-\rho_0}\left(\frac{q}{8(1-\rho_0)}\right)^{1-\rho_0}>0.
$$ 
Using \eqref{in.de}, we obtain for $\rho_0=\frac{\ta}{\ta-\varrho}$
\begin{equation}\label{sp.k13}
\begin{split}
\CK_3
\leq &\int_{\Pi_{j=1}^{k-1}\mathcal{V}_j\atop{\max\{|v_1|,|v_2|,\cdots,|v_{k-1}|\}\leq k}} \sum_{l=1}^{k-1}\int_{t_{l+1}}^{t_l}\mathbf{1}
_{\{t_{l+1}>0\}}e^{\frac{\la_0}{2}s^{\rho_0}}\|f(s)\|_{\FY}
e^{-\la_0(t_{1}-s)^{\rho_0}}e^{-\frac{\la_0}{2}s^{\rho_0}}\\[2mm]
&\times w_{q/2,\ta}(v_m)\mu^{-1/2}(v_m)d\sigma _{l}ds\Pi _{j=1}^{l-1} d\sigma _{j}
\\[2mm]
\leq&\sqrt{2\pi} e^{-\frac{\la_0}{2}t_1^{\rho_0}}
\int_{\Pi_{j=1}^{k-1}\mathcal{V}_j\atop{\max\{|v_1|,|v_2|,\cdots,|v_{k-1}|\}\leq k}} e^{\frac{q}{8}|v_m|^{\ta}}e^{\frac{|v_m|^2}{4}}\\[2mm]
&\times \left\{\sum_{l=1}^{k-1}
\mathbf{1}
_{\{t_{l+1}>0\}}\int_{t_{l+1}}^{t_l}e^{-\frac{\la_0}{2}(t_{1}-s)^{\rho_0}}ds\right\}\sup\limits_{0\leq s\leq t_1}\left\{e^{\frac{\la_0}{2}s^{\rho_0}}\|f(s)\|_{\FY}\right\}\Pi _{j=1}^{l-1} d\sigma _{j}\\
\leq&
e^{-\frac{\la_0}{2}t_1^{\rho_0}}\sup\limits_{0\leq s\leq t_1}\left\{e^{\frac{\la_0}{2}s^{\rho_0}}\|f(s)\|_{\FY}\right\}
\\&\times \frac{1}{\sqrt{2\pi}}\int_{n(x_m)\cdot v_{m}>0}(n(x_m)\cdot v_{m}) e^{-\frac{1}{4}|v_m|^2+\frac{q}{8}|v_m|^{\ta}}dv_m\\
\leq&
Ce^{-\frac{\la_0}{2}t_1^{\rho_0}}\sup\limits_{0\leq s\leq t_1}\left\{e^{\frac{\la_0}{2}s^{\rho_0}}\|f(s)\|_{\FY}\right\}.
\end{split}
\end{equation}
Here Lemma \ref{el.ine} is also used to guarantee $e^{-\frac{\la_0}{2}(t_{1}-s)^{\rho_0}}e^{-\frac{\la_0}{2}s^{\rho_0}}\leq e^{-\frac{\la_0}{2}t_1^{\rho_0}}$ for $0<\rho_0<1$. 

As to $\CK_4$, assume with no loss of generality $|v_i|\geq k$, following the calculations for $\CK_2$ in the proof of Lemma \ref{k}, one has
\begin{eqnarray}\label{sp.k14}
\mathcal {K}_4\leq&&\int_{\Pi_{j=1}^{k-1}\mathcal{V}_j}\sum\limits_{l=1}^{k-1}\int_{t_{l+1}}^{t_l}
e^{-\la_0(t_l-s)^\rho}ds\mathbf{1}
_{\{t_{l+1}>0\}}\sup\limits_{t_{l+1}\leq s\leq t_l}\|f(s)\|_{\FY}\notag \\&&\times w_{q/2,\ta}(v_l)\mu^{-1/2}(v_l)\Pi_{j=1}^{k-1} d\sigma_j\notag\\
\leq&&C\sum\limits_{l=1}^{k-1}\int_{\Pi_{j=1}^{k-1}\mathcal{V}_j}w_{q/2,\ta}(v_l)\mu^{-1/2}(v_l)\Pi_{j=1}^{k-1} d\sigma_j\sup\limits_{0\leq s\leq t_1}\|f(s)\|_{\FY} \notag\\
 \leq&&C\int_{\Pi_{j=1}^{i-1}\mathcal{V}_j}\Pi_{j=1}^{i-1} d\sigma_j\int_{n(x_i)\cdot v_{i}>0\atop{|v_i|> k}}(n(x_i)\cdot v_{i}) e^{-\frac{1}{4}|v_i|^2+\frac{q}{8}|v_i|^{\ta}}dv_i\sup\limits_{0\leq s\leq t_1}\|f(s)\|_{\FY}
 \notag\\&&+C\sum\limits_{l=1}^{i-1}\int_{\Pi_{j=1}^{l-1}\mathcal{V}_j}\Pi_{j=1}^{l-1} d\sigma_j\int_{n(x_l)\cdot v_{l}>0}(n(x_l)\cdot v_{l}) e^{-\frac{1}{4}|v_l|^2+\frac{q}{8}|v_l|^{\ta}}dv_l\\&&\quad\times\int_{\Pi_{j=l+1}^{i-1}\mathcal{V}_j}\Pi_{j=l+1}^{i-1} d\sigma_j
 \int_{n(x_i)\cdot v_{i}>0\atop{|v_i|> k}}e^{-\frac{|v_i|^2}{2}}(n(x_i)\cdot v_i)dv_i\sup\limits_{0\leq s\leq t_1}\|f(s)\|_{\FY}\notag\\
 &&+C\sum\limits_{l=i+1}^{k-1}\int_{\Pi_{j=1}^{i-1}\mathcal{V}_j}\Pi_{j=1}^{i-1} d\sigma_j\int_{n(x_i)\cdot v_{i}>0\atop{|v_i|> k}}e^{-\frac{|v_i|^2}{2}}(n(x_i)\cdot v_i)dv_i\notag\\&&\quad\times\int_{\Pi_{j=i+1}^{l-1}\mathcal{V}_j}\Pi_{j=i+1}^{l-1} d\sigma_j
\int_{n(x_l)\cdot v_{l}>0}n(x_l)\cdot v_{l} e^{-\frac{1}{4}|v_l|^2+\frac{q}{8}|v_l|^{\ta}}dv_l\sup\limits_{0\leq s\leq t_1}\|f(s)\|_{\FY}
 \notag\\
\leq&& C_{q,\ta}(k-1)e^{-\frac{k^2}{8}}\sup\limits_{0\leq s\leq t_1}\|f(s)\|_{\FY}\leq
C_{q,\ta}e^{-\frac{k^2}{16}}\sup\limits_{0\leq s\leq t_1}\|f(s)\|_{\FY}.\notag
\end{eqnarray}
Notice that $k=C_{1}[\al(t)]^{5/4}$, \eqref{kd2} then follows from \eqref{sp.k13} and \eqref{sp.k14}.
Just like the proof for Lemma \ref{k}, \eqref{kd1} can be handled in a similar way as \eqref{kd2}, and the proofs for \eqref{kd11} and \eqref{kd21} being similar and easier, we omit the details for brevity.
It remains now to prove \eqref{kd3}. To do that, we have, using decomposition as \eqref{sp.k1} again
\begin{equation*}
\begin{split}
\int_{\prod_{j=1}^{k-1}
\mathcal{V}_{j}}&\mathbf{1}_{\{0<t_{k}\}}|f(t_k,\cdot,v_{k-1})|\mu^{-1/2}(v_{k-1})d\Sigma
_{k-1}(t_{k})\\
=&\int_{\Pi_{j=1}^{k-1}\mathcal{V}_j\atop{\max\{|v_1|,|v_2|,\cdots,|v_{k-1}|\}\leq k}} \mathbf{1}_{\{0<t_{k}\}}|f(t_k,\cdot,v_{k-1})|\mu^{-1/2}(v_{k-1})
\\&\quad\times\Pi _{j=1}^{k-1}\{e^{\nu (v_{j})(t_{j+1}-t_{j})} d\sigma _{j}\}\\
&+\int_{\Pi_{j=1}^{k-1}\mathcal{V}_j\atop{\max\{|v_1|,|v_2|,\cdots,|v_{k-1}|\}>k}} \mathbf{1}_{\{0<t_{k}\}}|f(t_k,\cdot,v_{k-1})|\mu^{-1/2}(v_{k-1})
\\&\quad\times\Pi _{j=1}^{k-1}\{e^{\nu (v_{j})(t_{j+1}-t_{j})} d\sigma _{j}\}\\
\eqdef&\mathcal {K}_5+\mathcal {K}_6.
\end{split}
\end{equation*}
To compute $\CK_5$, let us denote $\max\{|v_1|,|v_2|,\cdots,|v_{k-1}|\}=|v_m|$ again, we first prove that there exists a constant $C>0$ independent of $t$  such that for all $1\leq m\leq k-1$ and small $\eps_0>0$
\begin{equation}\label{sp.k15}
\begin{split}
\int_{\Pi_{j=1}^{k-1}\mathcal{V}_j} \mathbf{1}_{\{0<t_{k}\}}w_{q/2,\ta}(v_m)\mu^{-1/2}(v_m)\Pi _{j=1}^{k-1} d\sigma _{j}
\leq C\eps_0.
\end{split}
\end{equation}
For this, we
define non-grazing
sets for $1\leq j\leq k-1$ as $\mathcal{V}_{j}^{\mathfrak{z}}=\{v_{j}\in
\mathcal{V}_{j}:$ $v_{j}\cdot n(x_{j})\geq \mathfrak{z}\}\cap
\{v_{j}\in \mathcal{V}_{j}:$ $|v_{j}|\leq \frac{1}{\mathfrak{z}}\}$ with
$\mathfrak{z}>0$ and sufficiently small.
Notice that $(q,\ta)\in\CA_{q,\ta}$, we obtain by a direct calculation
\begin{equation}
\begin{split}
\int_{\mathcal{V}_{j}\setminus \mathcal{V}_{j}^{\mathfrak{z}}}&w_{q/2,\ta}(v_j)\mu^{-1/2}(v_j)d\sigma
_{j}\\
\leq& \int_{0<v_{j}\cdot n(x_{j})\leq \mathfrak{z}}w_{q/2,\ta}(v_j)\mu^{-1/2}(v_j)d\sigma
_{j}+\int_{|v_{j}|\geq \frac{1}{\mathfrak{z}}}w_{q/2,\ta}(v_j)\mu^{-1/2}(v_j)d\sigma _{j}\\
\leq& C_{q,\ta}\int_{0<v_{j}\cdot n(x_{j})\leq \mathfrak{z}}\mu^{1/4}(v_j)v_{j}\cdot n(x_{j})dv
_{j}\\&+C_{q,\ta}\int_{|v_{j}|\geq \frac{1}{\mathfrak{z}}}\mu^{1/4}(v_j)v_{j}\cdot n(x_{j})dv
_{j}
\leq C\mathfrak{z},
\label{v-v}
\end{split}
\end{equation}
and
\begin{equation}
\begin{split}
\int_{\mathcal{V}_{j}}w_{q/2,\ta}(v_j)\mu^{-1/2}(v_j)d\sigma_j
\leq C,
\label{v-v2}
\end{split}
\end{equation}
where $C$ is independent of $j$. On the other hand, if $v_{j}\in \mathcal{V}%
_{j}^{\mathfrak{z}}$, we know from the definition of diffusive back-time cycle (%
\ref{diffusecycle}) that $x_{j}-x_{j+1}=(t_{j}-t_{j+1})v_{j}$.
Since $|v_{j}|\leq \frac{1}{\mathfrak{z}}$, and $v_{j}\cdot
n(x_{j})\geq \mathfrak{z}$, thanks to Lemma \ref{huang}, it follows that $(t_{j}-t_{j+1})\geq \frac{\mathfrak{z}^{3}}{%
C_{\xi }}.$ Hence, when $t_{k}(t,x,v,v_{1},v_{2}...,%
v_{k-1})>0$, there can be at most $\left[ \frac{C_{\xi }\al(t)}{%
\mathfrak{z}^{3}}\right] +1$ number of $v_{j}$ $\in \mathcal{V}_{j}^{%
\mathfrak{z}}$ for $1\leq j\leq k-1$. We therefore compute
\begin{equation*}
\begin{split}
\int_{\Pi_{j=1}^{k-1}\mathcal{V}_j}& \mathbf{1}_{\{0<t_{k}\}}w_{q/2,\ta}(v_m)\mu^{-1/2}(v_m)\Pi _{j=1}^{k-1} d\sigma _{j}\\
\leq& \sum_{l=1}^{\left[ \frac{C_{\xi }\al(t)}{\mathfrak{z}^{3}}\right]
+1}\int_{\mathcal{V}_1^{\dagger }}\Pi _{j=1}^{k-1}w_{q/2,\ta}(v_m)\mu^{-1/2}(v_m)d\sigma _{j}
\\&+\sum_{l=1}^{\left[ \frac{C_{\xi }\al(t)}{\mathfrak{z}^{3}}\right]+1}
\int_{\mathcal{V}_2^{\dagger }}\Pi _{j=1}^{k-1}w_{q/2,\ta}(v_m)\mu^{-1/2}(v_m)d\sigma _{j}\\
\leq &\sum_{l=1}^{\left[ \frac{C_{\xi }\al(t)}{\mathfrak{z}^{3}}\right]
+1}\binom{k-1}{l}\left|\sup_{j}\int_{\mathcal{V}_{j}^{\mathfrak{z}}}d\sigma
_{j}\right|^{l-1}\int_{\mathcal {V}^\mathfrak{z}_m}w_{q,\ta}(v_m)\mu^{-1/2}(v_m)d\sigma
_{m}\\&\times\left\{ \sup_{j}\int_{\mathcal{V}_{j}\setminus \mathcal{V}_{j}^{
\mathfrak{z}}}d\sigma _{j}\right\} ^{k-l-1}\\&+\sum_{l=1}^{\left[ \frac{C_{\xi }\al(t)}{\mathfrak{z}^{3}}\right]
+1}\binom{k-1}{l}\left|\sup_{j}\int_{\mathcal{V}_{j}^{\mathfrak{z}}}d\sigma
_{j}\right|^{l}\int_{\mathcal {V}_m/\mathcal {V}^\mathfrak{z}_m}w_{q,\ta}(v_m)\mu^{-1/2}(v_m)d\sigma
_{m}\\&\times\left\{ \sup_{j}\int_{\mathcal{V}_{j}\setminus \mathcal{V}_{j}^{
\mathfrak{z}}}d\sigma _{j}\right\} ^{k-l-2},
\end{split}
\end{equation*}%
where $\mathcal{V}_1^{\dagger }$ is the set where $\text{there are exactly }l
\text{ of }v_{j_{i}}\in \mathcal{V}_{j_{i}}^{\mathfrak{z}}$ including $v_m \in \mathcal{V}_{m}^{\mathfrak{z}}$, and $k-1-l
\text{ of }v_{j_{i}}\notin \mathcal{V}_{j_{i}}^{\mathfrak{z}}$, 
while $\mathcal{V}_2^{\dagger }$ is the set where $\text{there are exactly }l
\text{ of }v_{j_{i}}\in \mathcal{V}_{j_{i}}^{\mathfrak{z}},\text{ and }k-1-l
\text{ of }v_{j_{i}}\notin \mathcal{V}_{j_{i}}^{\mathfrak{z}}$ and also $v_m \notin \mathcal{V}_{m}^{\mathfrak{z}}$. Since $
d\sigma $ is a probability measure, $\int_{\mathcal{V}_{j}^{\mathfrak{z}
}}d\sigma _{j}\leq 1$, and
\begin{equation*}
\left\{ \int_{\mathcal{V}_{j}\setminus \mathcal{V}_{j}^{\mathfrak{z}
}}d\sigma _{j}\right\} ^{k-l-1}\leq \left\{ \int_{\mathcal{V}_{j}\setminus
\mathcal{V}_{j}^{\mathfrak{z}}}d\sigma _{j}\right\} ^{k-2-\left[ \frac{%
C_{\xi }\al(t)}{\mathfrak{z}^{3}}\right] }\leq \{C\mathfrak{z}\}^{k-2-%
\left[ \frac{C_{\xi }\al(t)}{\mathfrak{z}^{3}}\right] }.
\end{equation*}%
With this, from \eqref{v-v}, \eqref{v-v2} and $\binom{k-1}{l}\leq \{k-1\}^{l}\leq \{k-1\}^{\left[ \frac{%
C_{\xi }\al(t)}{\mathfrak{z}^{3}}\right] +1}$, we deduce that
\begin{equation*}
\begin{split}
\int &\mathbf{1}_{\{t_{k}>0\}}w_{q/2,\ta}(v_m)\mu^{-1/2}(v_m))\Pi _{l=1}^{k-1}d\sigma _{l}\\
\leq& C{\left(\left[
\frac{C_{\xi }\al(t)}{\mathfrak{z}^{3}}\right]+1\right) }(k-1)^{\left[
\frac{C_{\xi }\al(t)}{\mathfrak{z}^{3}}\right] +1}(C\mathfrak{z})^{k-2-\left[
\frac{C_{\xi }\al(t)}{\mathfrak{z}^{3}}\right] }.
\end{split}
\end{equation*}
For $\eps_0>0$, \eqref{sp.k15} follows for $C\mathfrak{z}<1$, and $k>>\left[
\frac{C_{\xi }\al(t)}{\mathfrak{z}^{3}}\right]+2.$
We now go back to
$\CK_5$, from \eqref{sp.k15} and \eqref{in.de}, it follows that
\begin{equation*}
\begin{split}
\CK_5
\leq &\int_{\Pi_{j=1}^{k-1}\mathcal{V}_j\atop{\max\{|v_1|,|v_2|,\cdots,|v_{k-1}|\}\leq k}} \mathbf{1}_{\{0<t_{k}\}}|f(t_k,\cdot,v_{k-1})|
e^{-\nu(v_m)(t_1-t_{k})}\mu^{-1/2}(v_m)\Pi _{j=1}^{k-1} d\sigma _{j}
\\
\leq&\int_{\Pi_{j=1}^{k-1}\mathcal{V}_j\atop{\max\{|v_1|,|v_2|,\cdots,|v_{k-1}|\}\leq k}} \mathbf{1}_{\{0<t_{k}\}}e^{-\la_0t_k^{\rho_0}}
e^{-\la_0(t_1-t_k)^{\rho_0}}w_{q/2,\ta}(v_m)\mu^{-1/2}(v_m)\\
&\times\Pi _{j=1}^{k-1} d\sigma _{j}\sup\limits_{0\leq t_k\leq t_1}\left\{e^{\frac{\la_0}{2}t_k^{\rho_0}}\|f(t_k)\|_\infty\right\}\\
\leq& C\eps_0 e^{-\frac{\la_0}{2}t_1^{\rho_0}}\sup\limits_{0\leq s\leq t_1}\left\{e^{\frac{\la_0}{2}s^{\rho_0}}\|f(s)\|_\infty\right\}.
\end{split}
\end{equation*}
As to $\CK_6$, assume with no loss of generality $|v_i|\geq k$, apply \eqref{in.de} to obtain
\begin{equation}\label{sp.k16}
\begin{split}
\CK_6
\leq &\int_{\Pi_{j=1}^{k-1}\mathcal{V}_j\atop{\max\{|v_1|,|v_2|,\cdots,|v_{k-1}|\}\geq k}} \mathbf{1}_{\{0<t_{k}\}}|f(t_k,\cdot,v_{k-1})|
\{\Pi _{j=l}^{k-1}e^{-\nu(v_l)(t_l-t_{l+1})}\}\\&\times\mu^{-1/2}(v_k)\Pi _{j=1}^{k-1} d\sigma _{j}
\\
\leq&\int_{\Pi_{j=1}^{k-1}\mathcal{V}_j\atop{\max\{|v_1|,|v_2|,\cdots,|v_{k-1}|\}\geq k}} \mathbf{1}_{\{0<t_{k}\}}e^{-\la_0t_k^{\rho_0}}
e^{-\la_0(t_1-t_k)^{\rho_0}}\left\{\Pi _{l=1}^{k-1}w_{q/2,\ta}(v_l)\right\}\mu^{-1/2}(v_k)\\
&\times\Pi _{j=1}^{k-1} d\sigma _{j}\sup\limits_{0\leq t_k\leq t_1}\left\{e^{\frac{\la_0}{2}t_k^{\rho_0}}\|f(t_k)\|_\infty\right\}\\
\leq& C e^{-\frac{\la_0}{2}t_1^{\rho_0}}\sup\limits_{0\leq s\leq t_1}\left\{e^{\frac{\la_0}{2}s^{\rho_0}}\|f(s)\|_\infty\right\}
\left(\int_{\mathcal {V}_l}w_{q/2,\ta}(v_l)\mu^{-1/2}(v_l)d\si_l\right)^{k-2}
\\&\times\int_{\mathcal {V}_i}w_{q/2,\ta}(v_i)\mu^{-1/2}(v_i)d\si_i\\
\leq& C e^{-\frac{\la_0}{2}t_1^{\rho_0}}\sup\limits_{0\leq s\leq t_1}\left\{e^{\frac{\la_0}{2}s^{\rho_0}}\|f(s)\|_\infty\right\}C_{q,\ta}^{k-1}e^{-k^2/16}.
\end{split}
\end{equation}
Choosing $k$ suitable large so that $C_{q,\ta}^{k-1}e^{-k^2/16}<\eps_0$, one sees that \eqref{sp.k16} also enjoys the bound \eqref{kd3}.
This completes the proof of Lemma \ref{k.d}.

\end{proof}

We now turn to prove exponential decay using Lemma \ref{k.d} and the uniform bound \eqref{u.bd}. The main difficulty with proving rapid decay
\eqref{decay} is created by the fact that the collision frequency has no positive lower bound in the case of soft potential. However, as it is shown in \eqref{in.de}, one can trade between exponential decay rates and the additional exponential momentum weight on the initial data and the solution itself.

\begin{proof}[The proof of \eqref{decay}]
Recall that $f(t,x,v)$ satisfies
\begin{eqnarray*}
\left\{
\begin{array}{rll}
&\partial _{t}f+v\cdot \nabla _{x}f+\nu f=Kf+\Ga(f,f),\text{ \ \ }f(0,x,v)=f_{0}, \\[2mm]
&f_{-}=P_{\gamma }f.
\end{array}\right.
\end{eqnarray*}
With this, by a same kind of computation as for obtaining \eqref{iteration1}, one has
\begin{equation}\label{iteration.fn}
\begin{split}
  |f(t,x,v)| \leq&
\underbrace{\left\{\mathbf{1}_{t_{1}\leq 0}
\int_{0}^{t}+\mathbf{1}_{t_{1}>0}
\int_{t_1}^{t}\right\}
e^{-\nu
(v)(t-s)}|K^{1-\chi}f(s,x-(t-s){v},v)|ds}_{\CJ_1}
\\
&\underbrace{+\left\{\mathbf{1}_{t_{1}\leq 0}
\int_{0}^{t}+\mathbf{1}_{t_{1}>0}
\int_{t_1}^{t}\right\}
e^{-\nu
(v)(t-s)}|K^{\chi}f(s,x-(t-s){v},v)|ds}_{\CJ_2}
\\
&
\underbrace{+\left\{\mathbf{1}_{t_{1}\leq 0}
\int_{0}^{t}+\mathbf{1}_{t_{1}>0}
\int_{t_1}^{t}\right\}e^{-\nu
(v)(t-s)}|g_f(s,x-(t-s){v},v)|ds}_{\CJ_3}
+\sum\limits_{n=4}^{8}\CJ_n,
\end{split}
\end{equation}%
with
\begin{equation*}
\begin{split}
\CJ_4=&\mathbf{1}_{t_{1}\leq 0}e^{-\nu
(v)t}|f(0,x-t{v},v)|
\\&+e^{-\nu (v)(t-t_{1})}\sqrt{\mu}\int_{\prod_{j=1}^{k-1}%
\mathcal{V}_{j}}\sum_{l=1}^{k-1}\mathbf{1}_{\{t_{l+1}\leq
0<t_{l}\}}|f(0,x_{l}-t_{l}{v}_{l},v_{l})|d\Sigma _{l}(0),
\end{split}
\end{equation*}
\begin{equation*}
\begin{split}
\CJ_5=&e^{-\nu (v)(t-t_{1})}\sqrt{\mu}\bigg\{\int_{\prod_{j=1}^{k-1}%
\mathcal{V}_{j}}\sum_{l=1}^{k-1}\int_{0}^{t_l}\mathbf{1}_{\{t_{l+1}\leq
0<t_{l}\}}\\
&\quad\times|[K^{1-\chi} f](s,x_{l}-(t_{l}-s){v}_{l},v_{l})|d\Sigma _{l}(s)ds
\\&+\int_{\prod_{j=1}^{k-1}%
\mathcal{V}_{j}}\sum_{l=1}^{k-1}\int_{t_{l+1}}^{t_{l}}
\mathbf{1}_{\{0<t_{l+1}\}}|[K^{1-\chi} f](s,x_{l}-(t_{l}-s){v}_{l},v_{l})|d\Sigma _{l}(s)ds\bigg\},
\end{split}
\end{equation*}
\begin{equation*}
\begin{split}
\CJ_6=&e^{-\nu (v)(t-t_{1})}\sqrt{\mu}\bigg\{\int_{\prod_{j=1}^{k-1}%
\mathcal{V}_{j}}\sum_{l=1}^{k-1}\int_{0}^{t_l}\mathbf{1}_{\{t_{l+1}\leq
0<t_{l}\}}\\
&\quad\times|[K^\chi f](s,x_{l}-(t_{l}-s){v}_{l},v_{l})|d\Sigma _{l}(s)ds
\\&+\int_{\prod_{j=1}^{k-1}%
\mathcal{V}_{j}}\sum_{l=1}^{k-1}\int_{t_{l+1}}^{t_{l}}
\mathbf{1}_{\{0<t_{l+1}\}}|[K^\chi f](s,x_{l}-(t_{l}-s){v}_{l},v_{l})|d\Sigma _{l}(s)ds\bigg\},
\end{split}
\end{equation*}
\begin{equation*}
\begin{split}
\CJ_7=&e^{-\nu (v)(t-t_{1})}\sqrt{\mu}\bigg\{\int_{\prod_{j=1}^{k-1}%
\mathcal{V}_{j}}\sum_{l=1}^{k-1}\int_{0}^{t_l}\mathbf{1}_{\{t_{l+1}\leq
0<t_{l}\}}\\
&\quad\times|g_f(s,x_{l}-(t_{l}-s){v}_{l},v_{l})|d\Sigma _{l}(s)ds \\&+\int_{\prod_{j=1}^{k-1}%
\mathcal{V}_{j}}\sum_{l=1}^{k-1}\int_{t_{l+1}}^{t_{l}}\mathbf{1}_{\{0<t_{l+1}\}}|g_f(s,x_{l}-(t_{l}-s){v}_{l},v_{l})|d\Sigma _{l}(s)ds\bigg\},
\end{split}
\end{equation*}
\begin{equation*}
\CJ_8=e^{-\nu (v)(t-t_{1})}\sqrt{\mu}\int_{\prod_{j=1}^{k-1}%
\mathcal{V}_{j}}\mathbf{1}_{\{0<t_{k}\}}|f(t_{k},x_{k},v_{k-1})|d\Sigma
_{k-1}(t_{k}),\ \ k\geq2,
\end{equation*}%
where $g_f=\Ga(f,f)$ and $\Sigma _{l}(s)$ $(l=1,2,\cdots,)$ is given by \eqref{measure1}.
We now turn to compute $\CJ_n$ $(n=1,2,\cdots,8)$ term by term. As the way to deal with \eqref{iteration}, let us first
compute $\CJ_1$, $\CJ_3$, $\CJ_4$, $\CJ_5$, $\CJ_7$ and $\CJ_8$, the estimates for the delicate terms $\CJ_2$ and $\CJ_6$ will be postponed
to a later step when the estimation like \eqref{hmain} is derived.

\noindent{\it Estimates on $\CJ_1$ and $\CJ_5$.} It follows from Lemma \ref{es.k} and \eqref{in.de} that
\begin{equation*}
\begin{split}
\CJ_1\leq& C\sup\limits_{0\leq s\leq t}\left\{e^{\frac{\la_0}{2}s^{\rho_0}}\|f(s)\|_\infty\right\}
\int_{0}^te^{-\frac{\la_0}{2}(t-s)^{\rho_0}}e^{-\frac{\la_0}{2}(t-s)^{\rho_0}}\\
&\times e^{-\frac{\la_0}{2}s^{\rho_0}}ds
w_{q/2,\ta}\int_{\R^3}K^{1-\chi}dv\\
\leq&C\eps^{\varrho+3}e^{-\frac{\la_0}{2}t^{\rho_0}}\sup\limits_{0\leq s\leq t}\left\{e^{\frac{\la_0}{2}s^{\rho_0}}\|f(s)\|_\infty\right\}.
\end{split}
\end{equation*}
Likewise, Lemmas \ref{es.k} and \ref{k.d} and inequality \eqref{in.de} imply
\begin{equation*}
\begin{split}
\CJ_5\leq& C\eps^{\varrho+3}e^{-\frac{\la_0}{2}(t-t_1)^{\rho_0}}e^{-\frac{\la_0}{2}t_1^{\rho_0}}w_{q/2,\ta}(v)\sqrt{\mu}(v)
\left\{e^{\frac{\la_0}{2}s^{\rho_0}}\|f(s)\|_\infty\right\}\\
\leq&C\eps^{\varrho+3}e^{-\frac{\la_0}{2}t^{\rho_0}}\sup\limits_{0\leq s\leq t}\left\{e^{\frac{\la_0}{2}s^{\rho_0}}\|f(s)\|_\infty\right\}.
\end{split}
\end{equation*}

\noindent{\it Estimates on $\CJ_3$ and $\CJ_7$.} We have, using \eqref{in.de}
\begin{equation*}
\begin{split}
\CJ_3\leq& C\sup\limits_{0\leq s\leq t}\left\{e^{\frac{\la_0}{2}s^{\rho_0}}\|w_{q/2,\ta}g_f(s)\|_\infty\right\}
\int_{0}^te^{-\frac{\la_0}{2}(t-s)^{\rho_0}}e^{-\frac{\la_0}{2}(t-s)^{\rho_0}}e^{-\frac{\la_0}{2}s^{\rho_0}}ds
\\
\leq&Ce^{-\frac{\la_0}{2}t^{\rho_0}}\sup\limits_{0\leq s\leq t}\left\{e^{\frac{\la_0}{2}s^{\rho_0}}\|w_{q/2,\ta}g_f(s)\|_\infty\right\}.
\end{split}
\end{equation*}
Similarly, applying Lemma \ref{k.d} and the inequality \eqref{in.de} again leads to
\begin{equation*}
\begin{split}
\CJ_7\leq& C\eps^{\varrho+3}e^{-\frac{\la_0}{2}(t-t_1)^{\rho_0}}e^{-\frac{\la_0}{2}t_1^{\rho_0}}w_{q/2,\ta}\sqrt{\mu}\sup\limits_{0\leq s\leq t}\left\{e^{\frac{\la_0}{2}s^{\rho_0}}\|w_{q/2,\ta}g_f(s)\|_\infty\right\}
\\
\leq&Ce^{-\frac{\la_0}{2}t^{\rho_0}}\sup\limits_{0\leq s\leq t}\left\{e^{\frac{\la_0}{2}s^{\rho_0}}\|w_{q/2,\ta}g_f(s)\|_\infty\right\}.
\end{split}
\end{equation*}

\noindent{\it Estimates on $\CJ_4$.} For the first term in $\CJ_4$, one directly has from \eqref{in.de} that
\begin{equation*}
\begin{split}
\mathbf{1}_{t_{1}\leq 0}e^{-\nu
(v)t}|f(0,x-t{v},v)|\leq e^{-\frac{\la_0}{2}t^{\rho_0}}\|w_{q/2,\ta}f_0\|_\infty.
\end{split}
\end{equation*}
As to the second term, applying the similar calculations as in the proof of Lemma \ref{k.d}, we obtain
\begin{equation*}
\begin{split}
e^{-\nu (v)(t-t_{1})}&\sqrt{\mu}\int_{\prod_{j=1}^{k-1}%
\mathcal{V}_{j}}\sum_{l=1}^{k-1}\mathbf{1}_{\{t_{l+1}\leq
0<t_{l}\}}|f(0,x_{l}-t_{l}{v}_{l},v_{l})|d\Sigma _{l}(0)\\
\leq& Ce^{-\frac{\la_0}{2}(t-t_1)^{\rho_0}}e^{-\frac{\la_0}{2}t_1^{\rho_0}}w_{q/2,\ta}\sqrt{\mu}\|f_0\|_\infty
\leq Ce^{-\frac{\la_0}{2}t^{\rho_0}}\|f_0\|_\infty.
\end{split}
\end{equation*}
Gathering the above two kind of estimates, we have
\begin{equation*}
\begin{split}
\CJ_4\leq Ce^{-\frac{\la_0}{2}t^{\rho_0}}\|w_{q/2,\ta}f_0\|_\infty.
\end{split}
\end{equation*}
\noindent{\it Estimates on $\CJ_8$.} \eqref{kd3} in Lemma \ref{k.d} directly yields
\begin{equation*}
\begin{split}
\CJ_8\leq& C\eps_0 e^{-\frac{\la_0}{2}(t-t_1)^{\rho_0}}w_{q/2,\ta}(v)\sqrt{\mu(v)}e^{-\frac{\la_0}{2}t_1^{\rho_0}}\sup\limits_{0\leq s\leq t_1}\left\{e^{\frac{\la_0}{2}s^{\rho_0}}\|f(s)\|_\infty\right\}\\
\leq& C\eps_0 e^{-\frac{\la_0}{2}t^{\rho_0}}\sup\limits_{0\leq s\leq t}\left\{e^{\frac{\la_0}{2}s^{\rho_0}}\|f(s)\|_\infty\right\}.
\end{split}
\end{equation*}
Substituting all the above estimates into \eqref{iteration.fn}, we arrive at
\begin{equation}\label{f26}
\begin{split}
  |f(t,x,v)| \leq&
\CJ_2+\CJ_6+A_2(t),
\end{split}
\end{equation}
with
\begin{equation*}
\begin{split}
A_2(t)=&Ce^{-\frac{\la_0}{2}t^{\rho_0}}\Big\{\|w_{q/2,\ta}f_0\|_\infty+(\eps_0+\eps^{\varrho+3}) \sup\limits_{0\leq s\leq t}e^{\frac{\la_0}{2}s^{\rho_0}}\|f(s)\|_\infty\\&+
\sup\limits_{0\leq s\leq t}e^{\frac{\la_0}{2}s^{\rho_0}}\|w_{q/2,\ta}g_f(s)\|_\infty\Big\}.
\end{split}
\end{equation*}
Next, plug \eqref{f26} into $K^{\chi}f$  and perform the similar calculation as \eqref{diff3} to obtain
\begin{equation}\label{diff4}
\begin{split}
K^\chi&f(s,X_{\mathbf{cl}}(s),v_{l})\\
\leq& \int_{
\mathbf{R}^{3}}\mathbf{k}^\chi(v_{l},v^{\prime })|f(s,X_{
\mathbf{cl}}(s),v^{\prime })|dv^{\prime }  \\
\leq& \iint \left\{\mathbf{1}_{t_{1}^{\prime }\leq 0}\int_{0}^{s}+\mathbf{1}_{t_{1}^{\prime }>0}\int_{t_{1}^{\prime }}^{s}\right\}e^{-\nu
(v^{\prime })(s-s_{1})} \mathbf{k}^\chi(v_{l},v^{\prime })
\mathbf{k}^\chi(v^{\prime },v^{\prime \prime })
\\&\quad\times|f(s_{1},X_{\mathbf{cl}}(s)-(s-s_{1})v^{\prime },v^{\prime \prime
})|ds_{1}dv^{\prime }dv^{\prime \prime }\\
&+\iint dv^{\prime }dv^{\prime \prime }\int_{\prod_{j=1}^{k-1}\mathcal{V
}_{j}^{\prime }}e^{-\nu (v^{\prime })(s-t_{1}^{\prime })}
\sqrt{\mu}(v^{\prime })
\\&\quad\times\sum_{l^{\prime }=1}^{k-1}\int_{0}^{t_{l^{\prime
}}^{\prime }}ds_{1}\mathbf{1}_{\{t_{l^{\prime }+1}^{\prime }\leq
0<t_{l^{\prime }}^{\prime }\}}  \mathbf{k}^\chi(v_{l},v^{\prime })
\mathbf{k}^\chi(v_{l^{\prime }}^{\prime },v^{\prime \prime
})\\&\quad\times|f(s_{1,}x_{l^{\prime }}^{\prime
}+(s_{1}-t_{l^{\prime }}^{\prime })v_{l^{\prime }}^{\prime },v^{\prime
\prime })|d\Sigma _{l^{\prime }}(s_{1})  \\
&+\iint dv^{\prime }dv^{\prime \prime }\int_{\prod_{j=1}^{k-1}\mathcal{V
}_{j}^{\prime }}e^{-\nu (v^{\prime })(s-t_{1}^{\prime })}
\sqrt{\mu}(v^{\prime })
\\&\quad\times\sum_{l^{\prime }=1}^{k-1}\int_{t_{l^{\prime}+1}^{\prime }}^{t_{l^{\prime }}^{\prime }}ds_{1}
\mathbf{1}_{\{t_{l^{\prime}+1}^{\prime }>0\}}  \mathbf{k}^\chi(v_{l},v^{\prime })
\mathbf{k}^\chi(v_{l^{\prime }}^{\prime },v^{\prime \prime
})\\&\quad\times|f(s_{1},x_{l^{\prime }}^{\prime }+(s_{1}-t_{l^{\prime}}^{\prime })v_{l^{\prime }}^{\prime },v^{\prime \prime })|d\Sigma
_{l^{\prime }}(s_{1}) \\
&+\int_{\R^{3}}
\mathbf{k}^\chi(v_{l},v^{\prime })dv^{\prime}A_{2}(s)\eqdef\sum\limits_{n=1}^4\CL_n.
\end{split}
\end{equation}
We now estimate $\CJ_6$ with the aid of \eqref{diff4}. Substituting \eqref{diff4} into $\CJ_6$ and applying \eqref{in.de} leads us to
\begin{equation}\label{CI6}
\begin{split}
\CJ_6\leq&C_{q,\ta}e^{-\frac{\la_0}{2}(t-t_1)^{\rho_0}}\int_{\prod_{j=1}^{k-1}
\mathcal{V}_{j}}\sum_{l=1}^{k-1}\left\{\int_{0}^{t_l}\mathbf{1}_{\{t_{l+1}\leq
0<t_{l}\}}+\int_{t_{l+1}}^{t_{l}}
\mathbf{1}_{\{0<t_{l+1}\}}\right\}\\&\times\sum\limits_{n=1}^4\CL_nd\Sigma_{l}(s)ds
= \sum\limits_{n=1}^4\CJ_{6,n},
\end{split}
\end{equation}
where $\CJ_{6,n}$ $(1\leq n\leq4)$ denote four terms on the right hand side of \eqref{CI6} containing $\CL_{n}$ $(1\leq n\leq4)$, respectively. We now estimate $\CJ_{6,n}$ $(1\leq n\leq4)$ term by term. We first consider the simple term $\CJ_{6,4}$, since $\int_{\R^{3}}
\mathbf{k}^\chi(v_{l},v^{\prime })dv^{\prime}<\infty$, in light of Lemma \ref{k.d}, it is straightforward to check
\begin{equation*}
\begin{split}
\CJ_{6,4}\leq& C_{q,\ta}e^{-\frac{\la_0}{2}(t-t_1)^{\rho_0}}e^{-\frac{\la_0}{2}t_1^{\rho_0}}\sup\limits_{0\leq s\leq t}\left\{e^{\frac{\la_0}{2}s^{\rho_0}}A_2(s)\right\}\\
\leq& C_{q,\ta}e^{-\frac{\la_0}{2}t^{\rho_0}}\Big\{\|w_{q/2,\ta}f_0\|_\infty+(\eps_0+\eps^{\varrho+3}) \sup\limits_{0\leq s\leq t}e^{\frac{\la_0}{2}s^{\rho_0}}\|f(s)\|_\infty\\&+
\sup\limits_{0\leq s\leq t}e^{\frac{\la_0}{2}s^{\rho_0}}\|w_{q/2,\ta}g_f(s)\|_\infty\Big\}.
\end{split}
\end{equation*}
For $\CJ_{6,2}$, we first show that there exists a sufficiently large $N>0$ such that
\begin{equation}\label{CJ62.es}
\begin{split}
\CJ^1_{6,2}=&C_{q,\ta}e^{-\frac{\la_0}{2}(t-t_1)^{\rho_0}}\int_{\prod_{j=1}^{k-1}
\mathcal{V}_{j}}\sum_{l=1}^{k-1}\int_{0}^{t_l}\mathbf{1}_{\{t_{l+1}\leq
0<t_{l}\}}\iint dv^{\prime }dv^{\prime \prime }
\\&\quad\times\int_{\prod_{j=1}^{k-1}\mathcal{V
}_{j}^{\prime }}e^{-\nu (v^{\prime })(s-t_{1}^{\prime })}
\sqrt{\mu}(v^{\prime })\sum_{l^{\prime }=1}^{k-1}\int_{0}^{t_{l^{\prime
}}^{\prime }}ds_{1}\mathbf{1}_{\{t_{l^{\prime }+1}^{\prime }\leq
0<t_{l^{\prime }}^{\prime }\}}  \\&\quad\times\mathbf{k}^\chi(v_{l},v^{\prime })
\mathbf{k}^\chi(v_{l^{\prime }}^{\prime },v^{\prime \prime
})|f(s_{1},x_{l^{\prime }}^{\prime }+(s_{1}-t_{l^{\prime}}^{\prime })v_{l^{\prime }}^{\prime },v^{\prime \prime })|d\Sigma
_{l^{\prime }}(s_{1})\Sigma_{l}(s)ds\\
\leq&C_{q,\ta}\left(T_0^{5/4}+\frac{1}{N}\right)e^{-\frac{\la_0}{2}t^{\rho_0}}\sup\limits_{0\leq s\leq t}e^{\frac{\la_0}{2} s^{\rho_0}}\|f(s)\|_\infty\\&+C_{q,\ta}e^{-\frac{\la_0}{2}t^{\rho_0}}\sup\limits_{0\leq s\leq t}e^{\frac{\la_0}{2} s^{\rho_0}}\|f(s)\|_2.
\end{split}
\end{equation}
As the proof for \eqref{J3}, our computation for $\CJ^1_{6,2}$ is divided into following several cases:

\noindent{\bf Case I:} $s_1>t_{l'}'-\frac{1}{k^2(s)}$, $k(s)$ is given by \eqref{ks}. From Lemma \ref{es.k}, we see that
$$
\iint\mathbf{k}^\chi(v_{l},v^{\prime })
\mathbf{k}^\chi(v_{l^{\prime }}^{\prime },v^{\prime \prime})<\infty.
$$
\eqref{in.de} implies that
$$
e^{-\nu (v^{\prime })(s-t_{1}^{\prime })}
\sqrt{\mu}(v^{\prime })\leq C_{q,\ta}e^{-\la_0(s-t_{1}^{\prime })^{\rho_0}}.
$$
And we get from \eqref{kd11} in Lemma \ref{k.d} that
\begin{equation*}
\begin{split}
\int_{\prod_{j=1}^{k-1}\mathcal{V}_{j}^{\prime }}&\sum_{l^{\prime }=1}^{k-1}\int_{t_{l^{\prime}}^{\prime}-\frac{1}{k^2(s)}}^{t_{l^{\prime}}^{\prime}}ds_{1}\mathbf{1}_{\{t_{l^{\prime }+1}^{\prime }\leq
0<t_{l^{\prime }}^{\prime }\}}
|f(s_{1},x_{l^{\prime }}^{\prime }+(s_{1}-t_{l^{\prime}}^{\prime })v_{l^{\prime }}^{\prime },v^{\prime \prime })|d\Sigma_{l^{\prime }}(s_{1})\\
\leq& \frac{C}{k(s)}e^{-\frac{\la_0}{2} (t'_1)^{\rho_0}}\sup\limits_{0\leq s_1\leq t_1'}e^{\frac{\la_0}{2} s_1^{\rho_0}}\|f(s_1)\|_\infty.
\end{split}
\end{equation*}
Substituting the above estimates into $\CJ^1_{6,2}$ and applying \eqref{kd1}, one has
\begin{equation*}
\begin{split}
\CJ^1_{6,2}\leq&\frac{C_{q,\ta}}{T_0^{5/4}}e^{-\frac{\la_0}{2}(t-t_1)^{\rho_0}}\int_{\prod_{j=1}^{k-1}
\mathcal{V}_{j}}\sum_{l=1}^{k-1}\int_{0}^{t_l}\mathbf{1}_{\{t_{l+1}\leq
0<t_{l}\}}e^{-\la_0(s-t_{1}^{\prime })^{\rho_0}}e^{-\la_0 (t'_1)^{\rho_0}}\\
&\times\sup\limits_{0\leq s_1\leq t_1'}e^{\la_0 s_1^{\rho_0}}\|f(s_1)\|_\infty\Sigma_{l}(s)ds\\
\leq&\frac{C_{q,\ta}}{T_0^{5/4}}e^{-\frac{\la_0}{2}(t-t_1)^{\rho_0}}\int_{\prod_{j=1}^{k-1}
\mathcal{V}_{j}}\sum_{l=1}^{k-1}\int_{0}^{t_l}\mathbf{1}_{\{t_{l+1}\leq
0<t_{l}\}}\\&\times\left\{e^{-\frac{\la_0}{2}s^{\rho_0}}\sup\limits_{0\leq s_1\leq s}e^{\frac{\la_0}{2} s_1^{\rho_0}}\|f(s_1)\|_\infty\right\}\Sigma_{l}(s)ds
\\
\leq&\frac{C_{q,\ta}}{T_0^{5/4}}e^{-\frac{\la_0}{2}(t-t_1)^{\rho_0}}e^{-\frac{\la_0}{2}t_1^{\rho_0}}\sup\limits_{0\leq s\leq t_1}e^{\frac{\la_0}{2} s^{\rho_0}}
\left\{e^{-\frac{\la_0}{2}s^{\rho_0}}\sup\limits_{0\leq s_1\leq s}e^{\frac{\la_0}{2} s_1^{\rho_0}}\|f(s_1)\|_\infty\right\}\\
\leq&\frac{C_{q,\ta}}{T_0^{5/4}}e^{-\frac{\la_0}{2}t^{\rho_0}}\sup\limits_{0\leq s\leq t}e^{\frac{\la_0}{2} s^{\rho_0}}\|f(s)\|_\infty.
\end{split}
\end{equation*}
\noindent{\bf Case II:} $s_1\leq t_{l'}'-\frac{1}{k^2(s)}$, by similar argument as in Case 1, Case 2 in the proof of \eqref{J3}, one can show that
if $|v_l|\geq N$ or $|v_{l'}'|\geq N$ or $|v_l|\leq N$ and $|v'|\geq 2N$, or $|v'_{l'}|\leq N$ and $|v''|\geq 2N$ with $N$ large enough, $\CJ^1_{6,2}$ bears the bound
$$
\frac{C_{q,\ta}}{N}e^{-\frac{\la_0}{2}t^{\rho_0}}\sup\limits_{0\leq s\leq t}e^{\frac{\la_0}{2} s^{\rho_0}}\|f(s)\|_\infty.
$$
Therefore, we need only to treat the case $|v_l|\leq N$, $|v'|\leq 2N$, $|v_{l'}'|\leq N$ and $|v''|\leq 2N$. As in Case 3 in the proof of \eqref{J3}, in this situation, one may also use the similar approximation \eqref{km} to obtain
\begin{equation*}
\begin{split}
\CJ^1_{6,2}\leq&\frac{C_{q,\ta}}{N}e^{-\frac{\la_0}{2}t^{\rho_0}}\sup\limits_{0\leq s\leq t}e^{\frac{\la_0}{2} s^{\rho_0}}\|f(s)\|_\infty\\&+C_{q,\ta}e^{-\frac{\la_0}{2}(t-t_1)^{\rho_0}}\int_{\prod_{j=1}^{k-1}
\mathcal{V}_{j}}\sum_{l=1}^{k-1}\int_{0}^{t_l}\mathbf{1}_{\{t_{l+1}\leq0<t_{l}\}}\iint dv^{\prime }dv^{\prime\prime}
\int_{\prod_{j=1}^{k-1}\mathcal{V}_{j}^{\prime }}\\&\quad\times e^{-\la_0(s-t_{1}^{\prime })^{\rho_0}}
\sum_{l^{\prime }=1}^{k-1}\int_{0}^{t_{l^{\prime}}^{\prime }}ds_{1}\mathbf{1}_{\{t_{l^{\prime }+1}^{\prime }\leq
0<t_{l^{\prime }}^{\prime }\}}|f(s_{1})|d\Sigma_{l^{\prime }}(s_{1})\Sigma_{l}(s)ds\\
\leq&\frac{C_{q,\ta}}{N}e^{-\frac{\la_0}{2}t^{\rho_0}}\sup\limits_{0\leq s\leq t}e^{\frac{\la_0}{2} s^{\rho_0}}\|f(s)\|_\infty\\&+C_{q,\ta}e^{-\frac{\la_0}{2}(t-t_1)^{\rho_0}}\int_{\prod_{j=1}^{k-1}
\mathcal{V}_{j}}\sum_{l=1}^{k-1}\int_{0}^{t_l}\mathbf{1}_{\{t_{l+1}\leq
0<t_{l}\}}\\&\times\left\{e^{-\frac{\la_1}{2}s^{\rho_0}}(k(s))^{7}\sup\limits_{0\leq s_1\leq s}e^{\frac{\la_1}{2} s_1^{\rho_0}}\|f(s_1)\|_2\right\}\Sigma_{l}(s)ds\\
\leq&\frac{C_{q,\ta}}{N}e^{-\frac{\la_0}{2}t^{\rho_0}}\sup\limits_{0\leq s\leq t}e^{\frac{\la_0}{2} s^{\rho_0}}\|f(s)\|_\infty
+C_{q,\ta}e^{-\frac{\la_0}{2}t^{\rho_0}}\sup\limits_{0\leq s\leq t}e^{\frac{\la_1}{2} s^{\rho_0}}\|f(s)\|_2.
\end{split}
\end{equation*}
Here $\la_0$ is chosen to be smaller then $\la_1$ so that $e^{-\frac{\la_1}{2}s^{\rho_0}}(k(s))^{7}\leq Ce^{-\frac{\la_0}{2}s^{\rho_0}}.$

Gathering the above estimates for $\CJ^1_{6,2}$, we see that \eqref{CJ62.es} is true. Once \eqref{CJ62.es} is obtained, the other terms in $\CJ_6$ and $\CJ_2$ can be treated in a similar fashion and after a tedious calculations it turns out that they share the same bound as \eqref{CJ62.es}. Namely, we obtain
\begin{equation}\label{CJ26}
\begin{split}
\CJ_2,\ \CJ_6
\leq&C_{q,\ta}\left(\frac{1}{T_0^{5/4}}+\frac{1}{N}\right)e^{-\frac{\la_0}{2}t^{\rho_0}}\sup\limits_{0\leq s\leq t}e^{\frac{\la_0}{2} s^{\rho_0}}\|f(s)\|_\infty\\&+C_{q,\ta}e^{-\frac{\la_0}{2}t^{\rho_0}}\sup\limits_{0\leq s\leq t}e^{\frac{\la_1}{2} s^{\rho_0}}\|f(s)\|_2.
\end{split}
\end{equation}
Now, substituting \eqref{CJ26} into \eqref{f26} and choosing $\eps, \eps_0>0$ suitably small and $N, T_0>0$ sufficiently large, we have
\begin{equation}\label{fif.es1}
\begin{split}
e^{\frac{\la_0}{2}t^{\rho_0}}\|f(t)\|_\infty\leq& C\|w_{q/2,\ta}f_0\|_\infty+C\sup\limits_{0\leq s\leq t}e^{\frac{\la_0}{2}s^{\rho_0}}\|w_{q/2,\ta}g_f(s)\|_\infty\\&+C\sup\limits_{0\leq s\leq t}e^{\frac{\la_1}{2} s^{\rho_0}}\|f(s)\|_2.
\end{split}
\end{equation}
Next, from \eqref{Ga.lif1} and \eqref{sol.es}, it follows that
\begin{equation}\label{gf.es}
\begin{split}
\|w_{q/2,\ta}g_f(s)\|_\infty=\|w_{q/2,\ta}\Ga(f,f)(s)\|_\infty\leq C\|w_{q,\ta}f(s)\|_\infty\|f(s)\|_\infty\leq C\vps_0\|f(s)\|_\infty.
\end{split}
\end{equation}
To control the last term in \eqref{fif.es1}, we appeal to deduce the exponential decay of $f$ in $L^2$. Notice that $f(t,x,v)$ as a global solution to \eqref{BE}, \eqref{ID} and \eqref{DBD} satisfies \eqref{sol.es}, we know thanks to \eqref{l-decay} in Proposition \ref{l2-lqn} that $f(t,x,v)$
also satisfies
\begin{equation}\label{f2d}
\begin{split}
\Vert f(t)\Vert _{2}\lesssim& e^{-\frac{\lambda_1}{2} t^{\rho_0}}\Bigg\{\|w_{q/2,\ta}f_0\|_2
+\sqrt{\int_{0}^{t}e^{\lambda_1 s^{\rho_0}}\Vert \nu^{-1/2}\Ga(f,f)(s)\Vert_{2}^{2}ds}
\\&+\sqrt{\int_{0}^{t}\Vert\nu^{-1/2} w_{q/2,\ta}\Ga(f,f)(s)\Vert_{2}^{2}ds}\Bigg\}.
\end{split}
\end{equation}
On the other hand, from Lemma \ref{es.nop} and the bound \eqref{sol.es}, it follows that
\begin{equation}\label{Ga21}
\begin{split}
\int_{0}^{t}e^{\lambda_1 s^{\rho_0}}\Vert \nu^{-1/2}\Ga(f,f)(s)\Vert
_{2}^{2}ds\leq& C\int_{0}^{t}e^{\lambda_1 s^{\rho_0}}\|w_{q/2,\ta}f(s)\|^2_{\infty}\|f(s)\|^2_{\nu}ds\\
\leq& C\sup\limits_{0\leq s\leq t}\|w_{q/2,\ta}f(s)\|^2_{\infty}\int_{0}^{t}e^{\lambda_1 s^{\rho_0}}\|f(s)\|^2_{\nu}ds\\
\leq& C\vps^2_0\sup\limits_{0\leq s\leq t}\|w_{q/2,\ta}f(s)\|^2_{\infty},
\end{split}
\end{equation}
and similarly
\begin{equation}\label{Ga22}
\begin{split}
\int_{0}^{t}\Vert \nu^{-1/2}w_{q/2,\ta}\Ga(f,f)(s)\Vert
_{2}^{2}ds
\leq& C\int_{0}^{t}\|w_{q,\ta}f(s)\|^2_{\infty}\|w_{q/2,\ta}f(s)\|_\nu^2ds
\\ \leq& C\vps_0^2\int_{0}^{t}\|w_{q/2,\ta}f(s)\|_\nu^2ds
\leq C\|w_{q,\ta}f_0\|^2_\infty.
\end{split}
\end{equation}

Consequently, \eqref{f2d}, \eqref{Ga21} and \eqref{Ga22} give rise to
\begin{equation}\label{f2dc}
\begin{split}
e^{\frac{\lambda_1}{2} t^{\rho_0}}\Vert f(t)\Vert _{2}\leq C\|w_{q,\ta}f_0\|_\infty
+C\vps_0\sup\limits_{0\leq s\leq t}\|f(s)\|_{\infty}.
\end{split}
\end{equation}
Now plugging \eqref{f2dc}  and \eqref{gf.es} into \eqref{fif.es1} 
leads us to
\begin{equation*}
\begin{split}
e^{\frac{\la_0}{2}t^{\rho_0}}\|f(t)\|_\infty\leq& C\|w_{q,\ta}f_0\|_\infty.
\end{split}
\end{equation*}
This completes the proof of the second part of Theorem \ref{ms1}. Therefore we conclude the proof of Theorem \ref{ms1}.

\end{proof}

\section{Specular reflection boundary value problem}\label{SVP}

\subsection{$L^2$ theory for the linearized equation}
Let us look at the boundary value problem for the linearized homogeneous equation
\begin{equation}
\partial _{t}f+v\cdot \nabla _{x}f+Lf=0,\text{ \ \ }f(0)=f_{0},\quad \text{
in }(0,\infty)\times \Omega \times \R^{3},  \label{sleq}
\end{equation}
\begin{equation}\label{slbd}
f(t,x,v)|_{\ga_-}=f(t,x,R_x v),\ \ \textrm{on}\  [0,\infty)\times\gamma _{-}.
\end{equation}
We first show that the macroscopic part of the solution of \eqref{sleq} and \eqref{slbd} can be dominated by the microscopic
part on the time interval $[0,1]$.
\begin{proposition}\label{smm}
Let $f(t,x,v)\in L^\infty([0,1],L^2(\Omega\times\R^3))$ be a solution to \eqref{sleq} and \eqref{slbd}, and $f_\ga\in L^2([0,1],L^2(\pa\Omega\times\R^3))$, then there exists $\de_0>0$ such that
\begin{equation}\label{smm.ex}
\int_0^1(Lf,f)ds\geq \de_0\int_0^1\| f(s)\|_\nu^2ds.
\end{equation}

\end{proposition}
\begin{proof}
The proof is based on contradiction and it is divided into four steps.

\noindent{\it Step 1. Proof of contradiction.} If Proposition \ref{smm} is false, then no $\de_0$ exists as in Proposition \ref{smm}. Hence for any $n\geq1$, there exists a sequence of non-zero $f_n\in L^\infty([0,1],L^2(\Omega\times\R^3))$ to the linearized Boltzmann equation \eqref{sleq} such that
\begin{equation}\label{op.mm}
0\leq\int_0^1(Lf_n,f_n)ds\leq\frac{1}{n} \int_0^1\| f_n(s)\|_\nu^2ds.
\end{equation}
Since $f_n$ satisfies
$$\pa_tf_n+v\cdot\na_xf_n+Lf_n=0,\quad \text{
in }(0,1]\times \Omega \times \R^{3},$$
and
$$
f_n(t,x,v)|_{\ga_-}=f_n(t,x,R_xv),\ \ \textrm{on}\  [0,1]\times\gamma _{-}.
$$
With this and by a similar argument as for obtaining Lemma 8 in \cite[pp.340]{G-soft}, one has
\begin{equation}\label{fnul}
\sup\limits_{0\leq t\leq 1}\|\nu^{1/2}f_n(t)\|_2^2\leq C\|\nu^{1/2}f_{n}(0)\|_2^2,\ \ \int_0^1\|f_n(s)\|^2_\nu ds\geq C\|\nu^{1/2}f_{n}(0)\|_2^2.
\end{equation}
Assume $f_{n}(0)$ is not identical to zero and set
$$
Z_n=\frac{f_n(t,x,v)}{\sqrt{\int_0^1\|f_n(s)\|_\nu^2ds}},
$$
then
\begin{equation}\label{zn}
\int_0^1\|Z_n(s)\|_\nu^2ds=1,
\end{equation}
and
\eqref{op.mm} is equivalent to
\begin{equation}\label{znlm}
0\leq\int_0^1(LZ_n,Z_n)ds\leq \frac{1}{n}.
\end{equation}
\eqref{zn} and \eqref{znlm} imply there exists $Z(t,x,v)$ such that
$$
Z_n\rightarrow Z\ \textrm{weakly in} \ \ \int_0^1\|\cdot\|_\nu^2ds,
$$
and
\begin{equation}\label{miZ0}
\int_0^1(LZ_n,Z_n)ds=\int_0^1(L(\mathbf{I}-\mathbf{P})Z_{n},(\mathbf{I}-\mathbf{P})Z_{n})ds\rightarrow0.
\end{equation}
Notice that it is straightforward to verify
\begin{equation*}
\mathbf{P}Z_{n}\rightarrow \mathbf{P}Z, \ (\mathbf{I}-\mathbf{P})Z_{n}\rightarrow (\mathbf{I}-\mathbf{P})Z,\ \text{ weakly in }\int_{0}^{1}||\cdot
||_{\nu }^{2}ds.
\end{equation*}%
It follows from \eqref{miZ0} that $(\mathbf{I}-\mathbf{P)}Z=0$, therefore
\begin{equation*}
Z(t,x,v)=\{a(t,x)+v\cdot b(t,x)+|v|^{2}c(t,x)\}\sqrt{\mu }.
\end{equation*}
Moreover, we have from $\pa_tf_n+v\cdot\na_xf_n+Lf_n=0$ that
\begin{equation}\label{zneq}
\pa_tZ_n+v\cdot\na_xZ_n+LZ_n=0,
\end{equation}
which yields
\begin{equation}\label{zeq}
\pa_tZ+v\cdot\na_xZ=0.
\end{equation}
In what follows, we will show on the one hand $Z=0$ from \eqref{zeq} and the inherited boundary condition \eqref{slbd}. On the other hand,
$Z_{n}$ will be proven to converge strongly to $Z$ in $\int_{0}^{1}\|\cdot
\|_\nu^{2}ds,$ and $\int_{0}^{1}\|Z\|_{\nu}^{2}ds\neq 0.$ This leads to a
contradiction.

\noindent{\it Step 2. The limit function $Z(t,x,v)$.}
\begin{lemma}
\label{limit} There exists constants $a_{0},c_0,c_{1},c_{2},$ and constant
vectors $b_{0},b_{1}$ and $\varpi $ such that $Z(t,x,v)$ takes the form:%
\begin{equation*}
\begin{split}
\Bigg(&\left\{\frac{c_{0}}{2}|x|^{2}-b_{0}\cdot
x+a_{0}\right\}+\left\{-c_{0}tx-c_{1}x+\varpi \times x+b_{0}t+b_{1}\right\}\cdot v\\&+\left\{\frac{%
c_{0}t^{2}}{2}+c_{1}t+c_{2}\right\}|v|^{2}\Bigg) \sqrt{\mu }.
\end{split}
\end{equation*}%
Moreover, these constants are finite:%
\begin{equation*}
|a_{0}|+|c_{0}|+|c_{1}|+|c_{2}|+|b_{0}|+|b_{1}|+|\varpi |<+\infty .
\end{equation*}
\end{lemma}
\begin{proof}
See Lemma 6 in \cite[pp.736]{Guo-2010}.
\end{proof}

\noindent{\it Step 3. Compactness.} To show the strong convergence $\lim\limits_{n\rightarrow \infty }\int_{0}^{1}\| \{Z_{n}-Z\}(s)\|_\nu^{2}ds=0,$ we resort to the Averaging Lemma. 
\begin{lemma}
\label{interior}
Up to a subsequence, it holds that $
\lim\limits_{k\rightarrow \infty }\int_{0}^{1}\|\{Z_{n}-Z\}(s)\|_\nu^{2}ds=0.$
\end{lemma}
\begin{proof}
Define
\begin{equation*}
\Omega _{\varepsilon ^{4}}\equiv \{x\in \Omega :\xi (x)<-\varepsilon ^{4}\}.
\end{equation*}
Choose any $\eta_0>0$ and a smooth cutoff function $\chi_1 (t,x,v)$ in $(0,1)\times \Omega \times\R^3$, such that $\chi_1 (t,x,v)=1$
in $[\eta_0,1-\eta_0]\times \Omega\setminus\Omega_{\vps^4} \times\{|v|\leq \frac{1}{\eta_0}\}$.
Next,
multiplying the equation \eqref{zneq} by $\chi_1 $, we obtain
\begin{equation*}
\lbrack \partial _{t}+v\cdot \nabla _{x}]\{\chi_1 Z_{n}\}=\{[\partial
_{t}+v\cdot \nabla _{x}]\chi_1 \}Z_{n}-\chi_1 LZ_{n}.
\end{equation*}%
Since $f_n\in L^\infty([0,1],L^2(\Omega\times\R^3))$, one sees that $\chi_1Z_n\in L^2([0,1],L^2(\Omega\times\R^3))$ and
$\{[\partial
_{t}+v\cdot \nabla _{x}]\chi_1 \}Z_{n}-\chi_1 LZ_{n}\in L^2([0,1],L^2(\Omega\times\R^3))$, then we know from the
Averaging Lemma cf. \cite{DL, DL2}, $\int \chi_1 Z_{n}e(v)dv$ are compact in
$L^{2}([0,1]\times \Omega )$ for any exponential decay function $e(v)$.
On the other hand, as \eqref{fnul}, from \eqref{zneq}, it follows that
\begin{equation*}
\sup\limits_{0\leq t\leq 1}\|\nu^{1/2}Z_n(t)\|_2^2\leq C\|\nu^{1/2}Z_{n}(0)\|_2^2,\ \ \int_0^1\|Z_n(s)\|^2_\nu ds\geq C\|\nu^{1/2}Z_{n}(0)\|_2^2.
\end{equation*}
Using this, one deduce
\begin{equation*}
\begin{split}
\int_0^1\int_\Omega&\left(\int(1-\chi_1) Z_{n}e(v)dv\right)^2dxds+\int_0^1\int_\Omega\left(\int(1-\chi_1) Ze(v)dv\right)^2dxds\\
\leq&C\int_0^1\int_{\Omega\times\R^3}\left\{(1-\chi_1)^2 Z_{n}^2e(v)+(1-\chi_1)^2 Z^2e(v)\right\}dvdxds\\
\leq&C\int_{0\leq s\leq \eta_0}\int_{\Omega\times\R^3}+C\int_{1-\eta_0\leq s\leq 1}\int_{\Omega\times\R^3}+C\int_0^1\int_{\Omega}\int_{|v|\geq\frac{1}{\eta_0}}\\
\leq& C\eta_0\sum\limits_{0\leq s\leq1}\int_{\Omega\times\R^3}(1+|v|)^\varrho(Z_{n}^2+Z^2)dvdx\leq C\eta_0.
\end{split}
\end{equation*}
Therefore, up to a subsequence,
the macroscopic parts of $Z_{k}$ satisfy $\mathbf{P}Z_{k}\rightarrow
\mathbf{P}Z=Z$ strongly in $L^{2}([0,1]\times \Omega \times
\R^{3}).$ Therefore, in light of $\int_{0}^{1}||(\mathbf{I}-\mathbf{%
P)}Z_{k}(s)||_{\nu }^{2}ds\rightarrow 0$ in \eqref{miZ0}, 
we conclude our lemma.

\end{proof}

\noindent{\it Step 4. Boundary condition leads to $Z=0$.} Performing the same calculations as that of Section 3.6 in \cite[pp.747]{Guo-2010}
, we see that $Z=0$, and this leads to a contradiction,
so finished up the proof of Proposition \ref{smm}.

\end{proof}

Once the coercivity estimate \eqref{smm.ex} is obtained, like Proposition \ref{l2-lqn}, one can now deduce the basic energy estimates and time decay rates as follows:
\begin{lemma}\label{sl2.eng}
Assume $f(t,x,v)$ satisfies \eqref{sleq} and \eqref{slbd}, then it holds that
\begin{equation}\label{seng1}
\|f(t)\|_2^2+\int_0^t\|f\|^2_\nu\leq C\|f_0\|_2^2,
\end{equation}
and
\begin{equation}\label{seng2}
\|w_{q/4,\ta}f(t)\|_2^2+\int_0^t\|w_{q/4,\ta}f\|^2_\nu\leq C\|w_{q/4,\ta}f_0\|_2^2.
\end{equation}
Moreover, there exists $\la>0$ such that
\begin{equation}\label{seng3}
\|f(t)\|_2^2+e^{-\la t^{\rho_0}}\int_0^te^{\la s^{\rho_0}}\|f\|^2_\nu\leq Ce^{-\la t^{\rho_0}}\|w_{q/4,\ta}f_0\|_2^2,
\end{equation}
here $\rho_0$ is given as in Proposition \ref{l2-lqn}.
\end{lemma}
\begin{proof}
We prove \eqref{seng3} only, the proof for \eqref{seng1} and \eqref{seng2} being similar and easier. Taking the inner product of \eqref{sleq} with $e^{\la t^{\rho_0}}f$ over $\Omega\times\R^3$, one has
\begin{equation}\label{sin}
\frac{d}{dt}\left\{e^{\la t^{\rho_0}}\|f(t)\|_2^2\right\}+2(e^{\la t^{\rho_0}}Lf,f)=\la\rho_0 t^{\rho_0-1}e^{\la t^{\rho_0}}\|f(t)\|_2^2.
\end{equation}
For any $t>0$, there exists a nonnegative integer $N$ such that $t\in[N,N+1)$.
For the time interval $[0,N]$ (we may assume without lose of generality $N\geq1$), it follows that
\begin{equation*}
e^{\la N^{\rho_0}}\|f(N)\|_2^2+2\int_0^N(e^{\la s^{\rho_0}}Lf,f)ds=\|f_0\|^2_2+\la\rho_0 \int_0^Ns^{\rho_0-1}e^{\la s^{\rho_0}}\|f(s)\|_2^2ds.
\end{equation*}
Split the time interval into $\cup_{j=0}^{N-1}[j,j+1)$ and define $f_j(s,x,v)=f(j+s,x,v)$ for $j=0,1,2,\cdots,N-1,$ to deduce
\begin{equation*}
\begin{split}
e^{\la N^{\rho_0}}\|f(N)\|_2^2&+2\sum\limits_{j=1}^{N-1}\int_0^1(e^{\la (j+s)^{\rho_0}}Lf_j,f_j)ds\\
\leq&\|f_0\|^2_2+\la\rho_0 \sum\limits_{j=1}^{N-1}\int_0^1(j+s)^{\rho_0-1}e^{\la (j+s)^{\rho_0}}\|f_j(s)\|_2^2ds,
\end{split}
\end{equation*}
which further implies
\begin{equation}\label{eng0N}
\begin{split}
e^{\la N^{\rho_0}}\|f(N)\|_2^2&+2\sum\limits_{j=1}^{N-1}\int_0^1(e^{\la j^{\rho_0}}Lf_j,f_j)ds\\
\leq&\|f_0\|^2_2+C\la\rho_0 \sum\limits_{j=1}^{N-1}\int_0^1j^{\rho_0-1}e^{\la j^{\rho_0}}\|f_j(s)\|_2^2ds,
\end{split}
\end{equation}
for $0<\rho_0<1.$

On the other hand, we get from \eqref{smm.ex} that
\begin{equation}\label{sLj}
\sum\limits_{j=1}^{N-1}\int_0^1(e^{\la j^{\rho_0}}Lf_j,f_j)ds\geq\de_0\sum\limits_{j=1}^{N-1}\int_0^1e^{\la j^{\rho_0}}\|f_j\|_{\nu}^2ds.
\end{equation}
Substituting \eqref{sLj} into \eqref{eng0N} leads us to
\begin{equation}\label{eng0Nr}
\begin{split}
e^{\la N^{\rho_0}}\|f(N)\|_2^2&+\sum\limits_{j=1}^{N-1}\int_0^1e^{\la j^{\rho_0}}\|f_j\|_{\nu}^2ds\\ \leq& C\|f_0\|^2_2+C\la\rho_0 \sum\limits_{j=1}^{N-1}\int_0^1j^{\rho_0-1}e^{\la j^{\rho_0}}\|f_j(s)\|_2^2ds.
\end{split}
\end{equation}
To handle the integral on the right hand side of the above inequality, we decompose the velocity integration domain as
$$
E_j=\{v\mid j^{\rho_0-1}\leq \ka'_0\nu\},\ \ E^c_j=\{v\mid j^{\rho_0-1}> \ka'_0\nu\},
$$
where $\ka'_0>0$ and small enough.
Therefore, for $\la=\frac{q}{16} (\ka'_0)^{\frac{\rho_0}{1-\rho_0}}$, it follows
\begin{equation}\label{sdec}
\begin{split}
\sum\limits_{j=1}^{N-1}&\int_0^1j^{\rho_0-1}e^{\la j^{\rho_0}}\|f_j(s)\|_2^2ds\\
\leq& \ka'_0\sum\limits_{j=1}^{N-1}\int_0^1e^{\la j^{\rho_0}}\|f_j(s)\|_\nu^2ds
\\&+C\sum\limits_{j=1}^{N-1}\int_0^1j^{\rho_0-1}e^{-\la j^{\rho_0}}
e^{2\la (\ka'_0)^{\frac{\rho_0}{\rho_0-1}}\nu^{\frac{\rho_0}{\rho_0-1}}}\|{\bf 1}_{E_j^c}f_j(s)\|_2^2ds\\
\leq& \ka'_0\sum\limits_{j=1}^{N-1}\int_0^1e^{\la j^{\rho_0}}\|f_j(s)\|_\nu^2ds
+C\sum\limits_{0\leq s\leq N}\|w_{q/4,\ta}f(s)\|_2^2\sum\limits_{j=1}^{N-1}j^{\rho_0-1}e^{-\la j^{\rho_0}}.
\end{split}
\end{equation}
Putting \eqref{sdec} back to \eqref{eng0Nr} and noticing that $\sum\limits_{j=1}^{N-1}j^{\rho_0-1}e^{-\la j^{\rho_0}}<\infty$, we arrive at
\begin{equation*}
e^{\la N^{\rho_0}}\|f(N)\|_2^2+\sum\limits_{j=1}^{N-1}\int_0^1e^{\la j^{\rho_0}}\|f_j\|_{\nu}^2ds\leq C\|f_0\|^2_2+C\|w_{q/4,\ta}f_0\|_2^2,
\end{equation*}
where we used \eqref{seng2}. Changing back to $f_j(s)=f(s+j)$ and using $e^{(j+s)^{\rho_0}-s^{\rho_0}}\leq e^{j^{\rho_0}}$, one further has
\begin{equation}\label{eng0N2}
e^{\la N^{\rho_0}}\|f(N)\|_2^2+\int_0^Ne^{\la s^{\rho_0}}\|f(s)\|_{\nu}^2ds\leq C\|f_0\|^2_2+C\|w_{q/4,\ta}f_0\|_2^2.
\end{equation}
Now integrate \eqref{sin} over $[N,t]$ to obtain
\begin{equation}\label{engNt1}
e^{\la t^{\rho_0}}\|f(t)\|_2^2+\int_N^te^{\la s^{\rho_0}}(Lf,f)ds\leq \la\rho_0 \int_N^ts^{\rho_0-1}e^{\la s^{\rho_0}}\|f(s)\|_2^2ds+e^{\la N^{\rho_0}}\|f(N)\|_2^2.
\end{equation}
Thanks to Lemma \ref{es.k}, one has
\begin{equation}\label{engNt2}
\int_N^te^{\la s^{\rho_0}}(Lf,f)ds\geq \de\int_N^te^{\la s^{\rho_0}}\|f(s)\|_{\nu}^2ds-C\int_N^te^{\la s^{\rho_0}}\|{\bf 1}_{|v|\leq C}f(s)\|_{\nu}^2ds.
\end{equation}
From \eqref{engNt1} and \eqref{engNt2}, it follows that
\begin{equation*}
\begin{split}
e^{\la t^{\rho_0}}\|f(t)\|_2^2&+\de\int_N^te^{\la s^{\rho_0}}\|f(s)\|_{\nu}^2ds\\ \leq& \la\rho_0 \int_N^ts^{\rho_0-1}e^{\la s^{\rho_0}}\|f(s)\|_2^2ds+C\|f_0\|_2^2
+e^{\la N^{\rho_0}}\|f(N)\|_2^2,
\end{split}
\end{equation*}
where the fact that $\int_N^te^{\la s^{\rho_0}}\|{\bf 1}_{|v|\leq C}f(s)\|_{\nu}^2ds\leq C\sum\limits_{0\leq s\leq t}\|f(s)\|_2^2\leq C\|f_0\|_2^2$ was used.
We then have by performing the similar calculations as for obtaining \eqref{eng0N2}
 \begin{equation}\label{engNt3}
e^{\la t^{\rho_0}}\|f(t)\|_2^2+\int_N^te^{\la s^{\rho_0}}\|f(s)\|_{\nu}^2ds\leq C\|f_0\|^2_2+C\|w_{q/4,\ta}f_0\|_2^2+Ce^{\la N^{\rho_0}}\|f(N)\|_2^2.
\end{equation}
Thereby, \eqref{seng3} follows from \eqref{eng0N2} and \eqref{engNt3}. This finishes the proof of Lemma \ref{sl2.eng}.



\end{proof}

\subsection{$L^\infty$ theory for the linearized equation}

Recall
\begin{equation*}
w_{q,\ta,\vth}=\exp\left\{\frac{q|v|^\ta}{8}+\frac{q|v|^\ta}{8(1+t)^\vth}\right\}, \ (q,\ta)\in \CA_{q,\ta},\ 0\leq\vth< -\frac{\ta}{\varrho}.
\end{equation*}
Let $h=w_{q,\ta,\vth}(t,v)f(t,x,v)$,
the problem \eqref{sleq} and \eqref{slbd} is now equivalent to
\begin{equation}
\partial_{t}h+v\cdot \nabla _{x}h+ \left(\nu +\frac{\vth q|v|^\ta}{8(1+t)^{\vth+1}}\right)h=K_{\overline{w}}h,\text{ \ \ }h(0)=h_{0},\quad \text{
in }(0,\infty)\times \Omega \times \R^{3},  \label{ssln.eq}
\end{equation}
with
\begin{equation}\label{sslbd}
h(t,x,v)|_{\ga_-}=h(t,x,R_{x}v),\ \ \textrm{on}\  [0,\infty)\times\gamma _{-}.
\end{equation}
Here $K_{\overline{w}}h=w_{q,\ta,\vth}K\left(\frac{h}{w_{q,\ta,\vth}}\right)$  as in the subsection \ref{L2th}. 

We express solution, $h(t,x,v)$, to \eqref{ssln.eq} and \eqref{sslbd} through semigroup $U(t)$ as
\begin{equation*}
h(t,x,v)=\{U(t)h_0\}(x,v),
\end{equation*}
with initial boundary data given by
$$
\{U(0)h_0\}(x,v)=h_0(x,v), \ \textrm{and} \ U(0)h_0(x,v)|_{\ga_-}=h_0(x,R_{x}v).
$$
For the sake of simplicity, we denote
\begin{equation*}
\widetilde{\nu}(v,t)=\nu +\frac{\vth q|v|^\ta}{8(1+t)^{\vth+1}}.
\end{equation*}
It is obvious to see $\widetilde{\nu}^{-1}< \nu^{-1},$ which plays a significant role in the later proof.

Applying Young's inequality, one can see that there exists $C_{\varrho,q,\vth}>0$ independent of $v$ such that
\begin{equation}\label{IP.lbd1}
\widetilde{\nu}(v,t)\geq C_{\varrho,q,\vth}(1+t)^{\frac{(1+\vth)\varrho}{\ta-\varrho}},
\end{equation}
and for $t>0$, one sees that
\begin{equation}\label{IP.lbd1s}
C_{\varrho,q,\vth}(1+t)^{\frac{(1+\vth)\varrho}{\ta-\varrho}}\thicksim C_{\varrho,q,\vth}t^{\frac{(1+\vth)\varrho}{\ta-\varrho}}.
\end{equation}
From \eqref{IP.lbd1} and \eqref{IP.lbd1s}, it follows
\begin{equation}\label{IP.lbd2}
e^{-\int_s^t\widetilde{\nu}(v,\tau)d\tau}\leq \exp\left(-\la_2\left\{t^{\frac{\ta+\vth\varrho}{\ta-\varrho}}-s^{\frac{\ta+\vth\varrho}{\ta-\varrho}}\right\}\right)\eqdef e^{\la_2s^{\rho_1}-\la_2t^{\rho_1}},\ t\geq s\geq0,
\end{equation}
here $\rho_1=\frac{\ta+\vth\varrho}{\ta-\varrho}$ with $\ta+\vth\varrho>0$, moreover  $\la_2>0$ is independent of $v$.

Our goal in this subsection will be to prove the following
\begin{proposition}
\label{specularbd}Let $0<\vth< -\frac{\ta}{\varrho}$ with $-3<\varrho<0$ and $(q,\ta)\in\CA_{q,\ta}$. Assume that $\xi $ is
both strictly convex \eqref{scon} and analytic, and the mass \eqref{mass.c} and energy \eqref{eng.c} are conserved. In the case of $\Omega $ has
rotational symmetry \eqref{axis}, we also assume conservation of
corresponding angular momentum \eqref{axiscon}. Let $h_{0}\in L^{\infty }.$ There exist $\la_0>0$ and $C>0$ such that
\eqref{ssln.eq} and \eqref{sslbd} admit a unique solution $U(t)h_0$ satisfying
\begin{equation}\label{Ubd}
\left\|U(t)h_0\right\|_{\infty}\leq Ce^{-\frac{\la_0}{2}t^{\rho_1}}\left\|h_{0}\right\|_{\infty},
\end{equation}
where $\rho_1=\frac{\ta+\vth\varrho}{\ta-\varrho}$.
\end{proposition}
The Duhamel Principle will be applied to prove Proposition \ref{specularbd} and the first step is an appropriate decomposition. Initially, we look for solutions to the linearized equation \eqref{ssln.eq} with the almost compact operator $K_{\overline{w}}$ removed. Namely, we first consider
\begin{equation}
\partial _{t}h+v\cdot \nabla _{x}h+\widetilde{\nu}(v,t)h=0,\text{ \ \ }h(0)=h_{0},\quad \text{
in }(0,\infty)\times \Omega \times \R^{3},  \label{sl.eq}
\end{equation}
with
\begin{equation}\label{slb}
h(t,x,v)|_{\ga_-}=h(t,x,R_{x}v),\ \ \textrm{on}\  [0,\infty)\times\gamma _{-}.
\end{equation}
Let us denote the solution to \eqref{sl.eq} and \eqref{slb} by semigroup $G(t)h_{0}.$

Priori to investigating the properties of the solution operators $U(t)$ and $G(t)$, we give the following definition

\begin{definition}
\label{specularcycles}Let $\Omega$
be convex \eqref{scon}. Fix any
point $(t,x,v)\notin \gamma _{0}\cap \gamma _{-},$ and define $%
(t_{0},x_{0},v_{0})=(t,x,v)$, and for $k\geq 1$
\begin{equation}
(t_{k+1},x_{k+1},v_{k+1})=(t_{k}-t_{\mathbf{b}}
(t_{k},x_{k},v_{k}),x_{%
\mathbf{b}}(x_{k},v_{k}),R_{x_{k+1}}v_{k}),  \label{specularcycle}
\end{equation}%
where $R_{x_{k+1}}v_{k}=v_{k}-2(v_{k}\cdot n(x_{k+1}))n(x_{k+1}).$ And we
define the specular back-time cycle
\begin{equation*}
X_{\mathbf{cl}}^{{}}( s )\equiv \sum_{k=1}\mathbf{1}%
_{[t_{k+1},t_{k})}( s )\left\{ x_{k}+v_{k}( s -t_{k})\right\} ,\text{ \ \ }V_{%
\mathbf{cl}}^{{}}( s )\equiv \sum_{k=1}\mathbf{1}%
_{[t_{k+1},t_{k})}( s )v_{k}.
\end{equation*}
\end{definition}
\begin{lemma}\label{slG}
Let $h_0\in L^\infty(\Omega\times\R^3)$. There exists a unique solution $G(t)h_0$ to
\begin{equation*}
\left\{\partial _{t}+v\cdot \nabla _{x}+\widetilde{\nu}(v,t)\right\}\{G(t)h_{0}\}=0,\text{ \ \ \ \ }%
\left\{G(0)h_{0}\right\}=h_{0},  \label{gh0}
\end{equation*}%
with the specular reflection $\{G(0)h_{0}\}(t,x,v)
=\{G(0)h_{0}\}(t,x,R_{x}v)$
for $x\in \partial \Omega .$ For almost any $(x,v)\in \overline{\Omega}\times
\R^{3}\setminus \gamma _{0},$
\begin{equation}
\begin{split}
\{G(t)h_{0}\}(t,x,v)
=&e^{-\int_0^t\widetilde{\nu}(v,\tau)d\tau}h_{0}
\left( X_{\mathbf{cl}}(0),V_{\mathbf{cl}}(0)\right)\\
=&\sum\limits_{k}^\infty\mathbf{1}_{[t_{k+1},t_{k})}(0)
e^{-\int_0^t\widetilde{\nu}(v,\tau)d\tau}h_{0}(x_{k}-t_{k}v_{k},v_{k}).
\label{bouncebackformular}
\end{split}
\end{equation}%
Here, we define $t_k=0$ if $t_k<0.$

Moreover, it holds
\begin{equation}\label{Gdecay0}
\left\|G(t)h_{0}\right\|_{\infty }\leq \left\|e^{-\int_0^t\widetilde{\nu}(v,\tau)d\tau}h_{0}\right\|_\infty,
\end{equation}
and there exists $\la_2>0$ such that
\begin{equation}\label{Gdecay}
\left\|G(t)h_{0}\right\|_{\infty }
\leq Ce^{-\la_2t^{\rho_1}}\left\|h_{0}\right\|_{\infty },\ \ t\geq0,
\end{equation}
and
\begin{equation}\label{Gdecay2}
\left\|G(t-s)h(s)\right\|_{\infty }
\leq Ce^{-\la_2\left\{t^{\rho_1}-s^{\rho_1}\right\}}\left\|h(s)\right\|_{\infty },\ \ t\geq s\geq0.
\end{equation}

\end{lemma}
\begin{proof}
The proof for \eqref{bouncebackformular} and \eqref{Gdecay0} is the same as that of Lemma 15 in \cite[pp.757]{Guo-2010}. \eqref{Gdecay} and
\eqref{Gdecay2}
directly follows from \eqref{IP.lbd2} and \eqref{Gdecay0}, this completes the proof of Lemma \ref{slG}.
\end{proof}

The following lemma shows that the solution operator $G(t)h_{0}$ is indeed continuous away from the grazing set. \begin{lemma}\label{specularcon}\cite[Lemma 21, pp.768]{Guo-2010}
Let $\xi $ be convex as in \eqref{scon}. Let $h_{0}$
be continuous in $\bar{\Omega}\times \R^{3}\setminus \gamma _{0}\,\ $
and $g(t,x,v)$ be continuous in the interior of $[0,\infty)\times \Omega
\times \R^{3}$ and $\sup_{[0,\infty)\times \Omega \times \R^{3}}|\frac{g(t,x,v)}{\widetilde{\nu} (v,t)}|<\infty .$ Assume that on $\gamma _{-},$ $h_{0}(x,v)=h_{0}(x,R(x)v)$. Then the specular solution $h(t,x,v)$ to
\begin{equation*}
\partial _{t}h+v\cdot \nabla _{x}h+\widetilde{\nu}(v,t) h= g(t,x,v),\text{ \ \ }h(0)=h_{0},\quad \text{
in }(0,\infty)\times \Omega \times \R^{3},
\end{equation*}
with
\begin{equation*}
h(t,x,v)|_{\ga_-}=h(t,x,R_{x}v),\ \ \textrm{on}\  [0,\infty)\times\gamma _{-},
\end{equation*}
 is continuous on $[0,\infty)\times \{\bar{\Omega}\times \R^{3}\setminus \gamma _{0}\}.$
\end{lemma}

We now go back to \eqref{ssln.eq} and \eqref{sslbd}, from Duhamel formula, it follows
\begin{equation*}
\{U(t)h_0\}(x,v)
=G(t)h_{0}(x,v)
+\int_0^{t}ds~G(t- s )K_{\overline{w}}
\{U( s )h_0\}(x,v).
\end{equation*}
Employing the decomposition $K_{\overline{w}}=K^{\chi}_{\overline{w}}+K^{1-\chi}_{\overline{w}}$ again,
we then expand out
\begin{equation*}
\begin{split}
\{U(t)h_0\}(x,v)
=&G(t)h_{0}(x,v)
+\int_0^{t}d s ~G(t- s )K^{1-\chi}_{\overline{w}}
\{U( s )h_0\}(x,v)\\
&+\int_0^{t}d s ~G(t- s )K^{\chi}_{\overline{w}}
\{U( s )h_0\}(x,v).
\end{split}
\end{equation*}
We further iterate the Duhamel formula of the last term, as did in \cite{Vi}
\begin{equation*}
\{U( s )h_0\}(x,v)
=G( s )h_{0}(x,v)
+\int_0^{ s }d s _1~G( s - s _1)K_{\overline{w}}
\{U( s _1)h_0\}(x,v).
\end{equation*}
Substituting this into previous expression and using $K_{\overline{w}}=K^{\chi}_{\overline{w}}+K^{1-\chi}_{\overline{w}}$ again yield a more elaborate formula
\begin{equation}\label{Unex}
\begin{split}
\{U(t)h_0\}(x,v)=&G(t)h_{0}(x,v)
+\int_0^{t}d s ~G(t- s )K^{1-\chi}_{\overline{w}}
\{U( s )h_0\}(x,v)\\
&+\int_0^{t}d s ~G(t- s )K^{\chi}_{\overline{w}}
\{G( s )h_{0}\}(x,v)\\
&+\int_0^{t}d s \int_0^{ s }d s _1~G(t- s )
K^{\chi}_{\overline{w}}
G( s - s _1)K^{1-\chi}_{\overline{w}}
\{U( s _1)h_0\}(x,v)\\
&+\int_0^{t}d s \int_0^{ s }d s _1~G(t- s )
K^{\chi}_{\overline{w}}
G( s - s _1)K^{\chi}_{\overline{w}}
\{U( s _1)h_0\}(x,v)\\
\eqdef& \sum\limits_{l=1}^5H_l(t,x,v).
\end{split}
\end{equation}
For any fixed point $(t,x,v)$ with $(x,v)\notin \gamma _{0},$ let the
back-time specular cycle of $(t,x,v)$ be
$[x_{\mathbf{cl}}( s ),v_{\mathbf{cl}}( s )],$ then the most delicate term $H_5$ in \eqref{Unex} can be rewritten as
\begin{equation*}
\begin{split}
H_5(t,x,v)=&\int_{0}^{t}d s \int_{0}^{ s }d s _1\int dv^{\prime }d v^{\prime \prime }e^{-\int^{t}_s\widetilde{\nu}(v,\tau)d\tau-\int_{s_1}^s\widetilde{\nu}(v',\tau)d\tau}\\
&\times\mathbf{k}^{\chi}_{\overline{w}}(V_{\mathbf{cl}}( s ),v^{\prime })\mathbf{k}^{\chi}_{\overline{w}}(V_{\mathbf{cl}}^{\prime }( s _1),v^{\prime \prime })h\left( X_{\mathbf{cl}
}^{\prime }( s _1),v^{\prime \prime }\right),
\end{split}
\end{equation*}
where $\mathbf{k}^{\chi}_{\overline{w}}(\cdot)=w_{q,\ta,\vth}\mathbf{k}^{\chi}(\frac{\cdot}{w_{q,\ta,\vth}})$ and
the back-time specular cycle from $( s ,X_{\mathbf{cl}
}^{{}}( s ),v^{\prime })$ is denoted by
\begin{equation}
X_{\mathbf{cl}}^{\prime }( s _1)=X_{\mathbf{cl}}( s _1; s ,X_{\mathbf{cl}%
}( s ),v^{\prime }),\text{ \ \ \ \ }V_{\mathbf{cl}}^{\prime }( s _1)=V_{%
\mathbf{cl}}^{{}}( s _1; s ,X_{\mathbf{cl}}^{{}}( s ),v^{\prime }).
\label{x'}
\end{equation}%
More explicitly, let $t_{k}$ and $t_{k^{\prime }}^{\prime }$ be the
corresponding times for both specular cycles as in (\ref{specularcycle}).
For $t_{k+1}\leq  s <t_{k},$ $t_{k^{\prime }+1}^{\prime }\leq
 s _1<t_{k^{\prime }}^{\prime }$
\begin{equation}
X_{\mathbf{cl}}^{\prime }( s _1)=X_{\mathbf{cl}}( s _1; s ,X_{\mathbf{cl}%
}( s ),v^{\prime })\equiv x_{k^{\prime }}^{\prime }+( s _1-t_{k^{\prime
}}^{\prime })v_{k^{\prime }}^{\prime },  \label{x'explicit}
\end{equation}%
where $x_{k^{\prime }}^{\prime }=X_{\mathbf{cl}}(t'_{k^{\prime
}}; s ,x_{k}+( s -t_{k})v_{k},v^{\prime }),v_{k^{\prime }}^{\prime }=V_{%
\mathbf{cl}}(t'_{k^{\prime }}; s ,x_{k}+( s -t_{k})v_{k},v^{\prime }).\,$%
\ Recall $\alpha $ in \eqref{alpha} and define naturally
\begin{equation*}
\alpha (x,v)\equiv \alpha (t)=\xi ^{2}(x)+[v\cdot \nabla \xi
(x)]^{2}-2[v\cdot\nabla ^{2}\xi (x)\cdot v]\xi (x).
\end{equation*}%
We define the main set
\begin{equation}
A_\alpha=\{(x,v):x\in \bar{\Omega},\text{ }\frac{1}{N}\leq |v|\leq N,%
\text{ and }\alpha (x,v)\geq \frac{1}{N}\}.  \label{aalpha}
\end{equation}

\begin{lemma}\label{specularlower}\cite[Lemm 22, pp.775]{Guo-2010}
Fix $k$ and $k^{\prime }.$ Define for $t_{k+1}\leq
 s \leq t_{k}, s _1\in \R$ and
\begin{equation*}
J\equiv J_{k,k^{\prime }}(t,x,v, s , s _{1},v^{\prime })\equiv \det \left( \frac{%
\partial \{x_{k^{\prime }}^{\prime }+( s _1-t_{k^{\prime }}^{\prime
})v_{k^{\prime }}^{\prime }\}}{\partial v^{\prime }}\right) .
\end{equation*}%
For any $\varepsilon >0$ sufficiently small, there is $\widetilde{\delta} (N,\varepsilon,T_0
,k,k^{\prime })>0$ and an open covering $\cup
_{i=1}^{m}B(t_{i},x_{i},v_{i};r_{i})$ of $[0,T_0]\times A_\alpha$ and
corresponding open sets $O_{t_{i},x_{i},v_{i}}$ for $[t_{k+1}+\varepsilon
,t_{k}-\varepsilon ]\times \R\times \R^{3}$ with $%
|O_{t_{i},x_{i},v_{i}}|<\varepsilon ,$ such that
\begin{equation*}
|J_{k,k^{\prime }}(t,x,v, s , s _{1},v^{\prime })|\geq \widetilde{\delta} >0,
\end{equation*}%
for $0\leq t\leq T_0,$ $(x,v)\in A_\alpha$ and $( s , s _{1},v^{\prime })$
in
\begin{equation*}
O_{t_{i,}x_{i},v_{i}}^{c}\cap \lbrack t_{k+1}+\varepsilon ,t_{k}-\varepsilon
]\times \lbrack 0,T_0]\times \{|v^{\prime }|\leq 2 N\}.
\end{equation*}
\end{lemma}


In order to prove Proposition \ref{specularbd}, we first show the following crucial estimates with the aid of Lemmas \ref{specularcon}
and \ref{specularlower}.
\begin{lemma}\label{finitelem}
There exist constants $T_0>0$ and $C_{T_0}>0$ such that
\begin{equation}\label{finitees}
\|U(T_0)h_0\|_\infty\leq e^{-\la_0T_0^{\rho_1}}\|h_0\|_\infty+C_{T_0}\int_0^{T_0}\|f(s)\|_2ds.
\end{equation}
\end{lemma}

\begin{proof}Our proof is divided into two steps.

\noindent{\it Step 1. Estimate of $h{\bf 1}_{A_\al}$.}
Let us split $h=h{\bf 1}_{A_\al}+h(1-{\bf 1}_{A_\al}),$
and we first express and estimate the main part $h{\bf 1}_{A_\al}$ through \eqref{Unex}.
By utilizing \eqref{IP.lbd2} and Lemmas \ref{es.k} and \ref{slG}, we see that
\begin{equation*}
|H_1(t,x,v)|
\leq Ce^{-\la_2t^{\rho_1}}\left\|h_0\right\|_{\infty},
\end{equation*}
\begin{equation*}
\begin{split}
|H_2(t,x,v)|
\leq& C\int_0^te^{-\frac{1}{2}\int_s^t\widetilde{\nu}(v)d\tau}\widetilde{\nu}(v)
e^{-\frac{\la_2}{2}(t^{\rho_1}-s^{\rho_1})}
e^{-\frac{\la_2}{2}s^{\rho_1}}
e^{\frac{\la_2}{2}s^{\rho_1}}\\&\times[K_{\overline{w}}^{1-\chi}U(s)h(s)]\widetilde{\nu}^{-1}(v)ds\\
\leq& C\eps^{3+\varrho}e^{-\frac{\la_2}{2}t^{\rho_1}}\sup\limits_{0\leq  s \leq t}
\left\|e^{\frac{\la_2}{2}s^{\rho_1}}U(s)h(s)\right\|_{\infty}\int_0^t
e^{-\frac{1}{2}\int_0^s\widetilde{\nu}(v')d\tau}\widetilde{\nu}(v)ds\\
\leq& C\eps^{3+\varrho}e^{-\frac{\la_2}{2}t^{\rho_1}}\sup\limits_{0\leq  s \leq t}
\left\|e^{\frac{\la_2}{2}s^{\rho_1}}U(s)h(s)\right\|_{\infty}.
\end{split}
\end{equation*}
here we have used the fact that $\int_0^te^{-\frac{1}{2}\int_s^t\widetilde{\nu}(v)d\tau}\widetilde{\nu}(v)ds<\infty$ and the significant observation $\widetilde{\nu}^{-1}\leq \nu^{-1}$.
Continuing, one has
\begin{equation*}
\begin{split}
|H_3(t,x,v)|
\leq& C\left\|h_0\right\|_{\infty}
\int_{\R^3}\int_{0}^te^{-\frac{1}{2}\int_s^t\widetilde{\nu}(v)d\tau}\widetilde{\nu}(v)
 e^{-\frac{\la_2}{2}(t^{\rho_1}-s^{\rho_1})}
e^{-\frac{\la_2}{2}s^{\rho_1}}\nu^{-1}(v) \\&\times{\bf k}_{\overline{w}}^\chi(V_{\mathbf{cl}}( s ),v^{\prime })dsdv'\\
\leq& Ce^{-\frac{\la_2}{2}t^{\rho_1}}\left\|h_0\right\|_{\infty},
\end{split}
\end{equation*}
and
\begin{equation*}
\begin{split}
|H_4(t,x,v)|
\leq& C\eps^{3+\varrho}\sup\limits_{0\leq  s \leq t}
\left\|e^{\frac{\la_2}{2}s^{\rho_1}}U(s)h(s)\right\|_{\infty}
\\&\times\int_{\R^3}\int_{0}^t\int_0^se^{-\frac{1}{2}\int_s^t\widetilde{\nu}(v)d\tau} e^{-\frac{1}{2}
\int_{s_1}^s\widetilde{\nu}(v')d\tau}\widetilde{\nu}(v)\widetilde{\nu}(v')\\&\times e^{-\frac{\la_2}{2}(t^{\rho_1}-s^{\rho_1})}
 e^{-\frac{\la_2}{2}(s^{\rho_1}-s_1^{\rho_1})}
e^{-\frac{\la_2}{2}s_1^{\rho_1}} \nu^{-1}(v)
{\bf k}_{\overline{w}}^\chi(V_{\mathbf{cl}}( s ),v^{\prime })dsds_{1}dv'\\
\leq& C\eps^{3+\varrho}e^{-\frac{\la_2}{2}t^{\rho_1}}\sup\limits_{0\leq  s \leq t}
\left\|e^{\frac{\la_2}{2}s^{\rho_1}}U(s)h(s)\right\|_{\infty}.
\end{split}
\end{equation*}

For the main contribution $H_5$,
notice that along the
back-time specular cycles $[X_{\mathbf{cl}}^{{}}( s ),V_{\mathbf{cl}}^{{}}( s )]$
and $[X_{\mathbf{cl}}^{\prime }( s _1),V_{\mathbf{cl}}^{\prime }( s _1)]$ in (\ref%
{x'}), $|V_{\mathbf{cl}}^{{}}( s )|\equiv |v|$ and $|V_{\mathbf{cl}%
}^{\prime }( s _1)|\equiv |v^{\prime }|.$ Therefore, the integration over $|v|>N$ or $%
|v^{\prime }|\geq 2N$ or $|v^{\prime }|\leq 2N$ but $|v^{\prime \prime
}|\geq 3N$ are bounded by
$$
C\left\{e^{-\frac{\vps N^2}{16}}+\frac{1}{N}\right\}e^{-\frac{\la_2}{2}t^{\rho_1}}\sup\limits_{0\leq  s \leq t}
\left\|e^{\frac{\la_2}{2}s^{\rho_1}}h(s)\right\|_{\infty}.
$$
As in Case 3 in subsection \ref{lifth}, by using the same
approximation, we only need to concentrate on the bounded set
$\{|v|\leq N,\ |v^{\prime}|\leq 2N$ and $|v^{\prime \prime }|\leq 3N\}$ of%
\begin{equation*}
\begin{split}
\int_{0}^{t}\int_{0}^{ s }&\int_{|v^{\prime }|\leq 2N,|v^{\prime \prime
}|\leq 3N}e^{-\int^t_s\widetilde{\nu}(v)\tau-\int_{s_1}^s\widetilde{\nu}(v')d\tau}
\left|h\left(  s _1,X'_{\mathbf{cl}}( s _1),v^{\prime \prime }\right)\right|dv^{\prime }dv^{\prime \prime
}d s _{1}d s \\
=&\int_{\substack{ \alpha (X_{\mathbf{cl}}( s ),v')<\varepsilon
\\ |v^{\prime }|\leq 2N,|v^{\prime \prime }|\leq 3N}}+\int_{\substack{ %
\alpha (X_{\mathbf{cl}}( s ),v')\geq \varepsilon  \\ |v^{\prime
}|\leq 2N,|v^{\prime \prime }|\leq 3N}}=H_{5,1}+H_{5,2}.
\end{split}
\end{equation*}%
In the case $\alpha (X_{\mathbf{cl}}( s ),v')\leq \varepsilon $,
$\xi ^{2}(X_{\mathbf{cl}}( s ))+[v'\cdot\nabla \xi (X_{\mathbf{
cl}}( s ))]^{2}\leq \varepsilon,$
notice that $|\nabla \xi (X_{\mathbf{%
cl}}( s ))|\geq c>0$, hence for $\varepsilon $ small and $X_{
\mathbf{cl}}( s )\backsim \partial \Omega ,$ $H_{5,1}$ is dominated by
\begin{equation*}
\begin{split}
H_{5,1}\leq& C_{N}\int_{0}^{t}\int_{0}^{ s }
e^{-\frac{1}{2}\int_s^t\widetilde{\nu}(v)d\tau} e^{-\frac{1}{2}
\int_{s_1}^s\widetilde{\nu}(v')d\tau}\widetilde{\nu}(v)\widetilde{\nu}(v')\\&\times e^{-\frac{\la_2}{2}(t^{\rho_1}-s^{\rho_1})}
 e^{-\frac{\la_2}{2}(s^{\rho_1}-s_1^{\rho_1})}
e^{-\frac{\la_2}{2}s_1^{\rho_1}}e^{\frac{\la_2}{2}s_1^{\rho_1}}\left\|\widetilde{\nu}^{-1}h( s _1)\right\|_{\infty
}dsds_{1}\\&\times\int_{\substack{ \alpha (X_{\mathbf{cl}}( s ),v^{\prime })\leq
\varepsilon  \\ |v^{\prime }|\leq 2N,|v^{\prime \prime }|\leq 3N}} \\
\leq &C_{N}e^{-\frac{\la_2}{2}t^{\rho_1}}\sup\limits_{0\leq  s \leq t}
\left\|e^{\frac{\la_2}{2}s^{\rho_1}}h(s)\right\|_{\infty}
\int_{|v^{\prime }\cdot \frac{\nabla \xi (X_{\mathbf{cl}}( s ))}{|\nabla
\xi (X_{\mathbf{cl}}( s ))|}|\leq c\varepsilon ,|v^{\prime }|\leq
2N,|v^{\prime \prime }|\leq 3N}\\
\leq& C_{N}\varepsilon
e^{-\frac{\la_2}{2}t^{\rho_1}}\sup\limits_{0\leq  s \leq t}
\left\|e^{\frac{\la_2}{2}s^{\rho_1}}h(s)\right\|_{\infty}.
\end{split}
\end{equation*}
As to the case $\alpha (X_{%
\mathbf{cl}}(s),v^{\prime })\geq \varepsilon $, from (\ref{x'explicit}), we bound $H_{5,2}$ as
\begin{equation*}
\begin{split}
C_{N}&\int_{0}^{t}e^{-\frac{\la_2}{2}(t^{\rho_1}-s_1^{\rho_1})}\int_{0}^{s}\int_{\substack{ \alpha
(X_{\mathbf{cl}}(s),v^{\prime })\geq \varepsilon  \\ |v^{\prime }|\leq
2N,|v^{\prime \prime }|\leq 3N}}|h\left( s_1,X_{%
\mathbf{cl}}^{\prime }(s_1),v^{\prime \prime }\right) |dv^{\prime }dv^{\prime
\prime } \\
=&C_{N}\sum_{k,k^{\prime }}\int_{t_{k+1}}^{t_{k}}\int_{t_{k^{\prime
}+1}^{\prime }}^{t_{k^{\prime }}^{\prime }}\int_{\substack{ \alpha (X_{%
\mathbf{cl}}(s),v^{\prime })\geq \varepsilon  \\ |v^{\prime }|\leq
2N,|v^{\prime \prime }|\leq 3N}}e^{-\frac{\la_2}{2}(t^{\rho_1}-s_1^{\rho_1})}\\&\times
|h\left( s_1,x_{k^{\prime }}^{\prime }+(s_1-t_{k^{\prime }}^{\prime
})v_{k^{\prime }}^{\prime },v^{\prime \prime }\right) |dv^{\prime }dv^{\prime
\prime },
\end{split}
\end{equation*}%
where $[t_{k^{\prime }}^{\prime },x_{k^{\prime }}^{\prime },v_{k^{\prime
}}^{\prime }]$ is the back-time cycle of $%
(s,x_{k}+(s-t_{k})v_{k},v_{k}),$ for $t_{k+1}\leq s\leq t_{k}.$

We now study $x_{k^{\prime }}^{\prime }+(s_1-t_{k^{\prime }}^{\prime
})v_{k^{\prime }}^{\prime }.$ By the repeatedly using Velocity Lemma \ref%
{velocity}, we deduce for $(t,x,v)\in A_{\alpha }$ and $0<t\leq T_0$ and $\alpha
(X_{\mathbf{cl}}(s),v^{\prime })\geq \varepsilon$
\begin{eqnarray*}
\alpha (t_{l}) &\sim&\{v_{l}\cdot n_{x_{l}}\}^{2}\geq e^{-\{C_{\xi
}N-1\}T_{0}}\alpha (t)\geq C_{T_{0},\xi ,N}>0; \\
\alpha (t_{l'}^{\prime }) &\sim&\{v_{l'}^{\prime }\cdot n_{x_{l'}^{\prime
}}\}^{2}\geq e^{-\{C_{\xi }N-1\}T_{0}}\alpha (X_{\mathbf{cl}%
}(s),v^{\prime })\geq C_{T_{0},\xi,N }\varepsilon >0.
\end{eqnarray*}%
Therefore, applying (\ref{tlower}) in Lemma \ref{huang} yields $%
t_{l}-t_{l+1}\geq \frac{c_{T_{0},\xi ,N}}{N^{2}}$ and $t_{l'}^{\prime
}-t_{l'+1}^{\prime }\geq \frac{c_{T_{0},\xi ,N}\varepsilon }{4N^{2}}$ so that
\begin{equation*}
k\leq \frac{T_{0}N^{2}}{c_{T_{0},\xi ,N}}=C_{T_{0},\xi ,N},\text{ \ \ }%
k^{\prime }\leq \frac{T_{0}N^{2}}{c_{T_{0},\xi ,N}\varepsilon }=C_{T_{0},\xi
,N,\varepsilon }.
\end{equation*}
With this, one can further split the $s-$integral as
\begin{equation*}
\begin{split}
C_{N}&\int_{t_{k+1}}^{t_{k}}\int_{\substack{  \\ |v^{\prime }|\leq
2N,|v^{\prime \prime }|\leq 3N}}\sum_{k\leq C_{T_{0},N},k^{\prime }\leq
C_{T_{0},N,}\varepsilon }\int_{t_{k^{\prime }+1}^{\prime }}^{t_{k^{\prime
}}^{\prime }}\mathbf{1}_{A_{\alpha }}e^{-\frac{\la_2}{2}(t^{\rho_1}-s_1^{\rho_1})}\\&\times|h\left(
s_1,x_{k^{\prime }}^{\prime }+(s_1-t_{k^{\prime }}^{\prime })v_{k^{\prime
}}^{\prime },v^{\prime \prime }\right) | \\
=&\int_{t_{k+1}+\varepsilon }^{t_{k}-\varepsilon
}+\int^{t_{k}}_{t_{k}-\varepsilon }+\int_{t_{k+1}}^{t_{k+1}+\varepsilon }.
\end{split}
\end{equation*}

Notice that $\sum_{k^{\prime }}\int_{t_{k^{\prime }+1}^{\prime }}^{t_{k^{\prime
}}^{\prime }}=\int_{0}^{s},$ the last two terms make small contribution
as
\begin{equation*}
\begin{split}
\varepsilon C_{N}&\sup_{0\leq s\leq t}e^{-\frac{\la_2}{2}(t^{\rho_1}-s^{\rho_1})}||h(s)||_{\infty
}\int_{0}^{T_{0}}\int_{\substack{  \\ |v^{\prime }|\leq 2N,|v^{\prime \prime
}|\leq 3N}}\\&=\varepsilon C_{N,T_0}\sup_{0\leq s\leq t}e^{-\frac{\la_2}{2}(t^{\rho_1}-s^{\rho_1})}\left\|h(s)\right\|_{\infty }.
\end{split}
\end{equation*}%
For the main contribution $\int_{t_{k+1}+\varepsilon }^{t_{k}-\varepsilon },$
By Lemma \ref{specularlower}, on the set $O_{t_{i},x_{i},v_{i}}^{c}%
\cap \lbrack t_{k+1}+\varepsilon ,t_{k}-\varepsilon ]\times \lbrack
0,T_{0}]\times \{|v^{\prime }|\leq N\},$ we can define a change of variable
\begin{equation*}
y\equiv x_{k^{\prime }}^{\prime }+(s_1-t_{k^{\prime }}^{\prime })v_{k^{\prime
}}^{\prime },
\end{equation*}%
so that $\det (\frac{\partial y}{\partial v^{\prime }})>\delta $ on the same
set. By the Implicit Function Theorem, there are an finite open covering $%
\cup _{j=1}^{m}V_{j}$ of $O_{t_{i},x_{i},v_{i}}^{c}\cap \lbrack
t_{k+1}+\varepsilon ,t_{k}-\varepsilon ]\times \lbrack 0,T_{0}]\times
\{|v^{\prime }|\leq N\}$, and smooth function $F_{j}$ such that $v^{\prime
}=F_{j}(t,x,v,y,s_{1},s)$ in $V_{j}$. We therefore have
\begin{equation*}
\begin{split}
\sum_{k,k^{\prime }}\int_{t_{k+1}+\varepsilon }^{t_{k}-\varepsilon }\int
_{\substack{  \\ |v^{\prime }|\leq 2N,|v''|\leq3N}}\int_{t_{k^{\prime }+1}^{\prime
}}^{t_{k^{\prime }}^{\prime }}\leq& \sum_{k,k^{\prime
}}\int_{t_{k+1}+\varepsilon }^{t_{k}-\varepsilon }\int_{\substack{  \\ %
|v^{\prime }|\leq 2N,|v''|\leq3N}}\int_{t_{k^{\prime }+1}^{\prime }}^{t_{k^{\prime
}}^{\prime }}\mathbf{1}_{O_{t_{i},x_{i},v_{i}}}\\&+\sum_{j,k,k^{\prime
}}\int_{t_{k+1}+\varepsilon }^{t_{k}-\varepsilon }\int_{\substack{  \\ %
|v^{\prime }|\leq 2N,|v''|\leq3N}}\int_{t_{k^{\prime }+1}^{\prime }}^{t_{k^{\prime
}}^{\prime }}\mathbf{1}_{V_{j}}.
\end{split}
\end{equation*}%
Since $\sum_{k^{\prime }}\int_{t_{k^{\prime }+1}^{\prime }}^{t_{k^{\prime
}}^{\prime }}=\int_{0}^{s}\leq \int_{0}^{T_{0}}$ and $%
|O_{t_{i},x_{i},v_{i}}|<\varepsilon ,$ the first part is bounded by
$$%
C_{N,T_0}\varepsilon e^{-\frac{\la_2}{2}t^{\rho_1}}\sup\limits_{0\leq s\leq t}\left\{e^{\frac{\la_2}{2}%
s^{\rho_1}}\left\|h(s)\right\|_{\infty }\right\}.$$

For the second part, we can make a change of variable $v^{\prime
}\rightarrow y=x_{k^{\prime }}^{\prime }+(s_1-t_{k^{\prime }}^{\prime
})v_{k^{\prime }}^{\prime }$ on each $V_{j}$ to get
\begin{equation*}
\begin{split}
C_{\varepsilon ,T_{0},N}&\sum_{j,k,k^{\prime }}\int_{V_{j}}\int_{|v^{\prime
\prime }|\leq 3N}e^{-\frac{\la_2}{2}(t^{\rho_1}-s_1^{\rho_1})}|h\left( s_1,x_{k^{\prime }}^{\prime
}+(s_1-t_{k^{\prime }}^{\prime })v_{k^{\prime }}^{\prime },v^{\prime \prime
}\right) | \\
=&C_{\varepsilon ,T_{0},N}\sum_{j}\int_{V_{j}}\int_{|v^{\prime \prime
}|\leq 3N}e^{-\frac{\la_2}{2}(t^{\rho_1}-s_1^{\rho_1})}|h\left( s_1,y,v^{\prime \prime }\right) |\frac{1}{%
\left|\det \{\frac{\partial y}{\partial v^{\prime }}\}\right|}dydv^{\prime \prime
}dsds_{1} \\
\leq &\frac{C_{\varepsilon ,T_{0},N}}{\delta }\int_{0}^{t}%
\int_{0}^{s}e^{-\frac{\la_2}{2}t^{\rho_1}}\int_{|v^{\prime \prime }|\leq 3N}e^{\frac{\la_2}{2}s_1^{\rho_1}}\left\{ \int_{\Omega }h^{2}\left( s_1,y,v^{\prime \prime }\right)
dy\right\} ^{1/2}dv^{\prime \prime }dsds_{1} \\
\leq &C_{\varepsilon ,T_{0},N}\int_{0}^{t}\|f(s)\|_2ds,
\end{split}
\end{equation*}%
where $f=\frac{h}{w_{q,\ta,\vth}}.$ We therefore conclude, summing over $k$ and $%
k^{\prime }$ and collecting terms
\begin{eqnarray}
\left\|h(t,x,v)\mathbf{1}_{A_{\alpha }}\right\|_{\infty } &\leq &Ce^{-\frac{\la_2}{2}t^{\rho_1}}\left\|h_{0}\right\|_{\infty }+C_{\varepsilon,T_{0},N}\int_{0}^{t}\|f(s)\|_2ds  \notag \\
&&+\left\{\frac{C}{N}+C_{N,T_{0}}\varepsilon \right\}e^{-\frac{\la_2}{2}t^{\rho_1}}\sup_{0\leq s\leq t}e^{\frac{\la_2}{2}s^{\rho_1}}\left\|h(s)\right\|_{\infty }.  \label{hm}
\end{eqnarray}
{\it Step 2:} Estimate of $h.$ We first get from $h(t,x,v)=G(t)h_{0}+\int_{0}^{t}G(t,s)K_{\overline{w}}h(s)ds$ that
\begin{equation}\label{hduhamel}
\begin{split}
\left\|h(t)\right\|_{\infty }\leq& e^{-\la_2t^{\rho_{2}}}\left\|h_{0}\right\|_{\infty }+\int_{0}^{t}e^{-\frac{1}{2}\int_s^t\widetilde{\nu}(v)d\tau}\widetilde{\nu}(v)e^{-\frac{\la_2}{2}(t^{\rho_1}-s^{\rho_1})}
\\&\times\left\|\nu^{-1}(v)K^{1-\chi}_{\overline{w}}h\right\|_{\infty }(s)ds
\\&+\int_{0}^{t}e^{-\frac{1}{2}\int_s^t\widetilde{\nu}(v)d\tau}\widetilde{\nu}(v)e^{-\frac{\la_2}{2}(t^{\rho_1}-s^{\rho_1})}
\left\|\nu^{-1}(v)K^{\chi}_{\overline{w}}h\right\|_{\infty }(s)ds
\\ \leq& e^{-\la_2t^{\rho_{2}}}\left\|h_{0}\right\|_{\infty }+C\eps^{3+\varrho}e^{-\frac{\la_2}{2}t^{\rho_{2}}}\sup_{0\leq s\leq t}e^{\frac{\la_2}{2}s^{\rho_1}}\left\|h(s)\right\|_{\infty }
\\&+\int_{0}^{t}e^{-\frac{1}{2}\int_s^t\widetilde{\nu}(v)d\tau}\widetilde{\nu}(v)e^{-\frac{\la_2}{2}(t^{\rho_1}-s^{\rho_1})}
\left\|\nu^{-1}(v)K^{\chi}_{\overline{w}}h\right\|_{\infty }(s)ds.
\end{split}
\end{equation}%
Next, since $\{K^{\chi}_{\overline{w}}h\}(s,x,v)=\int \mathbf{k}^{\chi}_{\overline{w}}(v,v^{\prime
})h(s,x,v^{\prime })dv^{\prime }$, we then rewrite
\begin{equation*}
\begin{split}
\int_{0}^{t}&e^{-\frac{1}{2}\int_s^t\widetilde{\nu}(v)d\tau}\widetilde{\nu}(v)e^{-\frac{\la_2}{2}(t^{\rho_1}-s^{\rho_1})}
\left\|\nu^{-1}(v)K^{\chi}_{\overline{w}}h\right\|_{\infty }(s)ds
\\=&
\int_{0}^{t}e^{-\frac{1}{2}\int_s^t\widetilde{\nu}(v)d\tau}\widetilde{\nu}(v)e^{-\frac{\la_2}{2}(t^{\rho_1}-s^{\rho_1})}
\\&\quad\times\left\|\nu^{-1}\int \mathbf{k}^{\chi}_{\overline{w}}(v,v^{\prime })h(s,x,v^{\prime })\{1-\mathbf{1}%
_{A_{\alpha }(x,v^{\prime })}\}dv^{\prime }\right\|_\infty ds
\\&+\int_{0}^{t}e^{-\frac{1}{2}\int_s^t\widetilde{\nu}(v)d\tau}\widetilde{\nu}(v)e^{-\frac{\la_2}{2}(t^{\rho_1}-s^{\rho_1})}
\left\|\nu^{-1}\int \mathbf{k}^{\chi}_{\overline{w}}(v,v^{\prime
})h(s,x,v^{\prime })\mathbf{1}_{A_{\alpha }(x,v^{\prime })}dv^{\prime }\right\|_\infty ds\\
\eqdef& H_6+H_7.
\end{split}
\end{equation*}%
From the definition of $A_{\alpha }$ in (\ref{aalpha}), it follows that
\begin{equation*}
\begin{split}
H_6\leq& C\left( \int_{|v^{\prime }|\geq N,\text{ or }|v^{\prime }|\leq \frac{1}{N}}|\nu^{-1}
\mathbf{k}^{\chi}_{\overline{w}}(v,v^{\prime })|dv^{\prime }+\int_{\alpha (x,v^{\prime })\leq
\frac{1}{N}}|\nu^{-1}\mathbf{k}^{\chi}_{\overline{w}}(v,v^{\prime })|\right)\\&\times e^{-\frac{\la_2}{2}t^{\rho_1}}\sup_{0\leq s\leq t}e^{\frac{\la_2}{2}s^{\rho_1}}\left\|h(s)\right\|_{\infty }.
\end{split}
\end{equation*}%
By approximation if necessary, one sees $\int_{|v^{\prime }|\geq
N,\text{ or }|v^{\prime }|\leq \frac{1}{N}}|\nu^{-1}\mathbf{k}^{\chi}_{\overline{w}}(v,v^{\prime
})|dv^{\prime }=o(1)$ as $N\rightarrow \infty .$ From $\alpha (x,v^{\prime
})\leq \frac{1}{N}$, $\xi ^{2}(x)+[v^{\prime }\cdot $ $\nabla \xi
(x)]^{2}\leq \frac{1}{N}.$ For $N$ large, $x\backsim \partial \Omega $ and $%
|\nabla \xi (x)|\geq c$ so that
\begin{equation*}
\int_{\alpha (x,v^{\prime })\leq \frac{1}{N}}
|\nu^{-1}\mathbf{k}^{\chi}_{\overline{w}}(v,v^{\prime})|dv^{\prime }\leq \int_{|v^{\prime }\cdot \frac{\nabla \xi (x)}{|\nabla
\xi (x)|}|\leq \frac{1}{c\sqrt{N}}}|\nu^{-1}\mathbf{k}^{\chi}_{\overline{w}}(v,v^{\prime })|dv^{\prime
}=o(1),
\end{equation*}%
as $N\rightarrow \infty .$ As a consequence, it follows that
\begin{equation*}
H_6\leq o(1) e^{-\frac{\la_2}{2}t^{\rho_1}}\sup_{0\leq s\leq t}e^{\frac{\la_2}{2}s^{\rho_1}}\left\|h(s)\right\|_{\infty }.
\end{equation*}
As to $H_7$,
in view of \eqref{hm}, one has
\begin{equation*}
\begin{split}
H_7 \leq& Ce^{-\frac{\la_2}{2}t^{\rho_1}}\left\|h_{0}\right\|_{\infty }+\left\{\frac{C}{N}+C_{N,T_{0}}\varepsilon \right\}e^{-\frac{\la_2}{2}t^{\rho_1}}\sup_{0\leq s\leq t}e^{\frac{\la_2}{2}s^{\rho_1}}\left\|h(s)\right\|_{\infty }  \\&+C_{\varepsilon,T_{0},N}\int_{0}^{t}\|f(s)\|_2ds.
\end{split}
\end{equation*}
Hence, substituting the estimates for $H_6$ and $H_7$ into \eqref{hduhamel}, we arrive at
\begin{equation*}
\begin{split}
\left\|h(t)\right\|_{\infty }\leq& Ce^{-\frac{\la_2}{2}t^{\rho_1}}\left\|h_{0}\right\|_{\infty }+\left\{\frac{C}{N}+C_{N,T_{0}}\varepsilon+o(1) \right\}e^{-\frac{\la_2}{2}t^{\rho_1}}\sup_{0\leq s\leq t}e^{\frac{\la_2}{2}s^{\rho_1}}\left\|h(s)\right\|_{\infty }  \\&+C_{\varepsilon,T_{0},N}\int_{0}^{t}\|f(s)\|_2ds.
\end{split}
\end{equation*}%
We choose $T_{0}$ large such that $2Ce^{-\frac{\la_2}{2}T_0^{\rho_1}}=e^{-\la_0T_0^{\rho_1}},$ for some $\la_0 >0.$ We then further choose $%
N $ large, and then $\varepsilon $ sufficiently small such that $C\{o(1)+%
\frac{1}{N}+C_{N,T_{0}}\varepsilon \}<\frac{1}{2}.$ Therefore, one has
\begin{equation*}
\sup_{0\leq s\leq t}\left\{e^{\frac{\la_2}{2}s^{\rho_1}}\left\|h(s)\right\|_{\infty }\right\}\leq
2C\left\|h_{0}\right\|_{\infty }+C_{T_{0}}\int_{0}^{t}\|f(s)\|_2ds.
\end{equation*}%
Choosing $s=t=T_{0},$ we deduce the finite-time estimate (\ref{finitees}),
and the proof of Lemma \ref{finitelem} is completed.
\end{proof}
We are ready to present
\begin{proof}[The proof of Proposition \ref{specularbd}]
It suffices to only prove (\ref{Ubd}) for $t\geq 1.$ For any $%
m\geq 1,$ we employ the finite-time estimate (\ref{finitees}) repeatedly to
functions $h(lT_{0}+s)$ for $l=m-1,m-2,...0$ to deduce
\begin{eqnarray*}
\left\|h(mT_{0})\right\|_{\infty } &\leq &e^{-\la_0 T^{\rho_1}_{0}}\left\|h(\{m-1\}T_{0})\right\|_{\infty}
+C_{T_{0}}\int_{0}^{T_{0}}\|f(\{m-1\}T_{0}+s)\|_2ds \\
&=&e^{-\la_0 T^{\rho_1}_{0}}\left\|h(\{m-1\}T_{0})\right\|_{\infty}
+C_{T_{0}}\int_{\{m-1\}T_{0}}^{mT_{0}}\|f(s)\|_2ds \\
&\leq &e^{-2\la_0 T^{\rho_1}_{0}}\left\|h(\{m-2\}T_{0})\right\|_{\infty }
+e^{-\la_0 T^{\rho_1}_{0}}C_{T_{0}}\int_{\{m-2\}T_{0}}^{\{m-1\}T_{0}}\|f(s)\|_2ds \\
&&+C_{T_{0}}\int_{\{m-1\}T_{0}}^{mT_{0}}\left\|f(s)\right\|_2ds \\
&\leq &e^{-m\la_0 T^{\rho_1}_{0}}\left\|h(0)\right\|_{\infty}
+C_{T_{0}}\sum_{k=0}^{m-1}e^{-k\la_0 T^{\rho_1}_{0}}\int_{\{m-k-1\}T_{0}}^{\{m-k\}T_{0}}\|f(s)\|_2ds,
\end{eqnarray*}%
where $h(t)=U(t)h_{0}.$

Next, by the $L^{2}$ decay constructed in Lemma \ref{sl2.eng}, in the interval
$\{m-k-1\}T_{0}\leq s\leq \{m-k\}T_{0},$ one has
$$\|f(s)\|_2\leq e^{-\lambda s^{\rho_0}}\|w_{q/4,\ta}f_{0}\|_2\leq e^{-\lambda (\{m-k-1\}T_{0})^{\rho_0}}\|w_{q/4,\ta}f_{0}\|_2.$$
Noticing that $\rho_0=\frac{\ta}{\ta-\varrho}>\frac{\ta+\vth\varrho}{\ta-\varrho}=\rho_1$, taking $\la_0=\min\{\la,\la_0\}$ and applying
$(k+1)T^{\rho_1}_{0}\geq ((k+1)T_{0})^{\rho_1}$ for $0<\rho_1<1,$
we further obtain
\begin{equation*}
\begin{split}
\left\|h(mT_{0})\right\|_{\infty }\leq&
e^{-m\la_0 T^{\rho_1}_{0}}\left\|h(0)\right\|_{\infty}
+C_{T_{0}}\sum_{k=0}^{m-1}e^{-k\la_0 T^{\rho_1}_{0}}\int_{\{m-k-1\}T_{0}}^{\{m-k\}T_{0}}\\&\times e^{-\lambda (\{m-k-1\}T_{0})^{\rho_0}}\|w_{q/4,\ta}f_{0}\|_2ds \\
\leq &e^{-m\la_0 T^{\rho_1}_{0}}\left\|h(0)\right\|_{\infty}+C_{T_{0}}e^{\la_0 T^{\rho_1}_{0}}mT_{0}e^{-\la_0m^{\rho_1} T^{\rho_1}_{0}}\|w_{q/4,\ta}f_{0}\|_2 \\
\leq &C_{T_{0},\la_0 }e^{-\frac{\la_0 m^{\rho_1} T^{\rho_1}_{0}}{2}}\left\|h(0)\right\|_{\infty},
\end{split}
\end{equation*}%
where we also used the fact that
$$\|w_{q/4,\ta}f_{0}\|_2=\left\|w_{q/4,\ta}w^{-1}_{q,\ta,\vth}h_{0}\right\|_2\leq
C\left\|h_{0}\right\|_{\infty },$$
and
$$(\{m-k-1\}T_{0})^{\rho_1}+(\{k+1\}T_{0})^{\rho_1}\geq (mT_{0})^{\rho_1},\ \
mT_{0}e^{-\la_0 m^{\rho_1}T^{\rho_1}_{0}}\leq e^{-\frac{\la_0
m^{\rho_1} T^{\rho_1}_{0}}{2}}.$$
Finally, for any $t,$ we can find $m$ such that $mT_{0}\leq
t\leq \{m+1\}T_{0},$ and
\begin{equation*}
\begin{split}
\left\|h(t)\right\|_{\infty }\leq& C\left\|h(mT_{0})\right\|_{\infty }
\leq C_{T_{0},\la_0 }e^{-%
\frac{\la_0 m^{\rho_1}T^{\rho_1}_{0}}{2}}\left\|h(0)\right\|_{\infty}\\
\leq&\left\{C_{T_{0},\la_0 }e^{\la_0 T^{\rho_1}_{0}}\right\}e^{-\frac{\la_0 }{2}t^{\rho_1}}\left\|h(0)\right\|_{\infty},
\end{split}
\end{equation*}%
according to the fact $e^{-\frac{\la_0m^{\rho_1} T^{\rho_1}_{0}}{2}}\leq
e^{-\frac{\la_0 }{2}t^{\rho_1}}
e^{\frac{\la_0 T^{\rho_1}_{0}}{2}}$. This ends the proof of Proposition \ref{specularbd}.
\end{proof}

\subsection{Nonlinear existence and time exponential decay}
In this subsection, we make use of Proposition \ref{specularbd} to prove the global existence and time exponential decay of the nonlinear Boltzmann equation with specular reflection boundary condition. Namely, we tend to complete
\begin{proof}[The proof of Theorem \ref{specularnl}]
We start with the following iteration scheme
\begin{eqnarray}\label{nn.sit}
\left\{\begin{array}{rll}
&&\partial _{t}h^{\ell +1}+v\cdot \nabla _{x}h^{\ell +1}+\widetilde{\nu} h^{\ell +1}-K_{\overline{w}}h^{\ell +1}
=w_{q,\ta,\vth}\Ga\left(\frac{h^{\ell}}{w_{q,\ta,\vth}},\frac{h^{\ell}}{w_{q,\ta,\vth}}
\right),\\
&&h^{\ell+1}(0,x,v)=h_0(x,v),
 \end{array}\right.
\end{eqnarray}%
with $h_{{-}}^{\ell +1}(t,x,v)=h^{\ell+1}(t,x,R_xv)$ and $h^{0}=h_0(x,v)$. Here $h^\ell=f^\ell w_{q,\ta,\vth}.$
From the Duhamel principle, it follows
\begin{equation*}
h^{\ell+1}=U(t)h_{0}+\int_{0}^{t}U(t-s)w_{q,\ta,\vth}\Gamma \left(\frac{h^{\ell}}{w_{q,\ta,\vth}},\frac{h^{\ell}}{w_{q,\ta,\vth}}%
\right)(s)ds.
\end{equation*}%
We then get from Proposition \ref{specularbd} and Lemma \ref{es.nop} that
\begin{equation}\label{hlsp}
\begin{split}
\left\|h^{\ell+1}(t)\right\|_{\infty }\leq& Ce^{-\frac{\la_0}{2}t^{\rho_1}}\|h_{0}\|_{\infty}+\left\|\int_{0}^{t}U(t-s)w_{q,\ta,\vth}\Gamma \left(\frac{h^{\ell}}{w_{q,\ta,\vth}},\frac{h^{\ell}}{w_{q,\ta,\vth}}%
\right)(s)ds\right\|_{\infty }\\
\leq& Ce^{-\frac{\la_0t^{\rho_1}}{2}}\|h_{0}\|_{\infty}+\int_{0}^{t}e^{-\frac{\la_0}{2}(t-s)^{\rho_1}-\la_0s^{\rho_1}}ds
\sup\limits_{0\leq s\leq t}\left\|e^{\frac{\la_0}{2}s^{\rho_1}}h^{\ell}(s)\right\|_{\infty }^2
\\
\leq& Ce^{-\frac{\la_0t^{\rho_1}}{2}}\|h_{0}\|_{\infty}+e^{-\frac{\la_0}{2}t^{\rho_1}}\sup\limits_{0\leq s\leq t}\left\|e^{\frac{\la_0}{2}s^{\rho_1}}h^{\ell}(s)\right\|_{\infty }^2,
\end{split}
\end{equation}%
where the fact the $\nu(v)<C$ was used.
This implies that
$$
\sup\limits_{\ell}\sup\limits_{0\leq t\leq \infty }
\left\{e^{\frac{\la_0}{2}t^{\rho_1}}\left\|h^{\ell}(t)\right\|_{\infty }\right\}\leq C\|h_{0}\|_{\infty },$$
for
$\|h_{0}\|_\infty$ sufficiently small. Moreover, subtracting $h^{\ell+1}-h^{\ell}$
yields
\begin{equation*}
\begin{split}
\{\partial _{t}&+v\cdot \nabla _{x}+\widetilde{\nu} -K_{\overline{w}}\}\{h^{\ell+1}-h^{\ell}\}\\&=w_{q,\ta,\vth}\left\{\Gamma \left(%
\frac{h^{\ell}}{w_{q,\ta,\vth}},\frac{h^{\ell}}{w_{q,\ta,\vth}}\right)-\Gamma \left(\frac{h^{\ell-1}}{w_{q,\ta,\vth}},\frac{h^{\ell-1}}{w_{q,\ta,\vth}}
\right)\right\},
\end{split}
\end{equation*}%
with $\{h^{\ell+1}-h^{\ell}\}(0,x,v)=0$ and $\{h^{\ell+1}-h^{\ell}\}(t,x,v)|_-=\{h^{\ell+1}-h^{\ell}\}(t,x,R_xv)$.
By the decomposition
\begin{equation*}
\begin{split}
\Gamma& \left(\frac{h^{\ell}}{w_{q,\ta,\vth}},\frac{h^{\ell}}{w_{q,\ta,\vth}}\right)-\Gamma\left (\frac{h^{\ell-1}}{w_{q,\ta,\vth}},\frac{%
h^{\ell-1}}{w_{q,\ta,\vth}}\right)\\&=\Gamma \left(\frac{h^{\ell}-h^{\ell-1}}{w_{q,\ta,\vth}},\frac{h^{\ell}}{w_{q,\ta,\vth}}\right)-\Gamma \left(\frac{%
h^{\ell-1}}{w_{q,\ta,\vth}},\frac{h^{\ell-1}-h^{\ell}}{w_{q,\ta,\vth}}\right).
\end{split}
\end{equation*}%
Performing the similar calculation as (\ref{hlsp}), we then obtain
\begin{equation*}
\begin{split}
\left\|\{h^{\ell+1}-h^{\ell}\}(t)\right\|_{\infty }\leq&
\left\|\int_{0}^{t}U(t-s)w_{q,\ta,\vth}\Gamma (\frac{h^{\ell}-h^{\ell-1}}{w_{q,\ta,\vth}},\frac{h^{\ell}}{w_{q,\ta,\vth}})(s)ds\right\|_{\infty } \\
&+\left\|\int_{0}^{t}U(t-s)w_{q,\ta,\vth}\Gamma (\frac{h^{\ell-1}}{w_{q,\ta,\vth}},\frac{h^{\ell-1}-h^{\ell}}{w_{q,\ta,\vth}
})(s)ds\right\|_{\infty }  \\
\leq &Ce^{-\frac{\la_0}{2}t^{\rho_1}}\sup\limits_{0\leq s\leq t}\left\{\left\|e^{\frac{\la_0}{2}s^{\rho_1}}h^{\ell}(s)\right\|_{\infty
}+\left\|e^{\frac{\la_0}{2}s^{\rho_1}}h^{\ell-1}(s)\right\|_\infty\right\}\\&\times \sup\limits_{0\leq s\leq t}
\left\|e^{\frac{\la_0}{2}s^{\rho_1}}\{h^{\ell}(s)-h^{\ell-1}(s)\}\right\|_{\infty }.
\end{split}
\end{equation*}%
Hence $h^{\ell}$ is a Cauchy sequence and the limit $h$ is a desired unique
solution satisfing
\begin{equation*}
\sup\limits_{0\leq t\leq \infty }\left\|e^{\frac{\la_0}{2}t^{\rho_1}}h(t)\right\|_{\infty }\leq C||h_{0}||_{\infty }.
\end{equation*}

In addition, if $\Omega $ is strictly convex, we {\it claim} that $h^{\ell+1}$ is
continuous in $[0,\infty )\times \{\bar{\Omega}\times \R%
^{3}\setminus \gamma _{0}\}$ inductively$.$ To prove this claim, for any
given fixed $\ell,$ we can use another iteration to solve the linear problem
for $h^{\ell+1}$ in (\ref{nn.sit}) as the limit of $\ell^{\prime }\rightarrow
\infty $:
\begin{equation*}
\{\partial _{t}+v\cdot \nabla _{x}+\widetilde{\nu} \}h^{\ell+1,\ell^{\prime
}+1}=K_{\overline{w}}h^{\ell+1,\ell^{\prime }}+w_{q,\ta,\vth}\Gamma \left(\frac{h^{\ell}}{w_{q,\ta,\vth}},\frac{h^{\ell}}{w_{q,\ta,\vth}}\right),
\end{equation*}%
with the initial boundary condition: $$h_{-}^{\ell +1,\ell'+1}(t,x,v)=h^{\ell +1,\ell'+1}(t,x,R_xv),\ h^{\ell +1,\ell'+1}(0)=h_0(x,v),$$
 and $
h^{\ell+1,0}\equiv h_0(x,v)$.
By induction over $\ell^{\prime },$ $h^{\ell+1,\ell^{\prime }}$
is continuous in $[0,\infty )\times \{\bar{\Omega}\times \R
^{3}\setminus \gamma _{0}\},$ and by Lemma \ref{es.k}, it is standard to
show that $K_{\overline{w}}h^{\ell+1,\ell^{\prime }}$ is continuous in the interior of $%
[0,\infty )\times \Omega \times \R^{3}.$ From the induction
hypothesis on continuity of $h^{\ell}$ in $[0,\infty )\times \{\bar{\Omega}%
\times \R^{3}\setminus \gamma _{0}\},$ it is also straightforward
and routine to verify that $w_{q,\ta,\vth}\Gamma(\frac{h^{\ell}}{w_{q,\ta,\vth}},\frac{h^{\ell}}{w_{q,\ta,\vth}})$ is
continuous in the interior of $[0,\infty )\times \Omega \times \R^{3}.$ 
In view of Lemma \ref{specularcon}%
, we thus deduce that $h^{\ell+1,\ell^{\prime }+1}$ is continuous in $[0,\infty )\times
\{\bar{\Omega}\times \R^{3}\setminus \gamma _{0}\}.$ Furthermore, it follows that
\begin{equation*}
\{\partial _{t}+v\cdot \nabla _{x}+\widetilde{\nu} \}\{h^{\ell+1,\ell^{\prime
}+1}-h^{\ell+1,\ell^{\prime }}\}=K_{\overline{w}}\{h^{\ell+1,\ell^{\prime }}-h^{\ell+1,\ell^{\prime }-1}\}
\end{equation*}%
with $\{h^{\ell+1,\ell^{\prime
}+1}-h^{\ell+1,\ell^{\prime }}\}(t,x,v)|_{-}=\{h^{\ell+1,\ell^{\prime
}+1}-h^{\ell+1,\ell^{\prime }}\}(t,x,R_xv)$ and $\{h^{\ell+1,\ell^{\prime
}+1}-h^{\ell+1,\ell^{\prime }}\}(0)=0$,
with this, one deduce that
\begin{eqnarray*}
\sup_{0\leq t\leq T}\left\|h^{\ell+1,\ell^{\prime }+1}(t)-h^{\ell+1,\ell^{\prime
}}(t)\right\|_{\infty } &\leq &C_{K}\int_{0}^{T}\left\|h^{\ell+1,\ell^{\prime
}}(s)-h^{\ell+1,\ell^{\prime }-1}(s)\right\|_{\infty }ds\leq \cdots \\
&\leq &C\frac{\{C_{K}T\}^{\ell^{\prime }}}{\ell^{\prime }!}.
\end{eqnarray*}%
Therefore, $\{h^{\ell+1,\ell^{\prime }}\}_{\ell'=1}^\infty$ is a Cauchy sequence in $L^{\infty },$ and its limit $%
h^{\ell+1}$ is continuous in $[0,\infty )\times \{\bar{\Omega}\times \R^{3}\setminus \gamma _{0}\}.$
We conclude our {\it claim}. Once $h^{\ell}$ is continuous, its limits $h$ is continuous as well.

Finally, the uniqueness and positivity of $F$ follows the same argument as the proof of the
the Theorem 3 in \cite[pp.804]{Guo-2010}, we omit the details for brevity. This finishes the proof of Theorem \ref{specularnl}.

\end{proof}

\medskip

\noindent {\bf Acknowledgements:}
Shuangqian Liu has received grants from the National Natural Science Foundation of China (contracts: 11471142 and 11271160) and China
Scholarship Council. Xiongfeng Yang has received grants from the National Natural Science Foundation of China (11171212) and
the SJTU's SMC Projection A.
The authors would like to thank the generous hospitality of the Division of Applied Mathematics at Brown University during their visit.
The authors are also very grateful to Professor Yan Guo for many fruitful discussions on the subject of the paper.

\medskip


\end{document}